%% file: criterion.tex
\documentclass[a4paper,twoside,10pt]{article}
\include{include}
\begin{document}

\title{A criterion for detecting trivial elements of Burnside groups \\ {\large  Un critère pour détecter les éléments triviaux dans les groupes de Burnside}}
\author{R\'emi Coulon}

\maketitle

\begin{abstract}
	In this article we give a sufficient and necessary condition to determine whether or not an element of the free group induces a non-trivial element of the free Burnside group of sufficiently large odd exponent.
	This criterion can be stated without any knowledge about Burnside groups, in particular about the proof of its infiniteness.
	Therefore it provides a useful tool that we will use later to study outer automorphisms of Burnside groups.
	We also state an analogue result for periodic quotients of torsion-free hyperbolic groups.
\end{abstract}

\selectlanguage{french}
\begin{abstract}
	Dans cet article, on propose une condition nécessaire et suffisante pour déterminer si un élément du groupe libre induit ou non un élément trivial dans les groupes de Burnside libre d'exposants impairs suffisamment grands.
	Ce critère peut être énoncé sans aucun pré-requis sur les groupes de Burnside.
	En particulier il n'est pas nécessaire de comprendre pourquoi les groupes de Burnside sont infinis pour l'appliquer.
	Pour cette raison il fournit un outil effectif qui nous permettra plus tard d'étudier les automorphismes du groupe de Burnside.
	Nous donnons aussi un résultat analogue pour les quotients périodiques d'un groupe hyperbolique sans torsion. 
\end{abstract} 

\selectlanguage{english}
\tableofcontents 

\newcommand{\varz}{x}


\begin{sympysilent}
from criterion import * 
\end{sympysilent}

\newcommand{\AAA}{A}
\newcommand{\AB}{B}
\newcommand{\AC}{C}
\newcommand{\AF}{F}
\newcommand{\Aa}{a}
\newcommand{\Ab}{b}
\newcommand{\Ac}{c}
\newcommand{\AcBis}{c'}
\newcommand{\Ad}{d}
\newcommand{\ALn}{\lambda^{-1}}

\newcommand{\Balpha}{\alpha}
\newcommand{\Bbeta}{\beta}

\newcommand{\LL}{L}
\newcommand{\Ll}{l}
\newcommand{\Kk}{k}
\newcommand{\KkBis}{k'}

\newcommand{\Nn}{n}

\newcommand{\PTLbase}{\pi\sinh r_0}

\newcommand{\Qalpha}{\alpha}

\newcommand{\Rr}{r_0}

\newcommand{\Sepsilon}{\epsilon}
\newcommand{\SepsilonOne}{\epsilon_1}
\newcommand{\SepsilonDot}{\dot\epsilon}
\newcommand{\SDeltaZero}{\Delta_0}

\newcommand{\Tdelta}{\delta}
\newcommand{\TdeltaZero}{\delta_0}
\newcommand{\TdeltaOne}{\delta_1}
\newcommand{\TdeltaBar}{\bar\delta}
\newcommand{\TdeltaDot}{\dot\delta}

\newcommand{\Xeta}{\eta}
\newcommand{\XetaBis}{\eta'}

\input{0_introduction}

\input{1_hyperbolic_spaces}
\input{2_cone_off}
\input{3_small_cancellation}
\input{4_burnside_groups}

\makebiblio

\noindent
R\'emi Coulon \\
remi.coulon@vanderbilt.edu \\
Vanderbilt University \\ 
Department of Mathematics \\
1326 Stevenson Center, Nashville, TN 37240, USA.


\end{document}

%% file: include.tex
\usepackage{ucs}
\usepackage[utf8x]{inputenc}
\usepackage[T1]{fontenc}
\usepackage[french,english]{babel}
\usepackage{subfigure}
\usepackage{amsmath, amsthm, amsfonts, amssymb}
\usepackage{latexsym}
\usepackage{color}
\usepackage{stmaryrd}
\usepackage{endnotes}
\usepackage{wrapfig}
\usepackage{fancyhdr,lastpage}
\usepackage{ccaption}
\usepackage{enumitem}
\usepackage{xspace}
\usepackage{ifthen}
\usepackage{xkeyval}
\usepackage{remreset}
\usepackage{xcolor}
\usepackage[marginpar]{todo}
\usepackage{calc}
\usepackage{multirow}
\usepackage{graphicx}
\usepackage{sympytex}

\makeatletter

\fancypagestyle{plain}{%
	\fancyhf{} %
	\fancyfoot{}%
}

\let\orig@enumerate =\enumerate
\renewenvironment{enumerate}{\orig@enumerate [label=(\roman{*}), ref=(\roman{*})]}
{\endlist}

\newenvironment{enumeratealph}
{\orig@enumerate [label=(\alph{*}), ref=(\alph{*})]}
{\endlist}

\newenvironment{foldable}
{\orig@enumerate [label=(C\arabic{*}), ref=(C\arabic{*})]}
{\endlist}

\definecolor{grispuce}{gray}{0.7}
\renewcommand{\labelitemi}{\raisebox{0.15em}{$\color{grispuce}{\scriptstyle \blacktriangleright}$}}

\newcommand{\resp}{respectively \xspace}
\newcommand{\oas}{$\omega$-almost surely ($\omega$-as)\xspace \renewcommand{\oas}{$\omega$-as\xspace}}
\newcommand{\oeb}{$\omega$-essentially bounded  ($\omega$-eb)\xspace \renewcommand{\oeb}{$\omega$-eb\xspace}}

\newtheorem{defi}{Definition}[section]
\newtheorem{lemm}[defi]{Lemma}
\newtheorem{prop}[defi]{Proposition}
\newtheorem{theo}[defi]{Theorem}
\newtheorem{coro}[defi]{Corollary}

\newtheorem*{mthm}{Theorem}

\@addtoreset{defi}{section}
\renewcommand{\thesection}{%
	\arabic{section}%
}
\renewcommand{\thesubsection}{%
	\ifnum \value{section}>0
		\thesection.%
	\else%
	\fi%
	\arabic{subsection}%
}
\renewcommand{\thedefi}{%
	\ifnum \value{section}>0
		\thesection.%
	\else%
	\fi%
	\arabic{defi}%
}

\newcommand{\rem}{\paragraph{Remark :}}

\newcommand{\nota}{\paragraph{Notation :}}
\newcommand{\notas}{\paragraph{Notations :}}
\newcommand{\voc}{\paragraph{Vocabulary :}}

\newcommand{\N}{\mathbf{N}}
\newcommand{\Z}{\mathbf{Z}}

\newcommand{\R}{\mathbf{R}}
\newcommand{\C}{\mathbf{C}}
\renewcommand{\H}{\mathbf{H}}
\newcommand{\intval}[2]{\left[#1\, , #2\right]}
\newcommand{\intvald}[2]{\left\{ #1,\dots , #2\right\} }

\define@choicekey*{planhyp}{corps}[\val \nr]{r,c,h,}[]{
	\ifcase\nr 
		\def\ph@corps{\R}
	\or
		\def\ph@corps{\C}
	\or
		\def\ph@corps{\H}
	\or
		\def\ph@corps{}
	\fi
}
\presetkeys{planhyp}{corps}{}
\newcommand{\HP}[2][]{
	\begingroup
		\setkeys{planhyp}{#1}
		\ifthenelse{\equal{\ph@corps}{}}
			{\mathbf{H}_{#2}}
			{\mathbf{H}_{#2}\left(\ph@corps \right)}
	\endgroup
}
\newcommand{\gro}[4][]{
	\ifthenelse{\equal{#1}{}}
	{\left< #2,#3\right>_{#4}}
	{\left< #2,#3\right>_{#4}^{#1}}
}

\newcommand*{\dist}[3][]{
	\ifthenelse{\equal{#1}{}}
		{\left| #2- #3 \right|}
		{
			\ifthenelse{\equal{#1}{SC}}
			{\left\| #2- #3 \right\|}
			{\left| #2- #3 \right|_{#1}}
		}
}
\newcommand{\diaminter}[2]{
	\left|#1 \cap #2\right|
}

\newcommand*{\distV}[1][]{
	\ifthenelse{\equal{#1}{}}
		{\left| \ . \ \right|}
		{
			\ifthenelse{\equal{#1}{SC}}
			{\left\| \ . \ \right\|}
			{\left| \ . \ \right|_{#1}}
		}
}

\newcommand*{\geo}[3][]{
	\ifthenelse{\equal{#1}{}}
		{\left[ #2, #3 \right]}
		{\left[ #2, #3 \right]_{#1}}
}

\newcommand*{\triang}[4][]{
	\ifthenelse{\equal{#1}{}}
		{\left[ #2, #3, #4 \right]}
		{\left[ #2, #3, #4 \right]_{#1}}
}

\newcommand{\nerf}[3]{
	\left( #1,#2\right)_{#3}
}

\define@boolkey{longueurtrans}{stable}[true]{}
\define@key{longueurtrans}{espace}[]{\def \lt@espace{#1}}
\presetkeys{longueurtrans}{stable=false,espace}{}
\newcommand*{\len}[2][]{
	\begingroup
		\setkeys{longueurtrans}{#1}
		\ifKV@longueurtrans@stable {
			\ifthenelse{\equal{\lt@espace}{}}
				{\left[ {#2}\right]^{\infty}}
				{\left[ {#2}\right]^{\infty}_{\lt@espace}}
		}
		\else {
			\ifthenelse{\equal{\lt@espace}{}}
				{\left[ {#2}\right]}
				{\left[ {#2}\right]_{\lt@espace}}
		}
		\fi
	\endgroup
}

\newcommand{\rinj}[2]{r_{\textit{inj}}\left(#1,#2\right)}

\newcommand{\rips}[3][]{
	\ifthenelse{\equal{#1}{}}
		{P_{#2}\left( #3\right)}
		{P_{#2}^{(#1)}\left( #3\right)}
}
\newcommand{\diam}{\operatorname{diam}}

\newcommand{\sdp}[3][]{
	\ifthenelse{\equal{#1}{}}
		{#2 \rtimes #3}
		{#2 \rtimes _{#1} #3}
}

\newcommand{\out}[1]{\operatorname{Out} \left( #1\right)}
\newcommand{\free}[1]{\mathbf F_{#1}}

\newcommand{\burn}[2]{\mathbf B _{#1}(#2)}
\newcommand{\dlim}{\displaystyle{\lim_{\longrightarrow}}\ }

\renewcommand{\sinh}{\operatorname{sh}}
\renewcommand{\cosh}{\operatorname{ch}}

\newcommand{\fantomA}{\vphantom{\frac 12}\!}
\newcommand{\fantomB}{\vphantom{\Big\vert}\!}
\renewcommand{\epsilon}{\varepsilon}
\renewcommand{\phi}{\varphi}
\renewcommand{\leq}{\leqslant}
\renewcommand{\geq}{\geqslant}

\newcommand{\makebiblio} {
	\IfFileExists{/Users/coulonr/Maths/Publications/papers2/bibliography.bib}{
		\bibliography{/Users/coulonr/Maths/Publications/papers2/bibliography}
	}{	
		\bibliography{/Volumes/Data/Maths/svn/papers2/bibliography}	 
	}
	\bibliographystyle{abbrv} 
}

\makeatother

%% file: 0_introduction.tex
\section*{Introduction}

\paragraph{}Let $n$ be an integer.
A group $G$ has exponent $n$ if for all $g \in G$, $g^n=1$.
In 1902, W.~Burnside asked whether a finitely  generated group with finite exponent  is necessarily finite or not \cite{Bur02}.
To study this question, it is natural to look at the free Burnside group $\burn rn = \free r / \free r^n$ which is the quotient of the free goup of rank $r$, denoted by $\free r$, by the subgroup  $\free r^n$ generated by all $n$-th powers.
It is indeed the largest group of rank $r$ and exponent $n$.
Until the work of P.S.~Novikov and S.I.~Adian,  it was only known that for some small exponents $\burn rn$ was finite ($n=2$ \cite{Bur02}, 3 \cite{Bur02,LevWae33}, 4 \cite{San40}, 6 \cite{Hal57}).
In 1968, they proved that for $r \geq 2$ and $n \geq 4381$ odd $\burn rn$ is infinite \cite{NovAdj68a,NovAdj68b,NovAdj68c}.
This result has been improved in many directions. 
A.Y.~Ol'shanski\u\i\ \cite{Olc82} proposed an other proof of the Novikov-Adian theorem using graded diagramms.
Moreover he extended the  result to the periodic quotients of a hyperbolic group \cite{Olc91}.
S.V.~Ivanov \cite{Iva94} and I.G.~Lysenok \cite{Lys96} solved the case of even exponents.

\paragraph{} The crucial fact used by P.S.~Novikov and S.I.~Adian is the following result (see \cite[Statement 1]{AdiLys92}).
Let $p$ be an integer and $w$ a reduced word representing an element of $\free r$.
If $w$ does not contain a subword of the form $u^p$, then $w$ induces a non-trivial element of $\burn rn$ where $n$ is an odd integer larger than $10000p$.
The infiniteness of the Burnside groups follows then from the existence of infinite words without third-power (like Thue-Morse words \cite{Adi79}).
Our goal is to improve this statement.
Given a reduced word $w$ of $\free r$ we provide a sufficient and necessary condition to decide wether $w$ represents a trivial element of $\burn rn$ or not.

\paragraph{}Before describing the criterion we would like to motivate this work.
We wish to investigate the outer automorphisms of Burnside groups.
Since $\free r^n$ is a characteristic subgroup of $\free r$, the projection $\free r \twoheadrightarrow \burn rn$ induces a map $\out{\free r} \rightarrow \out{\burn rn}$.
This map is not onto.
Nevertheless it provides numerous examples of automorphisms of the Burnside groups.
For instance if $n$ is an odd exponent large enough, the image of $\out{\free r}$ in $\out{\burn rn}$ contains free groups of arbitrary rank \cite{Cou10a}.
One important question is: which automorphisms of $\free r$ induce automorphisms of infinite order of $\burn rn$?
In \cite{Cou10a} we provided a large class of automorphisms of $\free r$ having this property.
However we are looking for a sufficient and necessary condition to characterize them.
To understand the difficulties that may appear, let us have a look at a simple example already studied by E.A.~Cherepanov \cite{Che05}.
Let $\phi$ be the automorphism of $\free 2 = \mathbf F (a,b)$ defined by $\phi(a)=ab$ and $\phi(b)=a$.
The idea is to compute the orbit of $b$ under $\phi$.
\begin{displaymath}
	\begin{array}{lclclcl}
	\phi^1(b) & = & a			&\quad 	& \phi^5(b) & = & abaababa\\
	\phi^2(b) & = & ab			&		& \phi^6(b) & = & abaababaabaab\\
	\phi^3(b) & = & aba			&		& \phi^7(b) & = & abaababaabaababaababa\\
	\phi^4(b) & = & abaab		&		& \dots & &
	\end{array}
\end{displaymath}
This sequence converges to a right-infinite word
\begin{displaymath}
	\phi^\infty(b)= abaababaabaababaababaabaababaabaab\dots
\end{displaymath}
which does not contain a subword which is a fourth-power \cite{Mos92}.
Using the criterion of P.S.~Novikov and S.I.~Adian, the $\phi^k(b)$'s define pairwise distinct elements of $\burn rn$ for some large $n$.
In particular $\phi$ induces an automorphism of infinite order of the Burnside groups of large exponents.
For an arbitrary automorphism the situation becomes more complicated.
Consider for instance the automorphism $\psi$ of $\free 4 = \mathbf F (a,b,c,d)$ defined by $\psi(a)=a$, $\psi(b)=ba$, $\psi(c)=c^{-1}bcd$ and $\psi(d)=c$.
As previously we compute the orbit of $d$ under $\psi$.
\begin{displaymath}
	\begin{array}{lcl}
	\psi^1(d) & = & c			\\
	\psi^2(d) & = & c^{-1}\mathbf bcd			\\
	\psi^3(d) & = & d^{-1}c^{-1} b^{-1}c\mathbf b\mathbf ac^{-1} bcdc			\\
	\psi^4(d) & = & c^{-1}d^{-1}c^{-1} b^{-1}c a^{-1} b^{-1}c^{-1} bcd\mathbf b\mathbf a^2d^{-1}c^{-1} b^{-1}c b ac^{-1} bcd bcd		\\
	\psi^5(d) & = & d^{-1}c^{-1} b^{-1}d^{-1}c^{-1} b^{-1}c a^{-1} b^{-1}c^{-1} bcd a^{-2} b^{-1}d^{-1}c^{-1} b^{-1}c\dots\\
	&&  b ac^{-1}  bcdc\mathbf b\mathbf a^3c^{-1}d^{-1}c^{-1} b^{-1}c a^{-1} b^{-1}c^{-1} bcd b a^2d^{-1}c^{-1} b^{-1}c\dots \\
	&&  b ac^{-1} bcdc b ac^{-1} bcdc
	\end{array}
\end{displaymath}
Note that each time $\psi^k(d)$ contains a subword $ba^m$ then $\psi^{k+1}(d)$ contains $ba^{m+1}$.
Hence the $\psi^k(d)$'s contain arbitrary large powers of $a$.
This cannot be avoided by choosing the orbit of an another element.
The result of P.S.~Novikov and S.I.~Adian cannot tell us if the $\psi^k(d)$'s are pairwise distinct in $\burn rn$.
Therefore, we need a more accurate criterion two distinguish two different elements of $\burn rn$.
This question about automorphisms of $\burn rn$ is solved in \cite{CouHil11}.

\paragraph{} To state our theorem we need to define elementary moves.
Let $\xi$ and $n$ be two integers.
A \emph{$(\xi, n)$-elementary move} consists in replacing a reduced word of the form $pu^ms \in \free r$ by the reduced representative of $pu^{m-n}s$, provided $m$ is an integer larger than $n/2- \xi$.
Note that an elementary move may increase the length of the word.
\begin{mthm}
	There exist numbers $\xi$ and $n_0$ such that for all odd integers $n \geq n_0$ we have the following property.
	Let $w$ be a reduced word of $\free r$.
	The element of $\burn rn$ defined by $w$ is trivial if and only if there exists a finite sequence of $(\xi, n)$-elementary moves that sends $w$ to the empty word. 
\end{mthm}

\paragraph{}
A.Y.~Ol'shanksi\u\i\ point us out that this theorem also follows from Lemma~5.5 of \cite{Olc82} when $m \geq n/3$.
Moreover his method could be adapted to cover the case where $m \geq n/2-\xi$.
However in this paper we follow the construction given by T.~Delzant and M.~Gromov.
In \cite{DelGro08}, they proposed an alternative proof of the Novikov-Adian Theorem.
Using a geometrical approach they built a sequence of hyperbolic groups $\free r \twoheadrightarrow G_1 \twoheadrightarrow G_2 \twoheadrightarrow \dots$ whose direct limit is $\burn rn$.
At each step the groups have - among others - the following properties.
\begin{itemize}
	\item $G_{k+1}$ is a small cancellation quotient of $G_k$
	\item The relations that define the the quotient $G_k \twoheadrightarrow G_{k+1}$ are $n$-th powers of elements of $G_k$.
\end{itemize}
Given a small cancellation group, one knows  an algorithm solving the word problem.
Consider for instance $w$ a reduced word of $\free r$ which is trivial in the first quotient $G_1$.
According to the Greendlinger Lemma, $w$ contains a subword which equals three fourth of a relation.
In our situation, this means that $w$ can be written $w=pu^ms$ where $m \geq 3n/4$.
Applying an elementary move, we obtain a new word $w' $ which represents $pu^{m-n}s$ and is shorter than the previous one.
Moreover $w'$ is still trivial in $G_1$.
By iterating the process we get a sequence of elementary moves that sends $w$ to the empty word.

\paragraph{}
For the Burnside groups the process is more tricky.
Let $w$ be a reduced word of $\free r$ which is trivial in $\burn rn$.
Since $\burn rn$ is the direct limit of the $G_k$'s, there exists a step $k$ such that $w$ is trivial in $G_{k+1}$ but not in $G_k$.
Roughly speaking, the Greendlinger Lemma tells us that a geodesic word of $G_k$ representing $w$ contains three fourth of a relation, i.e. a subword of the form $u^m$ with $m \geq 3n/4$.
One would like to apply an elementary move.
However there is no reason that $u^m$ should be a subword of $w$ in $\free r$.
Consider the following example.
Let $u$ and $v$ be two reduced words of $\free r$.
Assume that $u^n$ is trivial in $G_1$.
Let $w=\left(u^lv\right)^q\left(u^{l-n}v\right)^{n-q}$.
As an element of $G_1$, $w$ represents $\left(u^lv\right)^n$ which contains an $n$-th power.
Nevertheless this does not hold in $\free r$.
The fact is that the previous relations (here $u^n$) mess up the powers.
However despite $w$ does not contain a $n$-th power of $u^lv$, it contains a large power of $u$.
Thus $n-q$ elementary moves send $w$ to $\left(u^lv\right)^n$.
We can now ``read'' the power of $u^lv$ directly in $\free r$ and apply an elementary move to reduced the length of this last word.
This example actually describes the general situation.
Our main theorem is proved by induction on $k$ using this kind of arguments.
The technical difficulties come from the fact that to be rigorous we should formulate the ideas presented above in a hyperbolic framework, taking care of many parameters (hyperbolicity constants, small cancellation parameters,...). 

\paragraph{}Our study works in fact in a more general situation.
Let   $(X,x_0)$ be a $\delta$-hyperbolic, geodesic, pointed space and $G$ a non-elementary, torsion-free group acting properly, co-compactly, by isometries on it.
We provide indeed a sufficient and necessary condition to detect elements of $G$ which are trivial in the quotient $G/G^n$.
For this purpose we need to extend the definition of elementary moves to this context.
Let $v$ be a non-trivial isometry of $G$.
Since $G$ is torsion free, it fixes two points $v^-$ and $v^+$ of $\partial X$, the boundary at infinity of $X$.
We denote by $Y_v$ the set of points of $X$ which are $10\delta$-close to some bi-infinite geodesic joining $v^-$ and $v^+$.
This subset is quasi-isometric to a line.
Moreover $v$ roughly acts on it by translation of length $\len v$.
A \emph{$(\xi, n)$-elementary move} consists in replacing a point $y \in X$ by $v^{-n}y$ provided that we have in $X$
\begin{displaymath}
	\diaminter{\geo {x_0}y}{Y_v}\geq \len{v^m}, \text{ where } m \geq n/2-\xi.
\end{displaymath}
Here $\diaminter{\geo {x_0}y}{Y_v}$ is a quantity that measures the length of the part of the geodesic $\geo {x_0}y$ which is approximatively contained in $Y_v$.

\paragraph{}
Let us compare this definition with the previous one.
Let $X$ be the Cayley graph of $\free r$ and $x_0$ the vertex representing 1.
Let $g\in\free r$.
Assume that $g$ can be written as a reduced word $g=pu^ms$. 
Then the geodesic $\geo {x_0}{gx_0}$, labeled by $pu^ms$, intersects the axis of $v=pup^{-1}$ along a path of length $\len {v^m}$.
Moreover $v^{-n}g$ can be represented by the word $pu^{m-n}s$.
The next theorem is a generalization for hyperbolic groups of the previous one.
Not only does it tell that an element of $G$ trivial in a periodic quotient $G/G^n$ of $G$ can be reduced to the trivial element using elementary moves but it also explain how to decide whether or not two element of $G$ are the same in $G/G^n$ using the same kind of elementary moves.

\begin{mthm}
	Let $G$ be a non-elementary, torsion-free group acting freely, properly, co-compactly, by isometries on a proper, hyperbolic, geodesic, pointed space $(X,x_0)$.
	There exist numbers $\xi$ and $n_0$ such that for all odd integers $n \geq n_0$ we have the following property.
	Two elements $g$ and $g'$ of $G$ induce the same element of $G/G^n$ if and only if there are two finite sequences of $(\xi, n)$-elementary moves that respectively send $gx_0$ and $g'x_0$ to the same point.
\end{mthm}

\paragraph{Outline of the article.} In Section~\ref{sec:hyperbolic spaces}, we review some of the standard facts on hyperbolic geometry.
Since the proofs in the rest of the article are already quite technical, we also tried to compile in this section all the results that only require hyperbolic geometry.
Section~\ref{sec:cone-off over a metric space} investigates the cone-off construction used by T.~Delzant and M.~Gromov, in \cite{DelGro08}.
In particular we compare at a large scale the relation between the geometry of the cone-off over a metric space and the one of its base.
Section~\ref{sec:small cancellation} is devoted to the study of small cancellation theory.
Our goal is to understand how to lift figures from a small cancellation quotient $\bar G  = G/K$ in the group $G$.
For instance, let $g$ be an element of $G$ such that a geodesic of $\bar G$ representing the image of $g$ contains a large power.
Under which conditions $g$ already contains a large power? 
If not, what kind of transformations could send $g$ to an element containing a large power?
In the last section we summarize all this results in an induction that will proves our main theorem.

\paragraph{Acknowledgment.}
Part of this work was done during my stay at the \emph{Max-Planck-Institut f\"ur Mathematik}, Bonn, Germany.
I would like to express my gratitude to all faculty and staff from the MPIM for their support and warm hospitality.
I am also thankful to A.Y.~Ol'shanski\u\i\ who point me out Lemma~5.5 of \cite{Olc82} which provides an alternative proof of our result.

%% file: 1_hyperbolic_spaces.tex

\section{Hyperbolic spaces}
\label{sec:hyperbolic spaces}

\paragraph{} Let $X$ be a metric space.
Given two points $x, x' \in X$, we denote by $\dist[X]x{x'}$ (or simply $\dist x{x'}$) the distance between them.
Although it may not be unique, we write $\geo x{x'}$ for a geodesic joining $x$ and $x'$.
The Gromov's product of three points $x$, $y$ and $z$ of $X$ is defined by
\begin{displaymath}
	\gro xyz = \frac 12 \left(\fantomB \dist xz + \dist yz - \dist yz \right).
\end{displaymath}
From now on, we assume that $X$ is $\delta$-hyperbolic, which means that for all $x,y,z,t \in X$
\begin{equation}
\label{eqn: hyperbolicity 1}
	\gro xzt \geq \min \left\{\fantomB \gro xyt, \gro yzt \right\} - \sympy{F_hyp1(delta)}.
\end{equation}\
Equivalently, for all $x, y, z, t \in X$,
\begin{equation}
\label{eqn: hyperbolicity 2}
	\dist xy + \dist zt \leq \left\{ \fantomB \dist xz + \dist yt , \dist xt + \dist yz \right\} + \sympy{F_hyp2(delta)}.
\end{equation}
It follows from the hyperbolicity assumption that the geodesic triangles of $X$ are $\sympy{F_thin(delta)}$-thin (see \cite[Chap. 1, Prop. 3.1]{CooDelPap90}).
More precisely for all $x,y,z \in X$, for all $(r,s) \in \geo xy \times \geo xz$, if $\dist xr = \dist xs \leq \gro yzx$ then $\dist rs \leq \sympy{F_thin(delta)}$.
The Gromov's product $\gro xyz$ can be interpreted as an estimate of the distance of $z$ to $\geo xy$.
We have indeed $\gro xyz \leq d\left(z,\geo xy\right) \leq \gro xyz + \sympy{F_groVsDist(delta)}$ (see \cite[Chap. 3, Lemm. 2.7]{CooDelPap90}).
We denote by $\partial X$, the boundary at infinity of $X$ (see \cite[Chap.2]{CooDelPap90} for the definition and the main properties).


\subsection{Quasi-convex subsets}
\label{sec: hyp - quasi-convex}

Let $Y$ be a subset of $X$. 
We denote by $Y^{+ \alpha}$ the $\alpha$-neighbourhood of $Y$, i.e. the set of points $x \in X$ such that $d(x,Y) \leq \alpha$.
A point $y$ of $Y$ is called an $\eta$-projection of $x$ on $Y$ if $\dist xy \leq d(x,Y)+ \eta$.
A 0-projection is simply called a projection.

\begin{defi}
\label{def quasi-convex}
	Let $\alpha \geq0$.
	A subset $Y$ of $X$ is \emph{$\alpha$-quasi-convex} if for every $x \in X$ and $y,y' \in Y$, $d(x,Y) \leq \gro y{y'}x + \alpha$.
\end{defi}

\begin{defi}
\label{def: strong quasi-convex}
	A subset $Y$ of $X$ is \emph{strongly quasi-convex} if for all $y,y' \in Y$ there exist $z, z' \in Y$ and geodesics $\geo yz$, $\geo z{z'}$, $\geo {z'}{y'}$ contained in $Y$ such that $\dist yz, \dist{y'}{z'} \leq \sympy{F_SQCDef(delta)}$.
\end{defi}

\rem Our definition of quasi-convex is slightly different from the one usually given in the literature (every geodesic joining two points of $Y$ lies in the $\alpha$-neighbourhood of $Y$).
However an $\alpha$-quasi-convex in the regular sense is $(\alpha +4\delta)$-quasi-convex in our sense, and conversely.
This definition has the advantage of working even in a length space which is not geodesic (see \cite{Coulon:2013tx}).
Moreover since we defined hyperbolicity using Gromov's products it is more convenient to work with.
With this definition a geodesic is $\sympy{F_groVsDist(delta)}$-quasi-convex.
By hyperbolicity, a strongly quasi-convex subset is $ \sympy{F_SQCQc(delta)}$-quasi-convex.

\begin{lemm}[compare {\cite[Chap. 10, Prop. 2.1]{CooDelPap90}}]
\label{projection quasi-convex}
	Let $Y$ be an $\alpha$-quasi-convex subset of $X$.
	\begin{itemize}
		\item Let $x \in X$ and $y \in Y$. If $p$ is an $\eta$-projection of $x$ on $Y$, then $\gro xyp \leq \sympy{F_projQc1(alpha,eta)}$.
		\item Let $x, x' \in X$. If $p$ and $p'$ are respectively $\eta$- and $\eta'$-projections of $x$ and $x'$ on $Y$ then, 
		\begin{displaymath}
			\dist p{p'} \leq \max\left\{\fantomB \epsilon, \dist x{x'} - \dist xp - \dist {x'}{p'} + 2\epsilon \right\},
		\end{displaymath}
		where $\epsilon = \sympy{F_projQc2(alpha,eta,etaBis,delta)}$.
	\end{itemize}
\end{lemm}

\begin{lemm}
\label{res: exending projections on quasi-convex}
	Let $Y$ be an $\alpha$-quasi-convex subset of $X$.
	Let $x$ be a point of $X$ and $p$ an $\eta$-projection of $x$ on $Y$.
	For every $x' \in X$, $p$ is an $\epsilon$-projection of $x'$ on $Y$ where $\epsilon =  \gro xp{x'} + \sympy{F_extProj(alpha,eta,delta)}$.
\end{lemm}

\begin{proof}
	Let $\eta'>0$ and $p'$ be an $\eta'$-projection of $x'$ on $Y$.
	The previous lemma combined with the triangle inequality gives $\dist p{p'}\leq \epsilon(\eta')$ where $\epsilon(\eta') = \gro  xp{x'} + \sympy{Faux_extProj(alpha,eta,etaBis,delta)}$.
	Therefore $p$ is an $(\epsilon(\eta')+\eta')$-projection of $x'$ on $Y$.
	This property holds for every $\eta'>0$ which gives the result.
\end{proof}

\begin{defi}
\label{def:intersection quasi-convex}
Let $Y$ and $Z$ be two subsets of $X$ we denote by $\diaminter YZ$ the following quantity.
\begin{displaymath}
	\diaminter YZ = \frac 12 \sup_{\substack{y,y' \in Y \\ z,z' \in Z}} \left\{\fantomB 0, \dist y{y'} + \dist z{z'} - \dist yz - \dist {y'}{z'}\right\}.
\end{displaymath}
\end{defi}

\rem 
It follows from the definition that $\diaminter YZ \geq \diam \left(Y\cap Z\right)$.
Actually, if $Y$ and $Z$ are respectively $\alpha$- and $\beta$-quasi-convex subsets of $X$, $\diaminter YZ$ roughly measures the intersection of $Y$ and $Z$:
\begin{displaymath}
	\diaminter YZ \approx \diam\left(Y^{+\alpha+10\delta}\cap Z^{+\beta+10\delta}\right) +10\delta.
\end{displaymath}
However this notation has two advantages.
First the definition does not involve the hyperbolicity constant $\delta$ nor the quasi-convexity parameters $\alpha$ and $\beta$.
Moreover, given two points $x$ and $x'$ of $X$ joined by a geodesic the triangle inequality yields $\diaminter{\geo x{x'}}Y = \diaminter{\{x,x'\}}Y$.
Therefore $\diaminter{\geo x{x'}}Y$ does not depend on the choice of the geodesic but only on its endpoints.
This is convenient since our space is not necessary uniquely geodesic.

\paragraph{} Let $Y$ and $Z$ be two subsets of $X$.
Applying the triangle inequality we obtain the followings.
\begin{enumerate}
	\item \label{enu: diaminter - neighbourhood}
	For all $A,B\geq0$, $\diaminter{Y^{+A}}{Z^{+B}} \leq \diaminter YZ +\sympy{F_diamInterNbhood(A,B)}$.
	\item \label{enu: diaminter - gro}
	For all $x,x',z \in X$, $\diaminter {\geo xz}Y \leq \diaminter{\geo x{x'}}Y + \gro x{x'}z$.
\end{enumerate}
Combining \ref{enu: diaminter - gro} with the hyperbolicity condition (\ref{eqn: hyperbolicity 1}) we obtain for all $x, x', z, z' \in X$,
\begin{equation}	
\label{eqn: diameter with gromov products}
	\diaminter{\geo z{z'}}Y \leq \diaminter{\geo x{x'}}Y + \gro x{x'}z + \gro x{x'}{z'} + \sympy{F_diamInterProdGromov(delta)}.
\end{equation}

\begin{prop}
\label{res:diam quasi-convex and projections}
	Let $Y$ be an $\alpha$-quasi-convex subset of $X$.
	Let $x$ and $x'$ be two points of $X$.
	We assume that $y$ and $y'$ are respectively $\eta$- and $\eta'$-projections of $x$ and $x'$ on $Y$.
	Then $\dist{\diaminter {\geo x{x'}}Y}{\dist y{y'}} \leq  \epsilon$, where $\epsilon = \sympy{F_diamInterProj(alpha,eta,etaBis,delta)}$.
\end{prop}

\begin{proof}
	By projection on a quasi-convex we have,
	\begin{displaymath}
		\max\left\{\fantomB \dist x{x'} - \dist xy - \dist {x'}{y'} + 2\epsilon, \epsilon\right\} \geq \dist y{y'},
	\end{displaymath}
	where $\epsilon = \sympy{Faux1_diamInterProj(alpha,eta,etaBis,delta)}$.
	Therefore
	\begin{displaymath}
		\diaminter {\geo x{x'}}Y \geq \frac 12 \max\left\{\fantomB \dist x{x'} +\dist y{y'} - \dist xy - \dist {x'}{y'},0 \right\} \geq \dist y{y'} - \epsilon.
	\end{displaymath}
	On the other hand, $y$ and $y'$ being respective $\eta$- and $\eta'$-projections of $x$ and $x'$, the triangle inequality implies that for every $z,z' \in Y$
	\begin{eqnarray*}
		\frac12 \left(\fantomB\dist x{x'} + \dist z{z'} - \dist xz - \dist {x'}{z'} \right)
		& \leq & \dist y{y'} + \gro xzy + \gro {x'}{z'}{y'}\\
		& \leq & \dist y{y'} +\sympy{Faux2_diamInterProj(alpha,eta,etaBis,delta)}.
	\end{eqnarray*}
	This inequality holds for every $z,z' \in Y$ hence $\diaminter {\geo x{x'}}Y \leq \dist y{y'} + \sympy{F_diamInterProj(alpha,eta,etaBis,delta)}$, which ends the proof.
\end{proof}


\subsection{Quasi-geodesics}

In this article, all the paths that we consider are continuous.

\begin{defi}
	Let $k \geq 1$, $l \geq 0$ and $L>0$.
	Let $J$ be an interval of $\R$.
	A path $\sigma :J \rightarrow X$ is
	\begin{itemize}
		\item a \emph{$(k,l)$-quasi-geodesic} if for all $s,t \in J$,
		\begin{displaymath}
			k^{-1} \dist st - l \leq \dist{\sigma(s)}{\sigma(t)} \leq k \dist st +l.
		\end{displaymath}
		\item a \emph{$L$-local $(k,l)$-quasi-geodesic} if its restriction to every close interval of diameter $L$ is a $(k,l)$-quasi-geodesic.
		\item a \emph{$L$-local geodesic} if it is a $L$-local $(1,0)$-quasi-geodesic.
	\end{itemize}
\end{defi}

\rem By abuse of notation, we often write $\sigma$ for the image $\sigma(J)$ of $\sigma$ in $X$.

\begin{prop}[Stability of quasi-geodesics]
\label{stability quasi-geodesics}
	Let $l \geq 0$ and $k \geq 1$.
	There exist $L >0$, $k' \geq k$ and $d \geq 0$ depending only on $l$ and $k$ (not on $X$ nor $\delta$) with the following property.
	The Hausdorff distance between two $L\delta$-local $(k,l\delta)$-quasi-geodesics joining the same endpoints (possibly in $\partial X$) is at most $d\delta$.
	Moreover every $L\delta$-local $(k,l\delta)$-quasi-geodesic is a (global)  $(k',l\delta)$-quasi-geodesic.
\end{prop}

\begin{proof}
	The case where $\delta  = 1$ follows from \cite[Chap. 4, Th. 1.4 and 3.1]{CooDelPap90}. The general case is obtained by a rescaling argument.
\end{proof}

\begin{coro}[Stability of discrete quasi-geodesics]
\label{stability discrete quasi-geodesics}
	Let $l \geq 0$.
	There exist $L >0$ and $d \geq 0$ depending only on $l$ (not on $X$ nor $\delta$) with the following property.
	If $x_0, \dots, x_m$ is a sequence of points of $X$, such that for all $i \in \intvald 0{m-2}$, $\dist{x_{i+1}}{x_i} \geq L\delta$ and $\gro {x_i}{x_{i+2}}{x_{i+1}} \leq l\delta$.
	Then the Hausdorff distance between $\geo{x_0}{x_1}\cup \dots \cup \geo {x_{m-1}}{x_m}$ and $\geo{x_0}{x_m}$ is less than $d\delta$.
\end{coro}

If we only consider local geodesics, one can give simple quantitative estimations for the constants which appear in the stability of quasi-geodesics.
They will be often used later.

\begin{prop}
\label{stability local geodesic}
	Let $L >\sympy{F_stabQGL(delta)}$.
	The Hausdorff distance between two $L$-local geodesics joining the same endpoints of $X$ (\resp  $X\cup \partial X$) is at most $\sympy{F_stabQGd1(delta)}$ (\resp $\sympy{F_stabQGd2(delta)}$).
	Moreover every $L$-local geodesic is a (global) $(k,0)$-quasi-geodesic with $k=\sympy{F_stabQGk(L,delta)}$.
\end{prop}

\begin{proof}
	The case where the local geodesics join two points of $X$ is done in \cite[Chap. III.H, Th. 1.13]{BriHae99}.
	The general case follows then as in \cite[Chap. 3, Th. 3.1]{CooDelPap90}.
\end{proof}


\subsection{Isometries}

	\paragraph{} 
	In this section we assume that $X$ is geodesic and proper i.e., every close ball is compact.
	Let $g$ be an isometry of $X$. 
	In order to measure its action on $X$, we define two translation lengths.
	By the \emph{translation length} $\len[espace=X]g$ (or simply $\len g$) we mean 
	\begin{displaymath}
		\len[espace=X] g = \inf_{x \in X} \dist {gx}x.
	\end{displaymath}
	The \emph{asymptotic translation length} $\len[stable, espace=X] g$ (or simply $\len[stable]g$) is
	\begin{displaymath}
		\len[espace=X,stable] g = \lim_{n \rightarrow + \infty} \frac 1n \dist{g^nx}x.
	\end{displaymath}
	These two lengths satisfy the following inequality $\len[stable]g \leq \len g \leq \len[stable] g + \sympy{F_comparaisonTLength(delta)}$ (see \cite[Chap. 10, Prop 6.4]{CooDelPap90}).
	The \emph{axis} $A_g$ of $g$, defined as follows, is a $\sympy{F_axexQC(delta)}$-quasi-convex subset of $X$ (see \cite[Prop. 2.3.3]{DelGro08}).
	\begin{displaymath}
		A_g = \left\{ \fantomB x \in X/ \dist {gx}x \leq \max\left\{\len g,\sympy{F_axesDef(delta)}\right\} \right\}
	\end{displaymath}
	The isometry $g$ is \emph{hyperbolic} if its asymptotic translation length is positive.
	In this case, $g$ fixes exactly two points of $\partial X$ denoted by $g^-$ and $g^+$.
	The cylinder of $g$, denoted by $Y_g$, is defined to be the set of points of $X$ which are $\sympy{F_cylDef(delta)}$-close to some geodesic joining $g^-$ and $g^+$.
	It is a $g$-invariant, strongly quasi-convex subset of $X$. 
	
	\begin{prop}[see {\cite[Prop 2.3]{Cou10a}}]
	\label{embedded quasi-convex in an axis}
		Let $g$ be a hyperbolic isometry of $X$.
		We denote by $\geo{g^-}{g^+}$ a geodesic joining the points of $\partial X$ fixed by $g$.
		Then $\geo{g^-}{g^+}$ is contained in the $\sympy{F_geoInAxes(delta)}$-neighbourhood of $A_g$.
		In particular $Y_g$ lies in the $\sympy{F_cylInAxes(delta)}$-neighbourhood of $A_g$.
	\end{prop}

	Let $g$ be an isometry of $X$ such that $\len g > \sympy{F_nerveTL(delta)}$.
	(In particular, $g$ is hyperbolic.)
	Let $x$ be a point of $A_g$.
	We consider a geodesic $N : J \rightarrow X$ between $x$ and $gx$ parametrized by arc length. 
	We extend $N$ in a $g$-invariant path $N : \R \rightarrow X$ in the following way: for all $t \in J$, for all $m \in \Z$, $N\left(t + m \len g \right) = g^mN(t)$. 
	This is a $\len g$-local geodesic contained in $A_g$.
	We call such a path a \emph{nerve} of $g$.
	It is a very convenient tool for the proofs.
	Indeed $N$ is homeomorphic to a line on which $g$ acts by translation of length $\len g$.
	Moreover the Hausdorff distance between $N$ and $Y_g$ is less than $\sympy{F_nerveHausdorffCyl(delta)}$.
	Therefore one can replace $Y_g$ by $N$ with a little error. 
	We summarize here some of its properties which follow from the stability of the local geodesics and the projection on a quasi-convex.
	In order to lighten the proofs we will later use these facts without any justification.
	
	\paragraph{}The nerve $N$ is $\sympy{F_nerveQC(delta)}$-quasi-convex.
	Given two points $u = N(s)$ and $v = N(t)$ of $N$, we denote by $\nerf uvN$ the path $N\left(\geo st\right)$.
	The path $N$ is injective thus this definition makes sense.
	\begin{itemize}
		\item Let $x$ be a point of $X$ and $y$ its projection on $N$, for all $y' \in N$ and $z \in \nerf y{y'}N$, $\gro x{y'}y \leq \sympy{F_nerveProj0(delta)}$ and $\gro x{y'}z \leq \sympy{F_nerveProj1(delta)}$.
		\item Let $x$, $x'$ be two points of $X$ and $y$, $y'$ their respective projections on $N$.
		If $\dist y{y'} > \sympy{F_nerveProjAssumption(delta)}$ then for all $z \in \nerf y{y'}N$, $\gro x{x'}y \leq \sympy{F_nerveProj2(delta)}$ and $\gro x{x'}z \leq  \sympy{F_nerveProj3(delta)}$.
		\item For all $x, x' \in X$, we have $\dist{d\left(x,N\right)}{d\left(x,Y_g\right)}\leq \sympy{F_nerveDistPointCyl(delta)}$. 
		On the other hand, $\dist{\fantomB\diaminter {\geo x{x'}}N}{\diaminter{\geo x{x'}}{Y_g}} \leq \sympy{F_nerveDiamInterCyl(delta)}$.
	\end{itemize}
	
	\begin{lemm}
	\label{distance axe gromov product}
		Let $g$ be an isometry of $X$ such that $\len g > \sympy{F_axesDistVsGromovTL(delta)}$.
		For all $x \in X$ we have
		\begin{displaymath}
			\dist{\gro{gx}{g^{-1}x}x}{d\left(x, Y_g \right)} \leq \sympy{F_axesDistVsGromov(delta)}.
		\end{displaymath}
	\end{lemm}
	
	\begin{proof}
		We denote by $t$ the Gromov product $\gro{gx}{g^{-1}x}x$.
		Let $N$ be a nerve of $g$ and $y$ a projection of $x$ on $N$.
		By hyperbolicity we have
		\begin{displaymath}
			t - \gro{gx}{g^{-1}x}y 
			\leq \dist xy
			\leq t +\max\left\{ \gro x{gx}y, \gro x{g^{-1}x}y\right\} +\sympy{F_hyp1(delta)}.
		\end{displaymath}
		However $\len g > \sympy{F_axesDistVsGromovTL(delta)}$ hence $\dist{gy}{g^{-1}y}  > \sympy{F_nerveProjAssumption(delta)}$.
		Consequently $\gro x{gx}y \leq  \sympy{F_nerveProj2(delta)}$, $\gro x{g^{-1}x}y \leq  \sympy{F_nerveProj2(delta)}$ and $\gro{g^{-1}x}{gx}y \leq  \sympy{F_nerveProj3(delta)}$.
		It follows that $\dist t{\dist xy} \leq  \sympy{Faux_axesDistVsGromov(delta)}$.
		However $\dist xy$ is exactly $d\left(x, N \right)$.
		Hence $\dist t{d\left(x, Y_g \right)} \leq  \sympy{F_axesDistVsGromov(delta)}$.
	\end{proof}
		
	\begin{lemm}
	\label{shortening geodesic}
		Let $a \geq 0$.
		Let $g$ be an isometry of $X$ such that $\len g > \sympy{F_shorteningGeoTL(delta)}$.
		Let $x$ and $x'$ be two points of $X$.
		We assume that $\diaminter{\geo x{x'}}{Y_g}>\len g/2 + a  >\sympy{F_shorteningGeoAssumption(delta)}$.
		Then there exists $k \in \Z$ such that
		\begin{math}
			\dist{g^kx'}x < \dist {x'}x - a + \sympy{F_shorteningGeo(delta)}.
		\end{math}
	\end{lemm}
	
	\begin{proof}
		Let $N$ be a nerve of $g$.
		Its $\sympy{F_nerveHausdorffCyl(delta)}$-neighbourhood contains $Y_g$, therefore $\diaminter {\geo x{x'}}N > \len g/2 + a -\sympy{F_nerveDiamInterCyl(delta)}$.
		We denote by $y$ and $y'$ respective projections of $x$ and $x'$ on $N$.
		Lemma~\ref{res:diam quasi-convex and projections} gives $\dist {y'}y > \len g/2+a -\sympy{Faux_shorteningGeo(delta)} > \sympy{F_nerveProjAssumption(delta)}$.
		Combined with the projection on $N$ we obtain
		\begin{displaymath}
			\dist {x'}x 
			> \dist {x'}{y'} +  \frac 12 \len g + a + \dist yx  - \sympy{F_shorteningGeo(delta)}
		\end{displaymath} 
		On the other hand $g$ acts on $N$ by translation of length $\len g$.
		Hence there exists $k \in \Z$ such that $\dist {g^ky'}y \leq \len g/2$.
		The triangle inequality yields
		\begin{displaymath}
			\dist{g^kx'}x  
			\leq \dist {x'}{y'} +  \frac 12 \len g + \dist yx
			< \dist {x'}x - a + \sympy{F_shorteningGeo(delta)},
		\end{displaymath}
		which completes the proof.
	\end{proof}

	\begin{lemm}
	\label{shortening class length}
		Let $a \geq 0$.
		Let $g$ and $h$ be two isometries of $X$ such that $\len g > \sympy{F_shorteningIsomTL(delta)}$.
		We assume that 
		\begin{displaymath}
			\min\left\{ \fantomB \len h , \diaminter{Y_h}{Y_g}\right\}> \frac 12 \len g +a> \sympy{F_shorteningIsomAssumption(delta)}.
		\end{displaymath}
		Then, there exists $k \in \Z$ such that $\len{g^kh} < \len h - a + \sympy{F_shorteningIsom(delta)}$.
	\end{lemm}
	
	\begin{proof}
		Let $N$ be a nerve of $h$.
		Since $Y_h$ lies in the $\sympy{F_nerveHausdorffCyl(delta)}$-neighbourhood of $N$ we have $\diaminter {Y_g}N > \len g/2 +a -\sympy{F_nerveDiamInterCyl(delta)}$.
		Hence there exist $x$ and $x'$ in $Y_g$ such that $\diaminter {\geo x{x'}}N > \len g/2 +a -\sympy{F_nerveDiamInterCyl(delta)}$.
		We denote by $y=N(t)$ and $y'=N(t')$ respective projections of $x$ and $x'$ on $N$.
		Up to change the role of $x$ and $x'$ we can assume that $t' \geq t$.
		Recall that $N$ is parametrized by arclength. 
		Hence Lemma~\ref{res:diam quasi-convex and projections} gives 
		\begin{displaymath}
			\dist {t'}t \geq \dist {y'}y > \frac 12\len g+a -\sympy{Faux1_shorteningIsom(delta)} > \sympy{F_nerveProjAssumption(delta)}.
		\end{displaymath}
		Let us set $s = \len g/2 + a - \sympy{Faux1_shorteningIsom(delta)}$ and $z = N(t+s)$.
		The isometry $h$ acts on $N$ by translation of length $\len h$, thus $hy = N\left(t+\len h\right)$.
		Note that $t \leq t+ s \leq \min\left\{t', t+ \len h\right\}$.
		Consequently $\gro y{hy}z \leq \sympy{F_nerveGromovIntern(delta)}$ and 
		\begin{displaymath}
			\dist x{x'} \geq \dist xy + \dist yz + \dist z{x'} - \sympy{Faux2_shorteningIsom(delta)}.
		\end{displaymath}
		In particular $\diaminter{Y_g}{\geo yz} \geq \dist yz -\sympy{Faux3_shorteningIsom(delta)}$.
		It follows that 
		\begin{displaymath}
			\diaminter{Y_g}{\geo y{hy}} \geq \diaminter{Y_g}{\geo yz} - \gro y{hy}z \geq \dist yz - \sympy{Faux4_shorteningIsom(delta)} \geq \frac 12\len g +a - \sympy{Faux5_shorteningIsom(delta)} > \sympy{F_shorteningGeoAssumption(delta)}.
		\end{displaymath}
		According to Lemma~\ref{shortening geodesic}, there exists $k \in \Z$ such that $ \dist{g^khy}{y} < \dist {hy}{y} - a + \sympy{F_shorteningIsom(delta)}$.
		However $y$ is a  point of a nerve of $h$ and thus of the axis of $h$.
		Consequently $\len {g^k h} \leq \len h - a + \sympy{F_shorteningIsom(delta)}$.
	\end{proof}
			
	\paragraph{} The goal of the next two results is to describe a figure that will naturally arise in Part~\ref{sec:small cancellation}.
	Since the proof only requires some basic properties of hyperbolicity, we give it here.
	It will considerably lighten the proofs involving foldable configurations (see Sections~\ref{sec: close points}-\ref{sec:foldable configurations}).
	The constants $a$, $b$ and $c$ which appear in the following statements will be made precise in Part~\ref{sec:small cancellation}.
	They represent distances which are large in comparison to $\delta$ but small compared to $\len g$.

	\begin{prop}
	\label{projection conditionnelle sur un axe version simple}
		Let $a,b,c \geq 0$.
		Let $g$ be an isometry of $X$ such that $\len g > \sympy{F_simpleProjAxesTL(a,b,c,delta)}$.
		Let $x$, $y$ and $z$ be three points of $X$.
		We assume that there exists a point $s \in X$ such that $\diaminter{\geo sy}{Y_g} \leq \len g/2 +\sympy{F_simpleProjAxesHypA(a)}$ and $\dist xs \leq \gro yzx + \sympy{F_simpleProjAxesHypB(b)}$.
		Let $N$ be a nerve of $g$.
		We denote by $p$ and $q$ respective projections of $y$ and $z$ on $N$.
		Let $r$ be a projection of $x$ on $\nerf pqN$.
		If $\diaminter{\geo yz}{Y_g} \geq \len g - \sympy{F_simpleProjAxesHypC(c)}$, then we have
		\begin{enumerate}
			\item	\label{enum:axe simple distances}
			\begin{list}{\labelitemi}{\leftmargin=0em}
			\renewcommand{\labelitemi}{}
				\item $\dist pq \geq \len g \sympy{-1*F_simpleProjAxesMinPQ(c,delta)}$,
				\item $\dist pr \leq \len g/2 + \sympy{F_simpleProjAxesMajPR(a,b,delta)}$,
				\item $\dist qr \geq \len g/2 \sympy{-1*F_simpleProjAxesMinQR(a,b,c,delta)}$,
			\end{list}
			\item \label{enum axe simple produit gromov}
			$\gro xyz \geq  \gro xyr + \dist zq +\dist qr  - \sympy{F_simpleProjAxesMinGro(delta)}$.
		\end{enumerate}
	\end{prop}
	
	\rem The conditions on $s$ have the following signification.
	By hyperbolicity, $\geo xy$ is contained in the $\sympy{F_thin(delta)}$-neighbourhood of $\geo xz \cup \geo zy$.
	The part of the geodesic $\geo xy$ which lies in the $\sympy{F_thin(delta)}$-neighbourhood of $\geo yz$ can not have a large overlap with the cylinder of $g$ (see Figure~\ref{fig: projection conditionnelle sur un axe version simple}).
	We could have chosen for $s$ the point of $\geo xy$ such that $\dist xs = \gro yzx$ and asked that $\diaminter {\geo sy}{Y_g} \leq \len g/2 +\sympy{F_simpleProjAxesHypA(a)}$.
	However in Part~\ref{sec:small cancellation}, we will need this more general assumption.
	\begin{figure}[ht]
		\centering
		\includegraphics{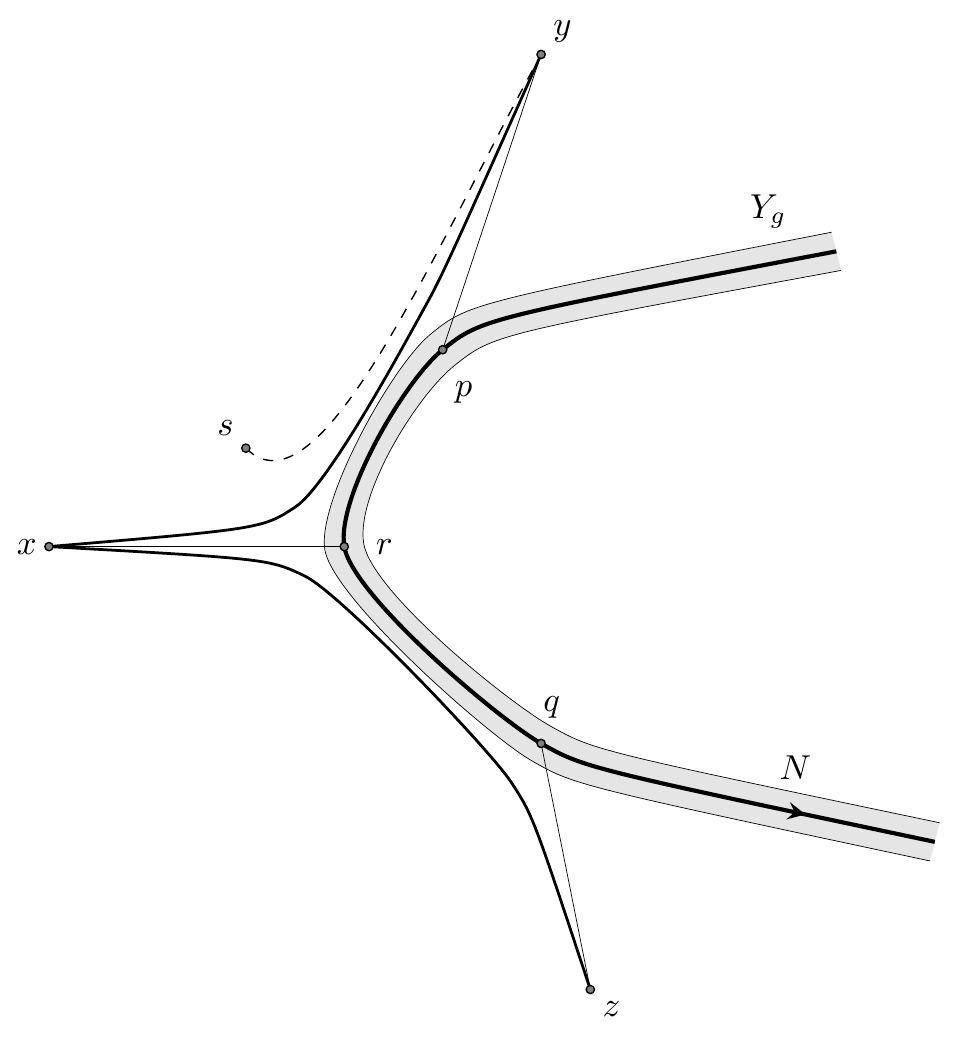}
		\caption{Signification of the point $s$}
		\label{fig: projection conditionnelle sur un axe version simple}
	\end{figure}
	
	\begin{proof}
	The $\sympy{F_nerveHausdorffCyl(delta)}$-neighbourhood of $N$ contains $Y_g$, thus $\diaminter{\geo yz}N \geq \len g \sympy{-1*Faux1_simpleProjAxes(c,delta)}$.
	Since $p$ and $q$ are respective projections of $y$ and $z$ on $N$, we get by Lemma~\ref{res:diam quasi-convex and projections} $\dist pq \geq \len g \sympy{-1*F_simpleProjAxesMinPQ(c,delta)}> \sympy{F_nerveProjAssumption(delta)}$.
	This proves the first inequality of Point~\ref{enum:axe simple distances}.
	
	\paragraph{Upper bound of $\dist pr$.}
	We may assume that $\dist pr > \sympy{Faux2_simpleProjAxes(delta)}$.
	Hence $\dist pq \geq \dist pr - \gro pqr > \sympy{F_nerveProjAssumption(delta)}$.
	The points $p$ and $q$ are respective projections of $y$, $z$ on $N$, thus $\gro yzx \leq \dist xr  + \gro yzr \leq \dist xr +\sympy{F_nerveProj3(delta)}$.
	Using our second assumption on $s$ we obtain $\dist xs \leq \dist xr + \sympy{Faux3_simpleProjAxes(b,delta)}$.
	However, by hyperbolicity we have
	\begin{displaymath}
		\gro syr \leq \max \left\{ \dist xs - \dist xr + 2 \gro xyr, \gro xyr \right\} +\sympy{F_hyp1(delta)}
	\end{displaymath}
	Since $\gro xyr \leq \sympy{F_nerveProj2(delta)}$ we get $\gro syr \leq \sympy{Faux4_projSimpleAxes(b,delta)}$.
	The point $p$ is a projection of $y$ on $N$.
	By Proposition~\ref{res:diam quasi-convex and projections} we have
	\begin{displaymath}
		\dist rp \leq \diaminter{\geo ry}N +\sympy{Faux5_simpleProjAxes(delta)} \leq \diaminter{\geo sy}N + \gro syr + \sympy{Faux5_simpleProjAxes(delta)} \leq \diaminter{\geo sy}{Y_g} + \sympy{Faux6_simpleProjAxes(b,delta)}.
	\end{displaymath}
	The second inequality of Point~\ref{enum:axe simple distances} follows then from the first assumption on $s$.
	
	\paragraph{Lower bound of $\dist qr$.}
	The third inequality of Point~\ref{enum:axe simple distances} follows by triangle inequality from the two previous ones.
	
	\paragraph{Estimation of $\gro xyz$.}
	As a consequence of Point~\ref{enum:axe simple distances}, $\dist qr >\sympy{Faux2_simpleProjAxes(delta)}$, thus $\gro xzr \leq \sympy{F_nerveProj2(delta)}$ and $\gro yzr \leq \sympy{F_nerveProj3(delta)}$.
	However 
	\begin{displaymath}
		\gro xyz = \gro xyr + \dist zr - \gro xzr - \gro yzr \geq \gro xyr + \dist zr - \sympy{Faux7_simpleProjAxes(delta)}.
	\end{displaymath}
	Since $q$ is a projection of $z$ on $N$ we have $\dist zr \geq \dist zq + \dist qr - \sympy{Faux8_simpleProjAxes(delta)}$, which combined with the previous inequality gives Point~\ref{enum axe simple produit gromov}.
	\end{proof}

	\begin{prop}
	\label{projection conditionnelle sur une axe version double}
		Let $a$, $b$ and $c$ be non-negative constants.
		Let $g$ be an isometry of $X$ such that $\len g > \sympy{F_doubleProjAxesTL(a,b,c,delta)}$.
		Let $x$, $y_1$ and $y_2$ be three points of $X$.
		We assume that there exist two points $s_1, s_2 \in X$ such that for all $i \in \left\{ 1,2 \right\}$, $\diaminter{\geo{s_i}{y_i}}{Y_g} \leq \len g/2 +\sympy{F_doubleProjAxesHypA(a)}$ and $\dist x{s_i} \leq \gro {y_1}{y_2}x + \sympy{F_doubleProjAxesHypB(b)}$.
		Let $N$ be a nerve of $g$.
		We denote by $r$, $q_1$ and $q_2$ respective projections of $x$, $y_1$ and $y_2$ on $N$.
		If $\diaminter{\geo {y_1}{y_2}}{Y_g} \geq \len g - \sympy{F_doubleProjAxesHypC(c)}$, then we have the followings
		\begin{enumerate}
			\item \label{enum: axe double nerf}
			$r$ belongs to $\nerf {q_1}{q_2}N$,
			\item \label{enum: axe double distances}
			\begin{list}{\labelitemi}{\leftmargin=0em}
			\renewcommand{\labelitemi}{}
				\item $\dist {q_1}{q_2} \geq \len g \sympy{-1*F_doubleProjAxesMinQs(c,delta)}$,
				\item $\len g/2 \sympy{-1*F_doubleProjAxesMinQR(a,b,c,delta)} \leq \dist r{q_i} \leq \len g/2 + \sympy{F_doubleProjAxesMajQR(a,b,delta)}$,
			\end{list}
			\item \label{enum: axe double petits produits gromov}
			$\gro x{y_i}r, \gro x{y_i}{q_i} \leq  \sympy{F_doubleProjAxesMajGro1(delta)}$ and $\gro{s_i}{y_i}{q_i} \leq \sympy{F_doubleProjAxesMajGro2(delta)}$.
			\item \label{enum: axe double gro x y1 y2}
			$\dist{\fantomB \gro{y_1}{y_2}x}{\dist xr} \leq \sympy{F_doubleProjAxesDiffGro(delta)}$.
		\end{enumerate}
	\end{prop}
	
	\rem Intuitively, we have Figure~\ref{fig: projection conditionnelle sur un axe version double} in mind.
		The goal of this proposition is to prove that this picture actually corresponds to the reality.
		
	\begin{figure}[ht]
		\centering
		\includegraphics{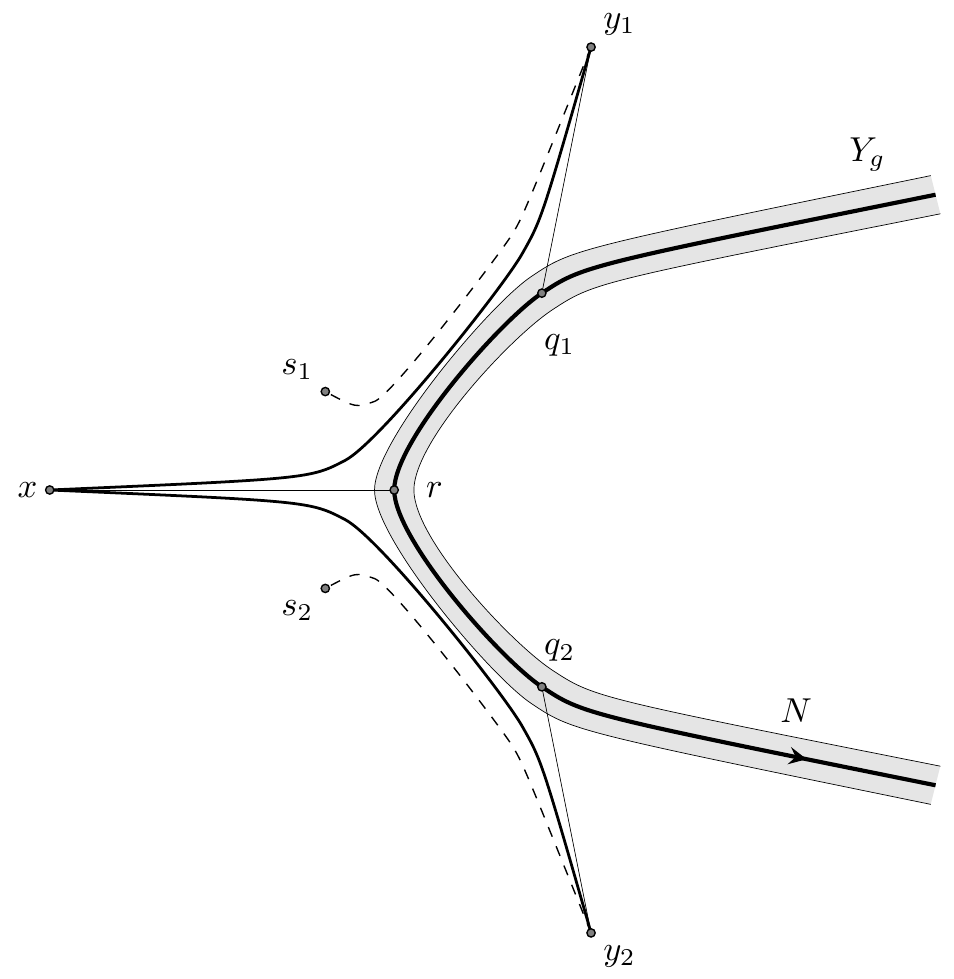}
		\caption{Positions of the points $q_1$, $q_2$ and $r$.}
		\label{fig: projection conditionnelle sur un axe version double}
	\end{figure}

	 \begin{proof}
		We prove Point~\ref{enum: axe double nerf} by contradiction.
		Assume that $r$ does not belong to $\nerf {q_1}{q_2}N$.
		By symmetry we can assume that $q_1$ is a point of $\nerf r{q_2}N$.
		Let $q$ be a point of $\nerf {q_1}{q_2}N$.
		Since $r$ is a projection of $x$ on $N$, $\dist xq \geq \dist xr + \dist rq - \sympy{Faux1_doubleProjAxes(delta)}$.
		However $q_1$ lies on $N$ between $r$ and $q$.
		Therefore we obtain $\dist xq \geq \dist x{q_1} -\sympy{Faux2_doubleProjAxes(delta)}$.
		Consequently $q_1$ is a $\sympy{Faux2_doubleProjAxes(delta)}$-projection of $x$ on $\nerf{q_1}{q_2}N$.
		By Proposition~\ref{projection quasi-convex}, the distance between $q_1$ and a projection $t$ of $x$ on $\nerf{q_1}{q_2}N$ is at most $\sympy{Faux3_doubleProjAxes(delta)}$.
		Nevertheless Proposition~\ref{projection conditionnelle sur un axe version simple} Point~\ref{enum:axe simple distances}  gives $\dist {q_1}t \geq \len g /2 \sympy{-1*F_simpleProjAxesMinQR(a,b,c,delta)}$.
		Contradiction.
		Hence $r$ belongs to $\nerf {q_1}{q_2}N$.
		Therefore, Point~\ref{enum: axe double distances} follows from Proposition~\ref{projection conditionnelle sur un axe version simple}.
		
		\paragraph{}
		The points $r$ and $q_i$ are respective projections of $x$ and $y_i$ on $N$.
		Thus $\gro x{y_i}r, \gro x{y_i}{q_i} \leq \sympy{F_nerveProj2(delta)}$ and $\gro {y_1}{y_2}r \leq \sympy{F_nerveProj3(delta)}$, which proves in particular the first part of Point~\ref{enum: axe double petits produits gromov}.
		The hyperbolicity condition yields
		\begin{displaymath}
			\gro {y_1}{y_2}x - \gro {y_1}{y_2}r \leq \dist xr \leq \gro {y_1}{y_2}x +\max\left\{\fantomB \gro x{y_1}r, \gro x{y_2}r\right\} +\sympy{F_hyp1(delta)}
		\end{displaymath}
		which leads to Point~\ref{enum: axe double gro x y1 y2}.
		What is left to show is that $\gro {s_i}{y_i}{q_i} \leq \sympy{F_doubleProjAxesMajGro2(delta)}$.
		By hyperbolicity we have 
		\begin{displaymath}
			\gro{s_i}{y_i}{q_i} \leq \max\left\{ \dist x{s_i} - \dist  x{q_i} + 2 \gro x{y_i}{q_i}, \gro x{y_i}{q_i}\right\} +\sympy{F_hyp1(delta)}.
		\end{displaymath}
		However $\gro x{y_i}{q_i} \leq \sympy{F_nerveProj2(delta)}$, thus it is sufficient to give an upper bound to  $\dist x{s_i} - \dist x{q_i}$.
		Since $r$ is a projection of $x$ on $N$, one has $\dist x{q_i} \geq \dist xr  +\dist r{q_i}-  \sympy{Faux4_doubleProjAxes(delta)}$.
		However we already proved that $\dist xr \geq \gro {y_1}{y_2}x - \sympy{F_nerveProj3(delta)} \geq \dist x{s_i} \sympy{-1*Faux5_doubleProjAxes(b,delta)}$.
		Hence $\dist x{q_i} \geq \dist x{s_i} +\dist r{q_i} \sympy{-1*Faux6_doubleProjAxes(b,delta)}$.
		It follows then from \ref{enum: axe double distances} that $ \dist x{s_i} - \dist  x{q_i} + 2 \gro x{y_i}{q_i} \leq \gro x{y_i}{q_i}$ which leads to the result.
	 \end{proof}
	 
	 \subsection{Hyperbolic groups}
	 \label{sec: hyperbolic groups}
	 
	\paragraph{}In this section $X$ is still geodesic and proper.
	We consider a group $G$ acting properly, co-compactly by isometries on $X$. It follows that every element of $G$ is either \emph{elliptic} (and has finite order) or \emph{hyperbolic} (see \cite[Chap. 9, Th. 3.4]{CooDelPap90}).
	A subgroup of $G$ is called \emph{elementary} if it is virtually cyclic.
	Every non-elementary subgroup of $G$ contains a copy of $\free 2$, the free group of rank 2 (see \cite[Chap. 8, Th. 37]{GhyHar90}). 
	Given a hyperbolic element $g$ of $G$, the subgroup of $G$ stabilizing $\left\{ g^-, g^+\right\} \subset \partial X$ is elementary.
	In particular the normalizer of $g$ is elementary (see \cite[Chap. 10, Prop 7.1]{CooDelPap90}).
	
	\nota If $P$ is a subset of $G$, we denote by $P^*$ the set of hyperbolic elements of $P$.
	
	\begin{defi}
	\label{defi:overlap and injectivity radius}
		Let $P$ be a subset of $G$.
		\begin{itemize}
			\item The \emph{injectivity radius} of $P$ on $X$, denoted by $\rinj PX$, is defined by $\rinj PX = \inf_{\rho \in P^*} \len[stable, espace=X]\rho$.
			\item The \emph{maximal overlap} of $P$ on $X$, denoted by $\Delta (P,X)$, is the quantity $ \Delta(P,X) = \sup_{\rho \neq \rho' \in P^*} \diaminter{Y_\rho}{Y_{\rho'}}$.
		\end{itemize}
	\end{defi}

	\begin{defi}
		The $A$ invariant  of $G$ on $X$, denoted by $A(G,X)$,  is the upper bound of $\diaminter{A_g}{A_h}$, where $g$ and $h$ are two elements of $G$ which generate a non-elementary subgroup and whose translation lengths are smaller than $\sympy{F_invADef(delta)}$.
	\end{defi}
	
	\begin{prop}[see {\cite[Prop. 2.4.3]{DelGro08}}, {\cite[Prop. 2.41]{Coulon:2013tx}}]
	\label{recouvrement axes}
		We assume that every elementary subgroup of $G$ is cyclic.
		Let $g$ and $h$ be two elements of $G$ such that $\len g \leq \sympy{F_invADef(delta)}$.
		If the subgroup generated by $g$ and $h$ is non-elementary, then
		\begin{displaymath}
			\diaminter{A_g}{A_h} \leq \len h + A(G,X) + \sympy{F_invAProp(delta)}.
		\end{displaymath}
	\end{prop}
	
	\voc \label{def:small centralizers hypothesis}
	The group $G$ satisfies the \emph{small centralizers hypothesis} if $G$ is non-elementary and every elementary subgroup of $G$ is cyclic.

%% file: 2_cone_off.tex

\section{Cone-off over a metric space}
\label{sec:cone-off over a metric space}

\paragraph{} In this section we focus on the cone-off over a metric space (see \cite{DelGro08}).
Let us fix a positive real number $r_0$.
Its value will be made precise later.
It should be thought as a very large scale parameter.

\subsection{Cone over a metric space}
\label{sec:cone}

\paragraph{} We review the construction of a cone over a metric space. 
For more details see \cite[Chap. I.5]{BriHae99}.
Let $Y$ be metric space.
The cone of radius $r_0$ over $Y$, denoted by $Z(Y,r_0)$ (or simply $Z(Y)$) is the quotient of $Y\times \intval 0{r_0}$ by the equivalence relation which identifies all the points of the form $(y,0)$, $y \in Y$.
The equivalence class of $(y,0)$ is the apex of the cone, denoted by $v$.
We endow $Y$ with a metric characterized as follows.
Given any two points $x=(y,r)$ and $x'=(y',r')$ of $Z(Y)$, 
\begin{displaymath}
	\cosh \left( \dist x{x'}\right) = \cosh r \cosh r' - \sinh r \sinh r' \cos \left( \min \left\{ \pi , \frac {\dist y{y'}}{\sinh r_0}\right\}\right).
\end{displaymath}
In order to compare the cone $Z(Y)$ and its base $Y$ we introduce two maps.
\begin{displaymath}
	\begin{array}{lcccp{2cm}lccc}
	\iota:	& Y	& \rightarrow	& Z(Y)	& & p:	& Z(Y)\setminus\{v\}	& \rightarrow	& Y \\
		& y	& \rightarrow	& (y,r_0)	& &		& (y,r)			& \rightarrow	& y
	\end{array}
\end{displaymath}
If $y$ and $y'$ are two points of $Y$, the distance between $\iota(y)$ and $\iota(y')$ is then given by $\dist{\iota(y)}{\iota(y')} = \mu\left(\dist y{y'}\right)$ where $\mu : \R^+ \rightarrow \R^+$ is defined in the following way: for all $t \in \R^+$, 
\begin{displaymath}
	\cosh\left(\mu(t)\right) = \cosh^2 r_0 - \sinh^2r_0\cos \left( \min \left\{ \pi , \frac t{\sinh r_0}\right\}\right).
\end{displaymath}
The function $\mu$ is non-decreasing, concave and subadditive.
Moreover, for all $t \in \R^+$, $\mu(t) \leq t$ (see \cite{Coulon:il}).
A coarse computation proves also that for all $t \in \left[0, \pi \sinh r_0\right]$, $t \leq \pi\sinh\left(\mu(t)/2\right)$.
It follows from the concavity that for every $r,s,t \geq 0$ 
\begin{equation}
\label{eqn: prop mu for shorthening}
	\mu(r+s) \leq \mu(r+t) +\mu(t+s) - \mu(t)
\end{equation}
\paragraph{}If $Y$ is a length space, so is $Z(Y)$.
More precisely, let $x=(y,r)$ and $x'=(y',r')$ be two points of $Z(Y)$.
Let $\sigma : I \rightarrow Y$ be a rectifiable path between $y$ and $y'$.
If its length $L(\sigma)$ is strictly smaller than $\pi \sinh r_0$, then there exists a rectifiable path $\tilde \sigma : I \rightarrow Z(Y)\setminus\{v\}$ between $x$ and $x'$ such that $p\circ \tilde \sigma = \sigma$ and whose length satisfies
\begin{displaymath}
	\cosh \left(L\left(\tilde \sigma\right)\right) \leq \cosh r \cosh r' - \sinh r \sinh r' \cos \left(  \frac {L(\sigma)}{\sinh r_0}\right).
\end{displaymath}

\paragraph{} We now consider a group $H$ acting properly, by isometries on $Y$.
We denote by $\bar Y$ the quotient $Y/H$. 
For all $y \in Y$, we write $\bar y$ for the image of $y$ in $\bar Y$.
The space $\bar Y$ is endowed with a metric defined by $\dist{\bar y}{\bar y'} = \inf_{h \in H}\dist y{hy'}$.
The action of $H$ on $Y$ can be extended to $Z(Y)$ by homogeneity: if $(y,r) \in Z(Y)$ and $h \in H$, then $h(y,r) = (hy,r)$.
Hence $H$ acts on $Z(Y)$ by isometries.
If $Y$ is not compact, this action may not be proper. 
The stabilzer of $v$ (i.e. $H$) may indeed be not finite.
Nevertheless the formula $\dist{\bar x}{\bar x'} = \inf_{h \in H}\dist x{hx'}$ still defines a metric on $Z(Y)/H$.
Moreover the spaces $Z(Y)/H$ and $Z(Y/H)$ are isometric (see \cite{Coulon:il}).

\begin{lemm}
\label{metric quotient cone}
	Let $l \geq 2\pi \sinh r_0$.
	We assume that for every $h \in H\setminus \{1\}$, $\len h \geq l$.
	Let $x=(y,r)$ and $x'=(y',r')$ be two points of $Z(Y)$.
	If $\dist[Y]y{y'} \leq l - \pi \sinh r_0$ then $\dist{\bar x}{\bar x'} = \dist x{x'}$.
\end{lemm}

\begin{proof}
	Since $Z(Y/H)$ and $Z(Y)/H$ are isometric, the distance between $\bar x$ and $\bar x'$ in $Z(Y)/H$ is given by 
	\begin{displaymath}
		\cosh \left( \dist{\bar x}{\bar x'}\right) = \cosh r \cosh r' - \sinh r \sinh r' \cos \left( \min \left\{ \pi , \frac {\dist[\bar Y]{\bar y}{\bar y'}}{\sinh r_0}\right\}\right).
	\end{displaymath}
	If $\dist y{y'} <l/2$, then we have $\dist{\bar y}{\bar y'} = \dist y{y'}$.
	It follows that $\dist{\bar x}{\bar x'} = \dist x{x'}$.
	Assume now that $\dist y{y'} \geq l/2$. 
	In particular $\dist y{y'} \geq \pi \sinh r_0$.
	Thus $\dist x{x'} = r + r'$.
	On the other hand, using the triangle inequality, for all $h \in H\setminus\{1\}$, $\dist  y{hy'} \geq l - \dist y{y'}$, thus $\dist {\bar y}{\bar y'} \geq \pi \sinh r_0$.
	Consequently $\dist{\bar x}{\bar x'} = r+ r' = \dist x{x'}$.
\end{proof}

\subsection{Cone-off over a metric space}
\label{sec:cone-off}

\paragraph{} We give here a brief exposition of the construction of the cone-off over a metric space.
For details and proofs we refer the reader to \cite{Coulon:il} and \cite{Coulon:2013tx}.
For the rest of this section $X$ denotes a geodesic, $\delta$-hyperbolic space and $Y = \left(Y_i\right)_{i \in I}$ a family of strongly quasi-convex subsets of $X$ (see Definition~\ref{def quasi-convex}).

\begin{defi}
	The maximal overlap between the $Y_i$'s is measured by the quantity
\begin{displaymath}
	\Delta(Y) = \sup_{i \neq j} \diaminter{Y_i}{Y_j}.
\end{displaymath} 
\end{defi}

For all $i \in I$ we define the following objects:
\begin{enumerate}
	\item $Y_i$ is endowed with the length metric $\distV[Y_i]$ induced by the restriction to $Y_i$ of $\distV[X]$.
	Since $Y_i$ is strongly quasi-convex, for all $y,y' \in Y_i$ we have
	\begin{displaymath}
		\dist[X]y{y'} \leq \dist[Y_i]y{y'} \leq \dist[X]y{y'} + \sympy{F_SQCLength(delta)}.
	\end{displaymath}
	\item $Z_i$ is the cone of radius $r_0$ over $\left(Y_i, \distV[Y_i]\right)$ and $v_i$ its apex.
	\item $\iota_i : Y_i \rightarrow Z_i$ and $p_i : Z(Y_i) \setminus\{v_i\} \rightarrow Y_i$ are the comparison maps defined in the previous section.
\end{enumerate}
The \emph{cone-off} of radius $r_0$ over $X$ relatively to $Y$ is the space obtained by attaching each cone $Z_i$ on $X$ along $Y_i$ according to $\iota_i$.
We denote it by $\dot X(Y, r_0)$ or simply $\dot X$.

\paragraph{}
The next step is to define a metric on $\dot X$.
Given $x$ and $x'$ two points of $\dot X$ we denote by $\dist[SC]x{x'}$ the minimal distance between two points of $X \sqcup \left(\bigsqcup_{i \in I} Z_i\right)$ whose images in $\dot X$ are respectively $x$ and $x'$.

\rem 
If $x$ and $x'$ are two points of the base $X$, $\dist[SC]x{x'}$ can be computed as follows:
\begin{displaymath}
	\dist[SC]x{x'} = \min \left[ \fantomB\dist[X]x{x'}, \inf \left\{\mu \left( \dist[Y_i]x{x'}\right) | i \in I,\; x,x' \in Y_i \right\}\right].
\end{displaymath}
In particular, 
\begin{displaymath}
\mu \left(\dist[X]x{x'}\right) \leq \dist[SC]x{x'} \leq \dist[X] x{x'}.
\end{displaymath}
Moreover, if there is $ i \in I$ such that $x,x' \in Y_i$ then $\dist[SC] x{x'} \leq \mu\left(\dist[X]x{x'}\right) +\sympy{F_diffSCMu(delta)}$.
	
\begin{defi}
	Let $x$ and $x'$ be two points of $\dot X$.
	A chain between $x$ and $x'$ is a finite sequence $C=\left(z_1,\dots,z_m\right)$ such that $z_1 = x$ and $z_m = x'$.
	Its length is $l(C) = \dist[SC]{z_1}{z_2}+ \dots+ \dist[SC]{z_{m-1}}{z_m}$.
\end{defi}

\begin{prop}
Given $x$ and $x'$ in $\dot X$, the following formula defines a length metric on $\dot X$.
	\begin{displaymath}
		\dist[\dot X]x{x'} = \inf \left\{ l(C) | C \text{ chain between }x \text{ and }x'\right\}.
	\end{displaymath}
\end{prop}

\paragraph{}Note that given a chain between two points of $X$, one can always find an shorter chain joining the same extremities, whose points belong to $X$. (Just apply the triangle inequality in $X \sqcup \left(\bigsqcup_{i \in I} Z_i\right)$.)
Therefore, in the rest of the section, we will only consider chains whose points lie in $X$.

\rem In the rest of Section~\ref{sec:cone-off over a metric space}, we will work with two metric spaces : $X$ and $\dot X$.
Unless stated otherwise all distances, Gromov's products and geodesics are computed with the distance of $X$.
To avoid any confusion the distance between two points $x$ and $x'$ in $\dot X$ will be written $\dist[\dot X] x{x'}$.

\begin{theo}[see {\cite[Prop. 6.4]{Coulon:2013tx} or \cite[Coro. 5.27]{Dahmani:2011vu}}]
\label{cone-off globally hyperbolic}
	There exist positive numbers $\delta_0$, $\delta_1$ and $\Delta_0$ and $r_1$ which do not depend on $X$ or $Y$ with the following property.
	If $r_0 \geq r_1$, $\delta \leq \delta_0$ and $\Delta(Y) \leq \Delta_0$, then $\dot X(Y,r_0)$ is $\delta_1$-hyperbolic.
\end{theo}

\subsection{Shortening chains}
\label{sec:shortening chains}

\paragraph{} Our goal is now to compare the geometry of $\dot X$ and $X$.
In \cite{DelGro08}, T.~Delzant and M.~Gromov proved that the natural map $X \rightarrow \dot X$ restricted to any ball of radius $1000 \delta$ is a quasi-isometric embedding. 
For our purpose we need to compare $X$ and $\dot X$ at a larger scale.
In particular we have to take into account paths passing through the apices of $\dot X$.

\paragraph{}Coarsly speaking we prove that the projection $p$ preserves the shapes.
For instance if $x$ and $x'$ are two points of $X$, the projection by $p$ of a quasi-geodesic of $\dot X$ between them remains in the neighbourhood of any geodesic of $X$ joining $x$ and $x'$ (see Proposition~\ref{projection of quasi-geodesic}).
To that end we proceed in two steps.
Let $x$, $y$, $z$ and $t$ be four points of $X$.
If $\gro xty$ or $\gro xtz$ is large (compare to $\Delta(Y)$ and $\delta$) we first explain how to shorten the chain $C=(x,y,z,t)$ (see Proposition~\ref{res: shortening chain - 4 points}).
Then we combine this fact with the stability of discrete quasi-geodesics to show that the points of a chain between $x$ and $x'$ whose length approximates $\dist[\dot X] x{x'}$ lie in the neighbourhood of $\geo x{x'}$ (see Proposition~\ref{res: chain in neighbourhood of geodesic}).

\begin{lemm}
\label{res: shortening chain - 2 points}
	Let $x, x' \in X$ and $p,p' \in \geo x{x'}$.
	There exists a chain $C$ joining $p$ to $p'$ whose length is at most $\dist[SC] x{x'} +\sympy{F_shorteningChain2Points(delta)}$.
\end{lemm}

\begin{proof}
	If $\dist[SC] x{x'} = \dist x{x'}$ then the chain $C = (p,p')$ works.
	Thus we can assume that there exists $i \in I$ such that $x, x'\in Y_i$.
	The subset $Y_i$ being $\sympy{F_SQCQc(delta)}$-quasi-convex, there are $q, q' \in Y_i$ such that $\dist pq \leq \sympy{F_SQCQc(delta)}$ and $\dist{p'}{q'} \leq \sympy{F_SQCQc(delta)}$.
	We choose for $C$ the chain $C = (p,q,q',p')$.
	Its length is bounded above by $\mu\left(\dist q{q'}\right) +\sympy{Faux1_shorteningChain2Points(delta)}$.
	However $\dist q{q'} \leq \dist x{x'} +\sympy{Faux2_shorteningChain2Points(delta)}$.
	Consequently $l(C) \leq \mu\left(\dist x{x'}\right) +\sympy{F_shorteningChain2Points(delta)} \leq \dist[SC]x{x'} + \sympy{F_shorteningChain2Points(delta)}$.
\end{proof}

\begin{lemm}
\label{res: shortening chain - 3 points}
	Let $x,y,z \in X$, $p \in \geo xy$ and $q \in \geo yz$.
	We assume that there is $i \in I$ such that $x,y \in Y_i$ but there is no $j \in I$ such that $x,y,z \in Y_j$.
	Then there exists a chain $C$ joining $p$ to $z$ satisfying
	\begin{displaymath}
		l(C) \leq 2\dist pq + \dist[SC] yz - \dist yq + \Delta(Y) +\sympy{F_shorteningChain3Points(delta)}.
	\end{displaymath}
\end{lemm}

\begin{proof}
	We distinguish two cases.
	Assume first that there exists $j \in I$ such that $y,z \in Y_j$.
	According to our hypothesis we necessary have $i \neq j$.
	Therefore $\diaminter{\geo xy}{\geo yz}\leq \diaminter{Y_i}{Y_j} \leq \Delta(Y)$ i.e., $\gro xzy \leq \Delta(Y)$.
	It follows from the triangle inequality that 
	\begin{equation}
	\label{eqn: shortening chain - 3 points}
		\dist yq \leq \gro xzy + \dist pq \leq \dist pq +\Delta(Y).
	\end{equation}
	By Lemma~\ref{res: shortening chain - 2 points}, there exists a chain $C_0$ joining $q$ to $z$ whose length is at most $\dist[SC] yz + \sympy{Faux1_shorteningChain3Points(delta)}$.
	We obtain $C$ by adding $p$ at the beginning of $C_0$.
	It satisfies $l(C) \leq \dist pq + \dist[SC] yz +\sympy{Faux1_shorteningChain3Points(delta)}$.
	Combined with (\ref{eqn: shortening chain - 3 points}) we get the required inequality.
	
	\paragraph{}Assume now that $\dist[SC] yz = \dist yz$.
	Then $\dist[SC] qz \leq \dist[SC] yz - \dist yq$.
	We choose for $C$ the chain $C = (p,q,z)$ which satisfies $l(C) \leq \dist pq +\dist[SC] yz - \dist yq$.
\end{proof}

\begin{lemm}
\label{res: shorteniing chain - pre 4 points}
	Let $x,y,z,t \in X$. 
	If there exists $i \in I$ such that $x,t \in Y_i$ then 
	\begin{displaymath}
		\dist[SC] xt \leq \dist[SC] xy + \dist[SC] yz + \dist[SC] zt - \mu\left(\max \left\{\gro xty, \gro xtz \right\}\right)  + \sympy{F_shorteningChainPre4Points(delta)}
	\end{displaymath}
\end{lemm}

\begin{proof}
	Since $x$ and $t$ are in $Y_i$, $\dist[SC] xt \leq \mu\left(\dist xt\right) + \sympy{F_shorteningChainPre4Points(delta)}$.
	Applying~(\ref{eqn: prop mu for shorthening}) we get 
	\begin{displaymath}
		\mu(\dist xt) \leq \mu(\dist xy) +\mu(\dist yt) -\mu(\gro xty).
	\end{displaymath}
	However by triangle inequality $\mu(\dist yt) \leq \mu(\dist yz) +\mu(\dist zt)$.
	Consequently $\dist[SC] xt \leq \dist[SC] xy + \dist[SC] yz + \dist[SC] zt - \mu(\gro xty ) +\sympy{F_shorteningChainPre4Points(delta)}$.
	By symmetry we have the same inequality with $\gro xtz$ instead of $\gro xty$.
\end{proof}
\begin{prop}
\label{res: shortening chain - 4 points}
	Let $x,y,z,t \in X$.
	There exists a chain $C$ joining $x$ to $t$  such that 
	\begin{displaymath}
		l(C) \leq \dist[SC] xy + \dist[SC] yz + \dist[SC] zt - \mu\left(\max \left\{\gro xty, \gro xtz \right\}\right) + 2 \Delta(Y) + \sympy{F_shorteningChain4Points(delta)}
	\end{displaymath}
\end{prop}

\begin{proof}
	If there is $i \in I$ such that $x,t \in Y_i$, Lemma~\ref{res: shorteniing chain - pre 4 points} says that the chain $C = (x,t)$ works.
	Therefore, for now on we assume that there is no such $i \in I$.
	By hyperbolicity 
	\begin{equation}
	\label{eqn: shortening chain - 4 points}
		\dist xz + \dist yt \leq \max  \left\{ \fantomB\dist xy + \dist zt, \dist xt + \dist yz\right\} +\sympy{F_hyp2(delta)}
	\end{equation}

	\paragraph{Part 1:}Assume first that the maximum is achieved by $\dist xt + \dist yz$.
	See Figure~\ref{fig: shortening chain 1}.
	In particular it follows that $\gro xtz \leq \gro ytz +\sympy{F10aux1_shorteningChain4Points(delta)}$ and $\gro xty \leq \gro xzy +\sympy{F10aux1_shorteningChain4Points(delta)}$.
	Moreover $\dist yz \geq \gro xty + \gro xtz - \sympy{F10aux1_shorteningChain4Points(delta)}$.
	We denote by $p$ and $q$ (\resp $r$ and $s$) points of $\geo xy$ and $\geo yz$ (\resp $\geo tz$ and $\geo zy$) such that 
	\begin{displaymath}
		\dist yp = \dist yq = \max\{0,\gro xty - \sympy{F10aux1_shorteningChain4Points(delta)}\} 
		\text{ and } 
		\dist zr = \dist zs = \max\{0,\gro xtz - \sympy{F10aux1_shorteningChain4Points(delta)}\}.
	\end{displaymath}
	By hyperbolicity $\dist pq \leq \sympy{F10aux2_shorteningChain4Points(delta)}$ and $\dist rs \leq \sympy{F10aux2_shorteningChain4Points(delta)}$. 
	Furthermore $\dist yz \geq \dist yq + \dist sz$.
	We need to distinguish several cases depending on whether or not the points $x$, $y$, $z$ and $t$ lie in a quasi-convex $Y_i$.
	In each case we implicitly exclude the previous ones.

	\begin{figure}[ht]
	\centering
		\includegraphics[width=0.48\linewidth]{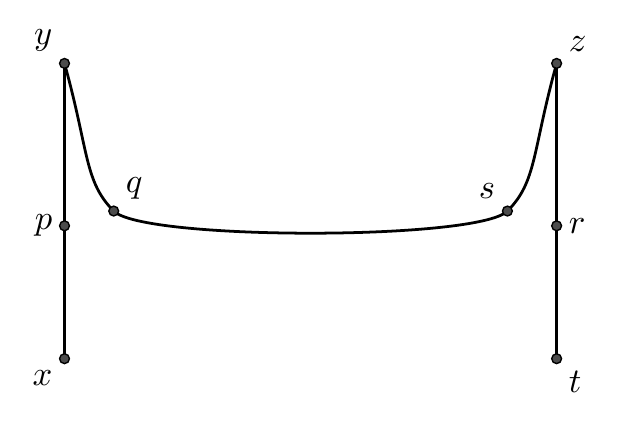}
	\caption{Shortening a four points chain - Part 1}
	\label{fig: shortening chain 1}
	\end{figure}
		
	\paragraph{Case 1.1:} \emph{There exist $i,j \in I$ such that $x,y,z \in Y_i$ and $y,z,t \in Y_j$}.
	According to our assumption at the beginning of the proof $i \neq j$.
	Since $y$ and $z$ belong to $Y_i$ and $Y_j$ they satisfy $\dist yz \leq \diaminter{Y_i}{Y_j} \leq \Delta(Y)$.
	Consequently $\gro xty + \gro xtz \leq \Delta(Y) +\sympy{F11aux1_shorteningChain4Points(delta)}$.
	We choose the chain $C = (x,y,z,t)$. 
	Thus 
	\begin{displaymath}
		l(C) \leq \dist[SC] xy + \dist[SC] yz + \dist[SC] zt - \gro xty - \gro xtz  +\Delta(Y) +\sympy{F11aux1_shorteningChain4Points(delta)}.
	\end{displaymath}

	\paragraph{Case 1.2:} \emph{There exists $i \in I$ such that $x,y,z \in Y_i$.} 
	The subset $Y_i$ being $\sympy{F_SQCQc(delta)}$--quasi-convex, there exists a point $s' \in Y_i$ such that $\dist s{s'} \leq \sympy{F12aux1_shorteningChain4Points(delta)}$.
	Hence $\dist [SC] x{s'} \leq \mu(\dist xs) +\sympy{F12aux2_shorteningChain4Points(delta)}$.
	Recall that $q$ lies on $\geo yz$ between $y$ and $s$.
	By (\ref{eqn: prop mu for shorthening}) we get 
	\begin{displaymath}
		\mu\left(\dist xs\right)
		\leq \mu\left(\dist xp + \dist qs\right) +\sympy{F10aux2_shorteningChain4Points(delta)}
		\leq \mu\left(\dist xy\right) + \mu\left(\dist ys\right) - \mu(\gro xty) + \sympy{F10aux1_shorteningChain4Points(delta) + F10aux2_shorteningChain4Points(delta)}.
	\end{displaymath}
	It follows that $\dist[SC] x{s'} \leq \dist [SC] xy +\dist[SC]yz - \mu(\gro xty) +\sympy{F12aux3_shorteningChain4Points(delta)}$.
	On the other hand, by Lemma~\ref{res: shortening chain - 3 points}, there exists a chain $C_0$ joining $s$ to $t$ such that 
	\begin{displaymath}
		l(C_0) \leq \dist[SC] zt - \dist zr +\Delta(Y) +\sympy{F_shorteningChain3Points(delta) + 2*F10aux2_shorteningChain4Points(delta)} \leq \dist[SC] zt - \gro xtz +\Delta(Y) +\sympy{F12aux4_shorteningChain4Points(delta)}
	\end{displaymath}
	We obtain $C$ by adding $x$ and $s'$ at the beginning of $C_0$.
	Its length satisfies
	\begin{displaymath}
		l(C) \leq \dist[SC] xy + \dist[SC] yz + \dist[SC] zt -\mu(\gro xty) - \gro xtz +\Delta(Y) +\sympy{F12aux5_shorteningChain4Points(delta)}.
	\end{displaymath}
	
	\paragraph{Case 1.3:} \emph{There exists $i \in I$ such that $y,z,t \in Y_i$.} This case is just the symmetric of the previous one.
	
	\paragraph{Case 1.4:} \emph{There exists $i \in I$ such that $y,z \in Y_i$.}
	By Lemma~\ref{res: shortening chain - 2 points} there exists a chain $C_0$ joining $q$ to $s$ whose length is at most $\dist[SC] yz +\sympy{F14aux1_shorteningChain4Points(delta)}$.
	Applying Lemma~\ref{res: shortening chain - 3 points}, there is a chain $C_-$ (\resp $C_+$) joining $x$ to $q$ (\resp $s$ to $t$) such that
	\begin{align*}
		l(C_-) & \leq & \dist[SC] xy - \dist yp +\Delta(Y) +\sympy{F14aux2_shorteningChain4Points(delta)} & \leq & \dist[SC] xy - \gro xty +\Delta(Y) +\sympy{F14aux3_shorteningChain4Points(delta)} \\
		l(C_+) & \leq & \dist[SC] zt - \dist zr +\Delta(Y) +\sympy{F14aux2_shorteningChain4Points(delta)} & \leq & \dist[SC] zt - \gro xtz +\Delta(Y) +\sympy{F14aux3_shorteningChain4Points(delta)}
	\end{align*}
	Concatenating $C_-$, $C_0$ and $C_+$ we obtain a chain $C$ such that 
	\begin{displaymath}
		l(C) \leq \dist[SC] xy + \dist[SC] yz + \dist[SC] zt -\gro xty - \gro xtz +2\Delta(Y) +\sympy{F14aux4_shorteningChain4Points(delta)}.
	\end{displaymath}
	
	\paragraph{Case 1.5:} \emph{This is the last case of Part 1.}
	Negating the previous one there is no $i \in I$ such that $y,z \in Y_i$.
	In particular $\dist[SC] yz = \dist yz$.
	Hence 
	\begin{displaymath}
		\dist[SC] qs \leq \dist[SC] yz - \dist yq - \dist zs \leq \dist[SC]yz -\gro xty - \gro xtz+\sympy{F15aux1_shorteningChain4Points(delta)}.
	\end{displaymath} 
	We put $C_0 = (q,s)$.
	According to Lemma~\ref{res: shortening chain - 2 points} there is a chain $C_-$ (\resp $C_+$) joining $x$ to $p$ (\resp $r$ to $t$) whose length is at most $\dist[SC] xy +\sympy{F15aux2_shorteningChain4Points(delta)}$ (\resp $\dist[SC] tz +\sympy{F15aux2_shorteningChain4Points(delta)}$).
	Concatenating $C_-$, $C_0$ and $C_+$ we obtain a chain $C$ such that 
	\begin{displaymath}
		l(C) \leq \dist[SC] xy + \dist[SC] yz + \dist[SC] zt -\gro xty - \gro xtz +\sympy{F15aux3_shorteningChain4Points(delta)}.
	\end{displaymath}
	
	\paragraph{Part 2:} Assume now that the maximum in (\ref{eqn: shortening chain - 4 points}) is achieved by $\dist xy + \dist zt$.
	See Figure~\ref{fig: shortening chain 2}.
	It follows that $\gro xyt \leq \gro yzt$.
	We assume that $\gro xty \geq \gro xtz$ (the other case is symmetric).
	We denote by $p$ and $q$ the respective points of $\geo xy$ and $\geo ty$ such that $\dist yp = \dist yq = \gro xty$.
	By hyperbolicity, $\dist pq \leq \sympy{F20aux1_shorteningChain4Points(delta)}$.
	On the other hand $\dist tq = \gro xyt \leq \gro yzt$.
	Consequently, if $r$ is the point of $\geo zt$ such that $\dist tr = \gro xyt$ then $\dist qr \leq \sympy{F20aux2_shorteningChain4Points(delta)}$.
	Thus $\dist pr \leq \sympy{F20aux3_shorteningChain4Points(delta)}$.
	Moreover the triangle inequality leads to $\gro xty \leq \dist zy + \dist zt - \gro xyt$ i.e., $\gro xty \leq \dist yz + \dist zr$.
	According to Lemma~\ref{res: shortening chain - 2 points} there exists a chain $C_-$ (\resp $C_+$) joining $x$ to $p$ (\resp $r$ to $t$) such that $l(C_-) \leq \dist[SC] xy +\sympy{F20aux4_shorteningChain4Points(delta)}$ (\resp $l(C_+) \leq \dist[SC] zt +\sympy{F20aux4_shorteningChain4Points(delta)}$).
	As previously we need to distinguish several cases.
	
	\begin{figure}[ht]
		\subfigure[][First configuration]{
			\includegraphics[width=0.48\linewidth]{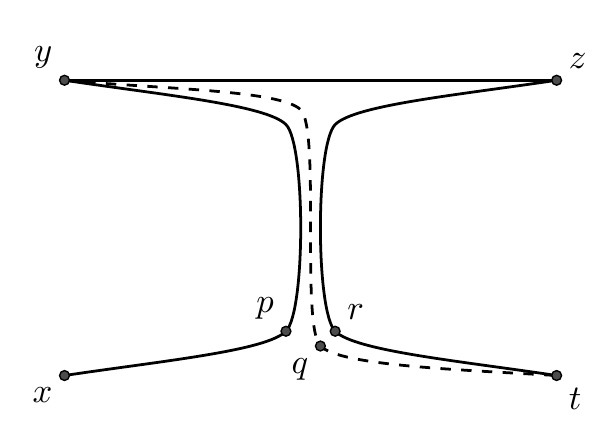}
			\label{fig: shortening chain 2a}
		}
		\subfigure[][Second configuration]{
			\includegraphics[width=0.48\linewidth]{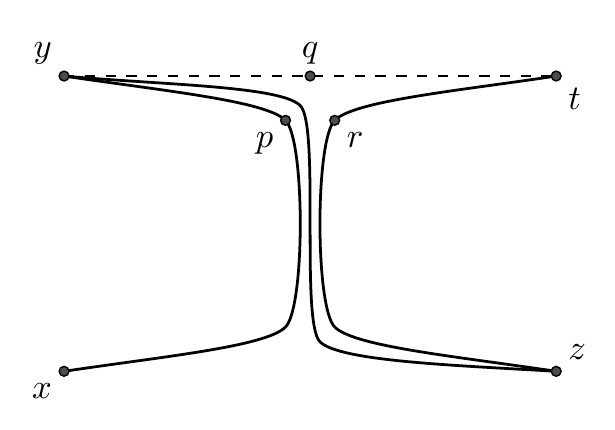}
			\label{fig: shortening chain 2b}
		}
	\caption{Shortening a four points chain - Part 2}
	\label{fig: shortening chain 2}
	\end{figure}

	\paragraph{Case 2.1:} \emph{There exist $i, j \in I$ such that $x,y \in Y_i$ and $z,t \in Y_j$.}
	According to our assumption at the beginning of the proof $i \neq j$.
	In particular $\diaminter{\geo xy}{\geo zt} \leq \diaminter{Y_i}{Y_j} \leq \Delta(Y)$, thus $\gro xty \leq \dist yz + \Delta(Y)$.
	It follows that $\mu(\gro xty) \leq \dist[SC] yz +\Delta(Y)$.
	By contatenating $C_-$ and $C_+$ we obtain a chain whose length satisfies
	\begin{displaymath}
		l(C) \leq \dist[SC]xy + \dist[SC] yz + \dist[SC]zt - \mu(\gro xty) + \Delta(Y) + \sympy{F21aux1_shorteningChain4Points(delta)}
	\end{displaymath}

	\paragraph{Case 2.2:} \emph{There exists $i \in I$ such that $x,y \in Y_i$.}
	In this case $\dist[SC] zt = \dist zt$, thus $\dist[SC] rt \leq \dist[SC] zt - \dist zr$.
	We obtain $C$ by adding $r$ and $t$ at the end of $C_-$.
	This new chain satisfies.
	\begin{displaymath}
		l(C) \leq \dist[SC]xy +  \dist[SC]zt - \dist zr +  \sympy{F22aux1_shorteningChain4Points(delta)}
	\end{displaymath}
	However we proved that $\gro xyt \leq \dist yz + \dist zr$.
	In particular $\mu(\gro xyt)  \leq \dist[SC] yz + \dist zr$.
	Consequently
	\begin{displaymath}
		l(C) \leq \dist[SC]xy + \dist[SC] yz + \dist[SC]zt - \mu(\gro xty) + \sympy{F22aux1_shorteningChain4Points(delta)}
	\end{displaymath}
	
	\paragraph{Case 2.3:} \emph{This is the last case of Part 2.}
	In particular $\dist[SC]xy = \dist xy$.
	It follows that $\dist[SC]xp \leq \dist[SC]xy - \dist yp$ i.e., $\dist[SC]xp \leq \dist[SC]xy -\gro xty$.
	We obtain $C$ by adding $x$ and $p$ at the beginning of $C_+$.
	It satisfies 
	\begin{displaymath}
		l(C) \leq \dist[SC]xy +\dist[SC]zt  - \gro xyt + \sympy{F23aux1_shorteningChain4Points(delta)}
	\end{displaymath}
\end{proof}

\begin{prop}
\label{res: chain in neighbourhood of geodesic}
	Let $\epsilon >0$.
	There exist positive numbers $\delta_0$, $\Delta_0$, $r_1$ and $\eta$ which only depend on $\epsilon$ with the following property.
	Assume that $r_0 \geq r_1$, $\delta \leq \delta_0$ and $\Delta(Y) \leq \Delta_0$.
	Let $x,x' \in X$.
	Let $C$ be a chain of points of $X$ joining $x$ to $x'$.
	If $l(C) \leq \dist[\dot X] x{x'} + \eta$, then every point of $C$ is contained in the $\epsilon$-neighbourhood of $\geo x{x'}$.
\end{prop}

\begin{proof}
	We start by defining the constants $\delta_0$, $\Delta_0$, $r_1$ and $\eta$.
 	Given $r_0$ the function $\mu$ defined in Section~\ref{sec:cone-off over a metric space} satisfies
	\begin{equation*}
		\forall t \in \R_+, \quad \mu(t) \geq t - \frac 1{24}\left(1 + \frac 1 {\sinh^2r_0} \right)t^3 
	\end{equation*}
	Thus there exist $r_1\geq 0$ and  $t_0>0$ with the following property.
	If $r_0 \geq r_1$ then for every $t \in \intval 0{t_0}$, $\mu(t) \geq t/2$.
	We now fix $r_0 \geq r_1$.
	Since $\mu$ is increasing, for every $t \in \R_+$ if $\mu(t) <\mu(t_0)$ then $t \leq 2\mu(t)$.
	Let us put $l=\sympy{Fauxl_chainCloseToGeodesic()}$.
	The numbers $L$ and $d$ are given by the stability of discrete quasi-geodesics (Corollary~\ref{stability discrete quasi-geodesics}).
	Without loss of generality, we can assume that $L > l$.
	We choose $\delta_0>0$, $\Delta_0>0$ and $\eta>0$ such that
	\begin{enumerate}
		\item $\sympy{FauxCondition1_chainCloseToGeodesic(Delta0,L,l,eta,delta0)} < \mu(t_0)$,
		\item $\sympy{FauxCondition2_chainCloseToGeodesic(Delta0,L,l,eta,delta0)} \leq l\delta_0$ 
		\item $\sympy{FauxCondition3_chainCloseToGeodesic(d,L,l,delta0)} \leq \epsilon$ 
	\end{enumerate}
	From now on we assume that $\delta \leq \delta_0$ and $\Delta(Y) \leq \Delta_0$.
	In particular $X$ is $\delta_0$-hyperbolic.
	Let $x, x' \in X$ and $C=(z_0, \dots, z_n)$ be a chain of points of $X$ joining $x$ to $x'$ such that $l(C) \leq \dist[\dot X] x{x'} + \eta$.
	Note that for every $i \leq j$, the length of the subchain $(z_i, z_{i+1}, \dots ,z_{j-1},z_j)$ is at most $\dist[\dot X]{z_j}{z_i} + \eta$.
	
	\paragraph{}We now extracts a subchain of $C$. 
	To that end we proceed in two steps. 
	First we define a subchain $C_1= (z_{i_0},\dots z_{i_m})$ of $C$ as explained in \cite[Section 3.2]{Coulon:il}. 
	\begin{itemize}
		\item Put $i_0 = 0$.
		\item Assume that $i_k$ is defined.
		If $\dist{z_{i_k+1}}{z_{i_k}} > \sympy{FauxExtract1_chainCloseToGeodesic(L,l,delta0)}$ then $i_{k+1} = i_k +1$, otherwise $i_{k+1}$ is the largest integer $i \in \intvald {i_k +1}n$ such that $\dist{z_i}{z_{i_k}} \leq \sympy{FauxExtract1_chainCloseToGeodesic(L,l,delta0)}$.
	\end{itemize}
	By construction, for all $k \in \intvald 0{m-2}$ either $\dist{z_{i_{k+2}}}{z_{i_{k+1}}}> \sympy{FauxExtract2_chainCloseToGeodesic(L,l,delta0)}$ or $\dist{z_{i_{k+1}}}{z_{i_k}}> \sympy{FauxExtract2_chainCloseToGeodesic(L,l,delta0)}$.
	Moreover every point of $C$ is $\sympy{FauxExtract1_chainCloseToGeodesic(L,l,delta0)}$-close to a point of $C_1$.
	
	\paragraph{Claim 1.} For every $k,k' \in \intvald 0m$ the length of the subchain $(z_{i_k},\dots, z_{i_{k'}})$ of $C_1$ is bounded above by $\dist[\dot X]{z_{i_k}}{z_{i_{k'}}} + \sympy{FauxErrorChain_chainCloseToGeodesic(kBis,k,L,l,eta,delta0)}$. (See \cite[Lemma 3.2.3]{Coulon:il}).
	
	\paragraph{}We now build the chain $C_2 = (x_0, y_0, x_1, y_1,\dots y_{p-1},x_p)$ as follows.
	\begin{itemize}
		\item Put $x_0 = z_{i_0}$.
		\item Assume that $x_j = z_{i_k}$ is already defined.
		If $\dist{z_{i_{k+1}}}{z_{i_k}}> \sympy{FauxExtract2_chainCloseToGeodesic(L,l,delta0)}$ we put $y_j = x_j$ and $x_{j+1} = z_{i_{k+1}}$, otherwise we chose $y_j = z_{i_{k+1}}$ and $x_{j+1} = z_{i_{k+2}}$.
		(If $z_{i_{k+1}}$ is already the last point of $C_1$ i.e., if $k+1 = m$ we chose $x_{j+1} = z_{i_{k+1}}$.)
	\end{itemize}
	In this way for all $j \in \intvald 0{p-2}$, $\dist{x_{j+1}}{y_j} > \sympy{FauxExtract2_chainCloseToGeodesic(L,l,delta0)}$.
	Moreover, every point of $C$ is $\sympy{FauxXClose_chainCloseToGeodesic(L,l,delta0)}$-close to a point of $\{x_0, x_1,\dots, x_p\}$.
	
	\paragraph{Claim 2.} For all $j \in \intvald 0{p-1}$, we have $\gro {x_j}{x_{j+1}}{y_j} \leq l\delta_0$.
	Let $j \in \intvald 0{p-1}$.
	According to Claim 1, we have
	\begin{displaymath}
		\dist[SC] {x_j}{y_j} + \dist[SC]{y_j}{x_{j+1}} \leq \dist[\dot X]{x_{j+1}}{x_j} +  \sympy{FClaim2aux1_chainCloseToGeodesic(L,l,eta,delta0)}.
	\end{displaymath}
	On the other hand applying Proposition~\ref{res: shortening chain - 4 points} with the points $x_j$, $y_j$ ,$y_j$ and $x_{j+1}$ we obtain a chain joining $x_j$ to $x_{j+1}$ whose length is at most $\dist[SC] {x_j}{y_j} + \dist[SC]{y_j}{x_{j+1}} - \mu(\gro{x_j}{x_{j+1}}{y_j}) + 2\Delta(Y) + \sympy{F_shorteningChain4Points(delta)}$.
	Hence
	\begin{displaymath}
		\mu(\gro{x_j}{x_{j+1}}{y_j}) \leq \sympy{FClaim2aux2_chainCloseToGeodesic(Delta0,L,l,eta,delta0)} < \mu(t_0)
	\end{displaymath}
	It follows from the definitions of $t_0$, $\delta_0$, $\Delta_0$ and $\eta$ that
	\begin{displaymath}
		\gro{x_j}{x_{j+1}}{y_j} \leq \sympy{FClaim2aux3_chainCloseToGeodesic(Delta0,L,l,eta,delta0)}  \leq l \delta_0.
	\end{displaymath}
	
	\paragraph{Claim 3.} For all $j \in \intvald 0{p-2}$, we have $\gro {x_j}{x_{j+2}}{x_{j+1}} \leq l\delta_0$.
	Let $j \in \intvald 0{p-2}$.
	Applying to Claim 1, we have
	\begin{displaymath}
		\dist[SC]{y_j}{x_{j+1}} +\dist[SC]{x_{j+1}}{y_{j+1}} +\dist[SC]{y_{j+1}}{x_{j+2}} \leq \dist[\dot X]{x_{j+2}}{y_j} +  \sympy{FClaim3aux1_chainCloseToGeodesic(L,l,eta,delta0)}.
	\end{displaymath}
	On the other hand according to Proposition~\ref{res: shortening chain - 4 points} applied to the points $y_j$, $x_{j+1}$, $y_{j+1}$ and $x_{j+2}$ there exists a chain joining $y_j$ to $x_{j+2}$ whose length is at most $\dist[SC]{y_j}{x_{j+1}} +\dist[SC]{x_{j+1}}{y_{j+1}} +\dist[SC]{y_{j+1}}{x_{j+2}} - \mu(\gro{y_j}{x_{j+2}}{x_{j+1}}) + 2\Delta(Y) +\sympy{F_shorteningChain4Points(delta)}$.
	Using the same argument as in Claim 2, we obtain that 
	\begin{displaymath}
		\gro{y_j}{x_{j+2}}{x_{j+1}} \leq \sympy{FClaim3aux3_chainCloseToGeodesic(Delta0,L,l,eta,delta0)} \leq (l-1) \delta_0.
	\end{displaymath}
	By hyperbolicity we get
	\begin{displaymath}
		\min\left\{ \gro{y_j}{x_j}{x_{j+1}}, \gro{x_j}{x_{j+2}}{x_{j+1}}\right\} \leq \gro{y_j}{x_{j+2}}{x_{j+1}} + \sympy{FClaim3aux4_chainCloseToGeodesic(delta0)} \leq l\delta_0	
	\end{displaymath}
	However using Claim 2,
	\begin{displaymath}
		\gro{y_j}{x_j}{x_{j+1}} = \dist{x_{j+1}}{y_j} - \gro {x_j}{x_{j+1}}{y_j}  > \sympy{FauxExtract2_chainCloseToGeodesic(L,l,delta0)} -l\delta_0>  l\delta_0.
	\end{displaymath}
	Consequently $\gro {x_j}{x_{j+2}}{x_{j+1}} \leq l\delta_0$.

	\paragraph{Claim 4.} For all $j \in \intvald 0{p-2}$ we have $\dist{x_{j+1}}{x_j} > L\delta_0$.
	The triangle inequality combined with Claim 2 gives 
	\begin{displaymath}
		\dist{x_{j+1}}{x_j} \geq \dist{x_{j+1}}{y_j} - \gro {x_j}{x_{j+1}}{y_j}  > \sympy{FauxExtract2_chainCloseToGeodesic(L,l,delta0)} -l\delta_0.
	\end{displaymath}
	
	Claims 3 and 4 exactly say that $x_0, x_1,\dots, x_p$ satisfies the assumptions of the stability of discrete quasi-geodesics (Proposition~\ref{stability discrete quasi-geodesics}).
	Therefore for every $j \in \intvald 0p$, $x_j$ lies in the $d\delta_0$-neigbourhood of $\geo {x_0}{x_p}$ i.e., $\geo x{x'}$.
	Nevertheless we noticed that every point of $C$ is $\sympy{FauxXClose_chainCloseToGeodesic(L,l,delta0)}$-close to some $x_j$.
	Thus the distance between any point of $C$ and $\geo x{x'}$ is a most $\sympy{FDistaux_chainCloseToGeodesic(d,L,l,delta0)} \leq \epsilon$. 
\end{proof}

\subsection{Paths in a cone-off}
\label{sec:path in cone-off}

In this section, $X$ is still a geodesic, $\delta$-hyperbolic space and $Y=\left(Y_i\right)_{i \in I}$ a family of strongly quasi-convex subsets of $X$. 
We denote by $\dot X$ the cone-off $\dot X(Y,r_0)$.

\begin{lemm}
\label{path closed to the norm}
	Let $x$ and $x'$ be two points of $X$.
	For all $\eta >0$, there exists a path $\sigma : J \rightarrow \dot X$ between them whose length $L(\sigma)$ is smaller than $\dist[SC] x{x'} + \eta$ and for all $t \in J$, if $ \sigma(t)$ is not the apex of a cone $Z_i$ then $p\circ \sigma(t)$ belongs to the $\sympy{F_pathCloseToNorm(delta)}$-neighbourhood of $\geo x{x'}$.
\end{lemm}

\begin{proof}
	If $\dist[SC]x{x'} = \dist[X]x{x'}$ the geodesic of $X$ joining $x$ to $x'$ works.
	Therefore we can assume that $\dist[SC]x{x'} \neq \dist[X]x{x'}$.
	Let $\epsilon >0$.
	By definition of $\distV[SC]$, there exists $i \in I$ such that $x, x' \in Y_i$ and $\dist[Z_i]x{x'} < \dist[SC]x{x'} + \epsilon$.
	We distinguish two cases.
	\paragraph{Case 1:}
	If $\dist[Y_i]x{x'} \geq \pi \sinh r_0$, then $\dist[Z_i]x{x'}=2r_0$.
	We chose for $\sigma : J \rightarrow Z_i$ the geodesic of $Z_i$ $\geo x{v_i}\cup \geo{v_i}{x'}$.
	(Recall that $v_i$ is the apex of the cone $Z_i$.)
	Its length (as a path of $Z_i$) is $2r_0$.
	Moreover for all $t \in J$, if $\sigma(t) \neq v_i$, then $p \circ \sigma(t) \in \left\{ x, x'\right\}$.
	\paragraph{Case 2:}
	If $\dist[Y_i]x{x'} < \pi \sinh r_0$.
	The space $\left(Y_i, \distV[Y_i]\right)$ is a length space.
	Thus there exists a path $\sigma_Y : J \rightarrow Y_i$ parametrized by arc length between $x$ and $x'$ whose length is less than $\min\{ \dist[Y_i]x{x'} + \epsilon,\pi \sinh r_0\}$.
	Hence there exists a path $\sigma : J \rightarrow Z_i \setminus\{ v_i\}$ between $x$ and $x'$ such that $p_i \circ \sigma = \sigma_Y$ and its length $L(\sigma)$ (as a path of $Z_i$) satisfies
	\begin{displaymath}
		L(\sigma) 
		\leq \mu \left( L\left(\sigma_Y\right)\right) 
		\leq \mu\left(\dist[Y_i]x{x'} + \epsilon\right)
		\leq \dist[SC]x{x'} + 2\epsilon
	\end{displaymath}
	However $Y_i$ is strongly quasi-convex.
	It follows that for all $y,y' \in Y_i$, $\dist[X]y{y'} \leq \dist[Y_i]y{y'} \leq \dist[X] y{y'} + \sympy{Faux1_pathCloseToNorm(delta)}$.
	Consequently, as a path of $X$, $\sigma_Y$ is a $(1,\sympy{Faux2_pathCloseToNorm(epsilon,delta)})$-quasi-geodesic.
	In particular $\sigma_Y(J)$ lies in the $\left(\sympy{Faux3_pathCloseToNorm(epsilon,delta)}\right)$-neighbourhood of $\geo x{x'}$.
	
	\paragraph{}
	Hence we have build a path $\sigma : J \rightarrow Z_i$, whose length (as a path of $Z_i$) is smaller than $\dist[SC]x{x'} + 2\epsilon$ and such that for all $t \in J$, if $\sigma(t) \neq v_i$, $p\circ \sigma(t)$ belongs to the $\left(\sympy{Faux3_pathCloseToNorm(epsilon,delta)}\right)$-neighbourhood of $\geo x{x'}$.
	However the map $Z_i \rightarrow \dot X$ is 1-lipschitz.
	It follows that the length of $\sigma$ as a path of $\dot X$ is also smaller than $\dist[SC]x{x'} + 2\epsilon$.
	By choosing $\epsilon$ small enough we obtain the announced result.
\end{proof}

\begin{prop}
\label{projection of quasi-geodesic}
	Let $\epsilon >0$.
	There exist positive constants $\delta_0$, $\Delta_0$ and $r_1$ which only depend on $\epsilon$ having the following property.
	Assume that $r_0 \geq r_1$, $\delta \leq \delta_0$ and $\Delta(Y) \leq \Delta_0$.
	Let $x$ and $x'$ be two points of $X \subset \dot X(Y,r_0)$.
	For all $\eta >0$, there exists a $(1,\eta)$-quasi-geodesic $\sigma : J \rightarrow \dot X$ joining  $x$ and $x'$ such that for all $t \in J$, if $\sigma(t)$ is not an apex of $\dot X$, $p\circ\sigma(t)$ belongs to the $\epsilon$-neighbourhood of $\geo x{x'}$.
\end{prop}

\begin{proof}
	By Proposition~\ref{res: chain in neighbourhood of geodesic}, there exist positive constants $\delta_0$, $\Delta_0$, $r_1$ and $\eta_0$ which only depend on $\epsilon$ satisfying the following property.
	Assume that $r_0 \geq r_1$, $\delta \leq \delta_0$ and $\Delta(Y) \leq \Delta_0$.
	Let $x$ and $x'$ be two points of $X$ and $C$ a chain of $X$ between them.
	If $l(C) \leq \dist[\dot X]x{x'} + \eta_0$, then every point of $C$ belongs to the $\epsilon/2$-neigbourhood of $\geo[X]x{x'}$.
	By replacing $\delta_0$ by a smaller constant if necessary, we may also assume that $\sympy{Faux2_projQuasiGeodesic(delta0)} \leq \epsilon/2$.
	
	\paragraph{} Consider now $\eta \in (0, \eta_0)$ and $x$ and $x'$ two points of $X$.
	By definition of $\dist[\dot X]x{x'}$, there exists a chain $C= \left(z_0, \dots,z_m\right)$ of $X$ between $x$ and $x'$ such that $l(C) \leq \dist[\dot X]x{x'} +  \eta/2$.
	By Proposition~\ref{res: chain in neighbourhood of geodesic}, every $z_j$ belongs to the $\epsilon/2$-neighbourhood of $\geo x{x'}$.
	Let $k \in \intvald 0{m-1}$.
	Applying Lemma~\ref{path closed to the norm}, there exists a rectifiable path $\sigma_k : J_k \rightarrow \dot X$ joining $z_k$ and $z_{k+1}$ whose length is smaller than $\dist[SC]{z_k}{z_{k+1}} +  \eta/2m$ and such that for all $t \in J_k$, if $\sigma_k(t)$ is not an apex of $\dot X$, $p \circ \sigma_k(t)$ belongs to the $\sympy{Faux1_projQuasiGeodesic(delta)}$-neighbourhood of $\geo{z_k}{z_{k+1}}$.
	In particular the distance of  $p \circ \sigma_k(t)$ to $\geo x{x'}$ is less than $ \epsilon/2 + \sympy{Faux2_projQuasiGeodesic(delta)} \leq \epsilon$.
	We now choose for $\sigma$ the concatenation of the $\sigma_k$'s.
	Its length is smaller than $l(C) +  \eta/2 \leq \dist[\dot X]x{x'} + \eta$.
	We reparametrize $\sigma$ by arc length, hence $\sigma$ is a $(1, \eta)$-quasi-geodesic.
	Moreover it satisfies the announced property.
\end{proof}

\begin{prop}
\label{res: comparing Gromov product base vs cone-off}
	There exist positive constants $\delta_0$, $\delta_1$, $\Delta_0$ and $r_1$ which do not depend on $X$ or $Y$ having the following property.
	Assume that $r_0 \geq r_1$, $\delta \leq \delta_0$ and $\Delta(Y) \leq \Delta_0$.
	For every $x,y,z \in X$ we have
	\begin{displaymath}
		\mu\left(\gro yzx\right) 
		\leq \frac 12\left(\fantomA \dist[\dot X] yx + \dist[\dot X]zx - \dist[\dot X]yz \right) +\sympy{F_compareGormovProduct(r0,delta1)}.
	\end{displaymath}	
\end{prop}

\begin{proof}
	The constant $\delta_1$, $\delta_0$, $\Delta_0$ and $r_1$ are given by Proposition~\ref{cone-off globally hyperbolic}.
	We fix $\epsilon_1$ such that $\mu(\epsilon_1) = \delta_1$.
	According to Proposition~\ref{projection of quasi-geodesic}, by decreasing (\resp increasing) if necessary $\delta_0$, $\Delta_0$ (\resp $r_1$) the following hold.
	Assume that $r_0 \geq r_1$, $\delta \leq \delta_0$ and $\Delta(Y) \leq \Delta_0$ then
	\begin{enumerate}
		\item \label{enu: comparing Gromov product base vs cone-off - prem 1}
		$\dot X$ is $\delta_1$-hyperbolic,
		\item \label{enu: comparing Gromov product base vs cone-off - prem 2}
		for every $x,x' \in X$, for every $\eta>0$ there is a $(1,\eta)$-quasi-geodesic $\sigma : J \rightarrow \dot X$ joining  $x$ and $x'$ such that for all $t \in J$, if $\sigma(t)$ is not an apex of $\dot X$, $p\circ\sigma(t)$ belongs to the $\epsilon_1$-neighbourhood of $\geo x{x'}$.
	\end{enumerate}
	Let $x$, $y$ and $z$ be three points of $X \subset \dot X$. 
	In all this section we kept the notation $\gro xyz$ for the Gromov product computed with the distance of $X$.
	Exceptionally we will denote the Gromov product of these three points computed in $\dot X$ by 
	\begin{displaymath}
		\gro[\dot X] xyz = \frac 12\left(\fantomA \dist[\dot X] zx + \dist[\dot X]zy - \dist[\dot X]xz \right)
	\end{displaymath}
	Let $\eta>0$.
	There exists a $(1, \eta)$-quasi-geodesic $\gamma: \intval 0a \rightarrow \dot X$ joining $y$ to $z$ and satisfying \ref{enu: comparing Gromov product base vs cone-off - prem 2}.
	Let us put
	\begin{displaymath}
		t = \min\left\{ \gro[\dot X]xzy, a -  \gro[\dot X]xyz\right\}
	\end{displaymath}
	Note that the definition of $\gamma(t)$ is symmetric in $y$ and $z$: using the reverse parametrization for the quasi-geodesic $\gamma$ would lead to the same point.
	The point $\gamma(t)$ is not necessary in $X$. 
	However the diameter of the cones that were attached to buid $\dot X$ is at most $2r_0$.
	The path $\gamma$ being a continuous $(1,\eta)$-quasi-geodesic, there exists $s \in \intval 0a$ such that $\dist st \leq \sympy{Faux1_compareGormovProduct(r0,eta)}$ and $\gamma(s) \in X$.
	The points $y$ and $z$ playing a symmetric role, we can assume without loss of generality that $s \leq t$.
	
	\paragraph{}We consider now a $(1, \eta)$-quasi-geodesic $\sigma: \intval 0b \rightarrow \dot X$  joining $y$ to $x$, satisfying \ref{enu: comparing Gromov product base vs cone-off - prem 2} and put $r = \min\{s,b\}$.
	Since $\sigma$ is $(1,\eta)$-quasi-geodesic we have $\dist[\dot X] x{\sigma(r)}  \leq \dist[\dot X]xy - r +\sympy{Faux2_compareGormovProduct(eta)}$ which leads to 
	\begin{equation}
	\label{eqn: comparing Gromov product base vs cone-off - eqn 1}
		\dist[\dot X] x{\sigma(r)} 
		\leq\gro[\dot X]yzx + \sympy{Faux3_compareGormovProduct(r0,eta)}
	\end{equation}
	Moreover by hyperbolicity of $\dot X$, $\dist[\dot X]{\sigma(r)}{\gamma(r)} \leq \sympy{Faux4_compareGormovProduct(delta1,eta)}$.
	In particular $\sigma(r)$ belongs to the $(\sympy{Faux4_compareGormovProduct(delta1,eta)})$-neighbourhood of $X$ in $\dot X$.
	Hence $\dist[\dot X]{p\circ\sigma(r)}{\gamma(r)}\leq \sympy{Faux5_compareGormovProduct(delta1,eta)}$.
	It follows that $\dist {p\circ\sigma(r)}{\gamma(r)}\leq \epsilon$, where $\mu(\epsilon) = \sympy{Faux5_compareGormovProduct(delta1,eta)}$.
	Nevertheless $p\circ\sigma(r)$ and $\gamma(r)$ respectively lie in the $\epsilon_1$-neighbourhood of $\geo yx$ and $\geo yz$.
	By triangle inequality
	\begin{displaymath}
		\dist{p\circ\sigma(r)}y \leq \gro xzy +\gro xy{p\circ\sigma(r)} + \gro yz{\gamma(r)} +\dist {p\circ\sigma(r)}{\gamma(r)}.
	\end{displaymath}
	Consequently $\dist{p\circ\sigma(r)}y \leq \gro xzy + \sympy{Faux6_compareGormovProduct(epsilon1,epsilon)}$, and
	\begin{displaymath}
	 	\gro yzx
		\leq \dist xy - \dist y{p\circ\sigma(r)} +  \sympy{Faux6_compareGormovProduct(epsilon1,epsilon)}
		\leq \dist x{p\circ\sigma(r)} +  \sympy{Faux6_compareGormovProduct(epsilon1,epsilon)}
	\end{displaymath}
	Applying $\mu$ to this inequality we get $\mu(\gro yzx) \leq \dist[\dot X] x{p\circ\sigma(r)} + \sympy{Faux7_compareGormovProduct(delta1,eta)}$ which combined with (\ref{eqn: comparing Gromov product base vs cone-off - eqn 1}) gives 
	\begin{displaymath}
		\mu\left(\gro yzx\right) 
		\leq \gro[\dot X]yzx +\sympy{Faux8_compareGormovProduct(r0,delta1,eta)}.
	\end{displaymath}
	This inequality holds for every $\eta>0$ which completes the proof.
\end{proof}

%% file: 3_small_cancellation.tex
\section{Small cancellation theory}
\label{sec:small cancellation}

\paragraph{} In this section we will be concerned with the small cancellation theory. We expose the geometrical point of view developed by T.~Delzant  and M.~Gromov in \cite{DelGro08} and used in Section~\ref{sec:Burnside groups} to prove the main theorem.

\subsection{General framework}
\label{sec:small cancellation constants}
\paragraph{} We require $X$ to be a proper, geodesic, $\delta$-hyperbolic space and $G$ a group acting properly, co-compactly, by isometries on $X$.
We assume that $G$ satisfies the small centralizers hypothesis (see Section~\ref{sec: hyperbolic groups}).

\paragraph{} 
Let $P$ be a set of hyperbolic elements of $G$.
We assume that $P$ is the union of a finite number of conjugacy classes.
We denote by $K$ the (normal) subgroup of $G$ generated by $P$. 
Our goal is to study the quotient $\bar G = G/K$.
The small cancellation parameters $\Delta(P,X)$ and $\rinj P X$ (see Definition~\ref{defi:overlap and injectivity radius}), respectively play the role of the length of the largest piece and the length of the smallest relation in the usual small cancellation theory.
We are interested in situations where the ratios $\delta/\rinj P X$ and $\Delta(P,X)/\rinj P X$ are very small.
To that end, we build a space $\bar X$ with an action of $\bar G$.
We only recall the main steps of this construction. 
This approach has been studied in  \cite{DelGro08}, \cite{Dahmani:2011vu} and \cite{Cou10a}.
We follow here \cite{Coulon:2013tx}.

\paragraph{} Fix $r_0>0$. 
Its value will be made precise in Theorem \ref{theo:small cancellation theorem}.
We consider the family of strongly quasi-convex subsets $Y = \left(Y_\rho\right)_{\rho \in P}$.
The cone-off of radius $r_0$ over $X$ relatively to $Y$ is denoted by $\dot X$.
We extend by homogeneity the action of $G$ on $X$ in an action of $G$ on $\dot X$.
Given a point $x=(y,r)$ of $C_\rho$ and $g$ an element of $G$, $gx$ is the point of $C_{g\rho g^{-1}}=gC_\rho$ defined by $gx=(gy,r)$.
The group $G$ acts by isometries on $\dot X$ (see \cite[Lemma 4.3.1]{Coulon:il}).
The space $\bar X$ is the quotient of $\dot X$ by $K$.

\begin{theo}[Small cancellation theorem, see {\cite[Th.~5.5.2]{DelGro08}} or {\cite[Prop. 6.7]{Coulon:2013tx}}]
\label{theo:small cancellation theorem}
	There exist positive numbers $\delta_0$, $\delta_1$, $\Delta_0$ and $r_1$ which do not depend on $X$ or $P$ with the following property.
	If $r_0 \geq r_1$, $\delta\leq \delta_0$, $\Delta(P,X) \leq \Delta_0$ and $\rinj P X\geq \pi \sinh r_0$, then $\bar X$ is proper, geodesic and $\bar \delta$-hyperbolic, with $\bar \delta \leq \delta_1$.
	Moreover $\bar G$ acts properly, co-compactly, by isometries on it.
\end{theo}

\paragraph{} 
Note that the constants $\delta_0$, $\delta_1$, $\Delta_0$ and $r_1$ in Theorem~\ref{theo:small cancellation theorem} are a priori different from the ones of Theorem~\ref{cone-off globally hyperbolic} or Propositions~\ref{projection of quasi-geodesic} and \ref{res: comparing Gromov product base vs cone-off}.
However by decreasing (\resp increasing) if necessary $\delta_0$, $\Delta_0$ (\resp $\delta_1$, $r_1$) we can always assume that they work for the three results.
Similarly we can require that $r_1 \geq 10^{100}\delta_1$ and $\delta_0, \Delta_0<10^{-5}\delta_1$.
We now fix them once for all.
By Proposition~\ref{stability quasi-geodesics}, we can find constants $r_0 \geq r_1$ and $k_S \geq 1$ having the following property.
Let $\eta \in (0, \delta_1)$.
If $\sigma$ is a $\sympy{F_SCQGLdot()}$-local $(1,\eta)$-quasi-geodesic in a $\delta_1$-hyperbolic space then it is a $(k_S,\eta)$-quasi-geodesic and lies in the $\sympy{F_SCQGDdot()}$-neighbourhood of every geodesic joining its endpoints.
Using Theorems~\ref{cone-off globally hyperbolic} and \ref{theo:small cancellation theorem}, Propositions~\ref{projection of quasi-geodesic} and \ref{res: comparing Gromov product base vs cone-off} we obtain that if $\delta \leq \delta_0$, $\Delta(P,X) \leq \Delta_0$ and $\rinj PX \geq \sympy{F_SCrinj()}$, then the followings hold.
\begin{enumerate}
	\item (Theorem~\ref{cone-off globally hyperbolic}) The cone-off $\dot X$ is $\delta_1$-hyperbolic.
	\item (Theorem~\ref{theo:small cancellation theorem})
	The space $\bar X$ is proper, geodesic and $\bar \delta$-hyperbolic, with $\bar \delta \leq \delta_1$.
	Moreover  $\bar G$ acts properly, co-compactly, by isometries on it.
	\item (Proposition~\ref{res: comparing Gromov product base vs cone-off})
	For all $x,y,z \in X$, 
	\begin{displaymath}
		\mu\left(\gro yzx\right) \leq \frac 12\left(\fantomB\dist[\dot X]yx + \dist[\dot X]zx - \dist[\dot X]yz\right) + \sympy{F_SCgromovBaseVsDot()}.
	\end{displaymath}
	\item (Proposition~\ref{projection of quasi-geodesic}) For all $x, x' \in X$, for all $\eta >0$, there exists a $(1,\eta)$-quasi-geodesic $\sigma : J \rightarrow \dot X$ between $x$ and $x'$ such that for all $t \in J$, if $\sigma(t)$ is not an apex of $\dot X$, then $p\circ\sigma(t)$ lies in the $\sympy{F_SCprojGeoBase()}$-neighbourhood of  $\geo x{x'}$.
\end{enumerate}

\rem The parameters $\delta_0$, $\Delta_0$, $\delta_1$ and $r_0$ are certainly not chosen in an optimal way.
What only matters is their orders of magnitude recalled below.
\begin{displaymath}
	\max\left\{\delta_0, \Delta_0\right\} \ll \delta _1 \ll r_0 \ll \pi \sinh r_0.
\end{displaymath}
An other important point to remember is the following.
The constants $\delta_0$, $\Delta_0$ and $\pi \sinh r_0$ are used to describe the geometry of $X$ whereas $\delta_1$ and $r_0$ refers to the one of $\dot X$ or $\bar X$.

\notas 
\begin{itemize}
	\item Given $g$ is an element of $G$ we write $\bar g$ for the image of $g$ by the canonical projection $\pi : G \twoheadrightarrow \bar G$.
	\item We will denote by $\bar x$ the image of a point $x$ of $X$ by the natural map $\nu : X \rightarrow \dot X \rightarrow \bar X$.
	\item Unless otherwise stated all distances, diameters, Gromov's products, etc will be compute with the distance of $X$ or $\bar X$ (but not of $\dot X$).
\end{itemize}

\subsection{A Greendlinger Lemma}

\begin{lemm}[see {\cite[Prop. 5.6.1]{DelGro08} or {\cite[Prop. 3.15]{Coulon:2013tx}}}]
\label{developing map}
	Let $x$ be a point of $\dot X$ such that $d(x,X) \leq \frac {r_0}2$.
	The map $\dot X \rightarrow \bar X$ induces an isometry from $B\left(x,\sympy{F_SCdeveloppingMap()}\right)$ onto its image.
\end{lemm}

\begin{prop}
\label{existence good quasi-geodesic in cone-off and quotient}
	Let $x$ and $x'$ be two points of $X$.
	We assume that for all $\rho \in P$, $\diaminter{\geo x{x'}}{Y_\rho} \leq \len \rho  \sympy{-1*F_SCpreGreendlingerOverlap()}$.
	Then for all $\eta >0 $ there exists a $(1, \eta)$-quasi-geodesic $\sigma : J \rightarrow \dot X$ between $x$ and $x'$, such that the path $\bar \sigma : J \rightarrow \dot X \rightarrow \bar X$ is a $\sympy{F_SCpreGreendlingerL()}$-local $(1, \eta)$-quasi-geodesic of $\bar X$.
\end{prop}

\begin{proof}
	Let $\eta \in \left( 0 , \sympy{Faux1_SCpreGreendlinger()}\right)$.
	Applying Proposition~\ref{projection of quasi-geodesic} there exists a $(1,\eta)$-quasi-geodesic $\sigma : J \rightarrow \dot X$ between $x$ and $x'$ such that for all $t \in J$, if $\sigma(t)$ is not an apex of $\dot X$, then $p\circ\sigma(t)$ lies in the $\sympy{F_SCprojGeoBase()}$-neighbourhood of $\geo x{x'}$.
	Let $s, t \in J$ such that $\dist st \leq \sympy{Faux2_SCpreGreendlinger()}$.
	Since $\sigma$ is a $(1,\eta)$-quasi-geodesic, $\dist[\dot X]{\sigma(s)}{\sigma(t)} \leq \dist st + \eta < \sympy{Faux3_SCpreGreendlinger()}$.
	We now distinguish two cases.
	\begin{itemize}
		\item Assume that $d\left(\sigma(s), X\right) \leq \sympy{Faux4_SCpreGreendlinger()}$.
		By Lemma~\ref{developing map}, the map $\dot X \rightarrow \bar X$ restricted to the ball of center $\sigma(s)$ and radius $\sympy{Faux5_SCpreGreendlinger()}$ preserves the distances.
		Hence $\dist[\bar X]{\bar \sigma(s)}{\bar \sigma(t)} = \dist[\dot X]{\sigma (s)}{\sigma (t)}$. 
	
		\item Assume that $d\left(\sigma(s), X\right) > \sympy{Faux4_SCpreGreendlinger()}$.
		There exists $\rho \in P$ such that $\sigma(s)$ and $\sigma(t)$ are two points of the same cone $C_\rho$.
		If $\sigma(s)$ or $\sigma(t)$ is the apex of the cone then $\dist[\bar X]{\bar \sigma(s)}{\bar \sigma(t)} = \dist[\dot X]{\sigma (s)}{\sigma (t)}$, otherwise $p\circ \sigma(s)$ and $p\circ \sigma(t)$ belong to $Y_\rho$ and the $\sympy{F_SCprojGeoBase()}$-neghbourhood of $\geo x{x'}$.
		Thus
		\begin{displaymath}
			\dist[Y_\rho]{p\circ \sigma(s)}{p\circ \sigma(t)}
			\leq \diaminter{\geo x{x'}}{Y_\rho} + \sympy{Faux6_SCpreGreendlinger()}
			\leq \len[espace=Y_\rho] \rho \sympy{-1*Faux7_SCpreGreendlinger()}
		\end{displaymath}
		It follows from Lemma~\ref{metric quotient cone}, that $\dist[\bar X]{\bar \sigma(s)}{\bar \sigma(t)} = \dist[\dot X]{\sigma (s)}{\sigma (t)}$.
	\end{itemize}
	Thus for all $s,t \in J$, if $\dist st \leq \sympy{Faux2_SCpreGreendlinger()}$, $\dist[\bar X]{\bar \sigma(s)}{\bar \sigma(t)} = \dist[\dot X]{\sigma (s)}{\sigma (t)}$.
	Since $\sigma$ is a $(1,\eta)$-quasi-geodesic, $\bar \sigma$ is a $\sympy{F_SCpreGreendlingerL()}$-local $(1,\eta)$-quasi-geodesic.
\end{proof}

\begin{theo}[Greendlinger's Lemma]
\label{Greendlinger Lemma}
	Let $x$ be a point of $X$.
	Let $g$ be an element of $G\setminus\{1\}$.
	If $g$ belongs to $K$, then there exists $\rho \in P$ such that $\diaminter{\geo x{gx}}{Y_\rho} > \len \rho  \sympy{-1*F_SCGrenndlinger()}$.
\end{theo}

\begin{proof}
	We prove the theorem by contradiction.
	Assume that for all $\rho \in P$, $\diaminter{\geo x{gx}}{Y_\rho} \leq \len \rho  \sympy{-1*F_SCGrenndlinger()}$.
	Let $\eta \in (0, \delta_1)$.
	Applying Proposition~\ref{existence good quasi-geodesic in cone-off and quotient}, there exists a $(1, \eta)$-quasi-geodesic $\sigma : \intval ab \rightarrow \dot X$ between $x$ and $gx$, such that the path $\bar \sigma :  \intval ab \rightarrow \dot X \rightarrow \bar X$ is a $\sympy{Faux1_SCGreendlinger()}$-local $(1, \eta)$-quasi-geodesic of $\bar X$.
	In particular $\bar \sigma$ is a $(k_S, \eta)$-quasi-geodesic (see Proposition~\ref{stability quasi-geodesics}).
	Hence, $\dist[\dot X] {gx}x \leq k_S\dist {\bar g \bar x}{\bar x} + 3\eta = 3\eta$.
	This inequality holds for all $\eta >0$.
	It implies $gx=x$. 
	However $K$ acts freely on $X$ (see \cite[Prop. 5.6.2]{DelGro08}), thus $g=1$.
	Contradiction.
\end{proof}

\begin{prop}[Preserving shape Lemma]
\label{Preserving shape Lemma}
	Let $x$, $y$ and $z$ be three points of $X$ such that for all $\rho \in P$, 
	\begin{displaymath}
		\max\left\{\fantomB \diaminter{\geo xy}{Y_\rho}, \diaminter{\geo x{z}}{Y_\rho}\right\} \leq \len \rho \sympy{-1*F_SCpreservingShapeAxe()}.
	\end{displaymath}
	If $\gro{\bar y}{\bar z}{\bar x} \leq \sympy{F_SCpreservingShapeGro()}$, then $\gro yzx \leq  \sympy{F_SCpreservingShape()}$
\end{prop}

\begin{proof}
	As we wrote before, we keep the notation $\gro yzx$ for the Gromov product computed with the distance of $X$.
	Therefore we denote by $t$ the same product computed with the distance of $\dot X$.
	\begin{displaymath}
		t =  \frac 12\left(\fantomB\dist[\dot X]yx + \dist[\dot X]zx - \dist[\dot X]yz\right) 
	\end{displaymath}
	By Proposition~\ref{res: comparing Gromov product base vs cone-off}, we have $\mu(\gro yzx) \leq t +  \sympy{Faux1_SCpreservingShape()}$.
	The goal is now to compare $t$  and $\gro{\bar y}{\bar z}{\bar x}$.
	We can assume that $\min\{\dist[\dot X] xy,\dist[\dot X]xz\} >  \gro{\bar y}{\bar z}{\bar x} + \sympy{Faux17_SCpreservingShape()}$.
	Otherwise we would have $t \leq  \gro{\bar y}{\bar z}{\bar x} +\sympy{Faux17_SCpreservingShape()}$.

	\paragraph{}Let $\eta \in (0, \delta_1)$ such that $\min\{\dist[\dot X] xy,\dist[\dot X]xz\} > \gro{\bar y}{\bar z}{\bar x} + \sympy{Faux14_SCpreservingShape(eta)}$.
	According to Proposition~\ref{existence good quasi-geodesic in cone-off and quotient}, there exists a $(1,\eta)$-quasi-geodesic $\sigma : \intval 0a \rightarrow \dot X$ between $x$ and $y$ whose image $\bar \sigma : J \rightarrow \dot X \rightarrow \bar X$ in $\bar X$ is a $\sympy{Faux2_SCpreservingShape()}$-local $(1, \eta)$-quasi-geodesic. 
	In particular $\bar \sigma$ lies in the $\sympy{Faux3_SCpreservingShape()}$-neighbourhood of $\geo{\bar x}{\bar y}$.
	We also construct a path $\gamma : \intval 0b \rightarrow \dot X$ between $x$ and $z$ having the same properties.
	Let $ s \in [0, \min\{a,b\}]$ such that $s> \gro{\bar y}{\bar z}{\bar x} + \sympy{Faux13_SCpreservingShape(eta)}$.
	Without loss of generality we can also require that $s \leq \sympy{Faux15_SCpreservingShape()}$.
	Let us denote by $p$ and $q$ the points $\sigma(s)$ and $\gamma(s)$.
	By hyperbolicity of $\dot X$ we have 
	\begin{displaymath}
		\dist[\dot X] pq  \leq \max\left\{\dist{\fantomB\dist[\dot X]xp}{\dist[\dot X]xq} +\sympy{Faux4_SCpreservingShape(eta)}, \dist[\dot X] xp + \dist [\dot X]xq - 2t \right\} +\sympy{Faux5_SCpreservingShape()}
	\end{displaymath}
	which leads to 
	\begin{equation}
	\label{eqn: Preserving shape Lemma - hyp dot X}
		\dist[\dot X] pq  \leq \max\left\{\sympy{Faux6_SCpreservingShape(eta)}, 2s - 2t \right\} +\sympy{Faux7_SCpreservingShape(eta)}
	\end{equation}
	Our next step is to give a lower bound for $\dist[\dot X]pq$.
	Recall that $s \leq \sympy{Faux15_SCpreservingShape()}$.
	Thus $p$ and $q$ are contained in the ball of $\dot X$ of center $x$ and radius $\sympy{Faux8_SCpreservingShape()}$.
	However the map $\dot X \rightarrow \bar X$ restricted to this ball is an isometry, hence $\dist[\dot X] pq = \dist{\bar p}{\bar q}$ and $\dist{\bar x}{\bar p}+ \dist{\bar x}{\bar q} \geq 2s \sympy{-1*Faux9_SCpreservingShape(eta)}$ .
	By triangle inequality that
	\begin{displaymath}
		\dist{\bar p}{\bar q} \geq \dist{\bar x}{\bar p} + \dist{\bar x}{\bar q} - 2 \gro{\bar y}{\bar z}{\bar x} - 2\gro{\bar x}{\bar y}{\bar p} - 2 \gro{\bar x}{\bar z}{\bar q}
	\end{displaymath}
	Since $\bar p$ (\resp $\bar q$) lies in the $\sympy{Faux3_SCpreservingShape()}$-neighbourhood of $\geo {\bar x}{\bar y}$ (\resp $\geo{\bar x}{\bar z}$) we get 
	\begin{displaymath}
		\dist[\dot X]pq = \dist{\bar p}{\bar q} \geq 2s - 2 \gro{\bar y}{\bar z}{\bar x} \sympy{-1*Faux10_SCpreservingShape(eta)} > \sympy{Faux11_SCpreservingShape(eta)}
	\end{displaymath}
	It follows then from (\ref{eqn: Preserving shape Lemma - hyp dot X}) that $t \leq s + \sympy{Faux12_SCpreservingShape(eta)}$.
	This inequality holds for every sufficiently small $\eta$ and for every $s> \gro{\bar y}{\bar z}{\bar x} + \sympy{Faux13_SCpreservingShape(eta)}$ thus $t \leq \gro{\bar y}{\bar z}{\bar x} + \sympy{Faux16_SCpreservingShape()}$.
	
	\paragraph{}
	We proved that $\mu(\gro yzx) \leq \gro{\bar y}{\bar z}{\bar x} + \sympy{Faux18_SCpreservingShape()} < \sympy{Faux19_SCpreservingShape()}$.
	The conclusion follows from the estimate of the function $\mu$ (see Section~\ref{sec:cone}).
\end{proof}

\subsection{$P$-close points}
\label{sec: close points}

\begin{defi}
	Two points $x$ and $x'$ of $X$ are \emph{$P$-close} if for all $\rho \in P$, $\diaminter{\geo x{x'}}{Y_\rho} \leq \len \rho/2 + \sympy{V_Pclose}$.
\end{defi}

\rem There is a very simple way to get $P$-close points.
Let $x$ and $x'$ be two points of $X$. 
Let $u \in K$. 
If $\dist x{ux'} \leq \inf_{v \in K} \dist x{vx'} + \delta$, then $x$ and $ux'$ are $P$-close.
Indeed, if it was not the case, according to Lemma~\ref{shortening geodesic} one could reduce the distance between $x$ and $ux'$.

\begin{prop}
\label{satisfaire Greendlinger geodesique}
	Let $\alpha \geq 0$
	Let $x$ and $x'$ be two $P$-close points of $X$.
	Let $y \in X$ such that for all $u \in K$, $\gro x{x'}y  < \gro x{x'}{uy} + 2\alpha$.
	Then for all $\rho \in P$, $\diaminter{\geo xy}{Y_\rho} \leq \len \rho - c$ where $c =  \sympy{F_SCsatisfaireGreendlinger(Qalpha)}$.
\end{prop}

\begin{proof}
	We prove this result by contradiction.
	Assume that there exists $\rho \in P$ such that $\diaminter{\geo xy}{Y_\rho} > \len \rho - \sympy{Faux0_SCsatisfaireGreendlinger(c)}$. 
	Let $N$ be a nerve of $\rho$.
	We denote by $p$ and $q$ respective projections of $x$ and $y$ on $N$.
	Let $r$ be a projection of $x'$ on $\nerf pqN$.
	Recall that $\len \rho \geq \sympy{F_SCrinj()}$.
	It follows from Proposition~\ref{projection conditionnelle sur un axe version simple}, that
	\begin{enumerate}
		\item \label{enu: greendlinger geodesic - distances}
		\begin{list}{\labelitemi}{\leftmargin=0em}
		\renewcommand{\labelitemi}{}
			\item $\dist pq \geq \len \rho \sympy{-1*Faux1_SCsatisfaireGreendlinger(c)}$,
			\item $\dist qr \geq \len \rho/2 \sympy{-1*Faux2_SCsatisfaireGreendlinger(c)}$,
		\end{list}
		\item \label{enu: greendlinger geodesic - gromov product}
		$\gro x{x'}y \geq \gro x{x'}r +   \dist yq  + \dist qr \sympy{-1*Faux3_SCsatisfaireGreendlinger()}$.
	\end{enumerate}
	The isometry $\rho$ acts on $N$ by translation of length $\len \rho$. 
	Therefore there exists $\epsilon \in \{\pm1\}$, such that $p$ and $\rho^\epsilon q$ belong to the same component of $N\setminus\{q\}$.
	We want to compare $\gro x{x'}y$ and $\gro x{x'}{\rho^\epsilon y}$.
	To that end, we distinguish two cases depending on the relative positions of $p$, $r$, and $\rho^\epsilon q$ on $N$.

	\begin{figure}[ht]
	\centering
		\includegraphics{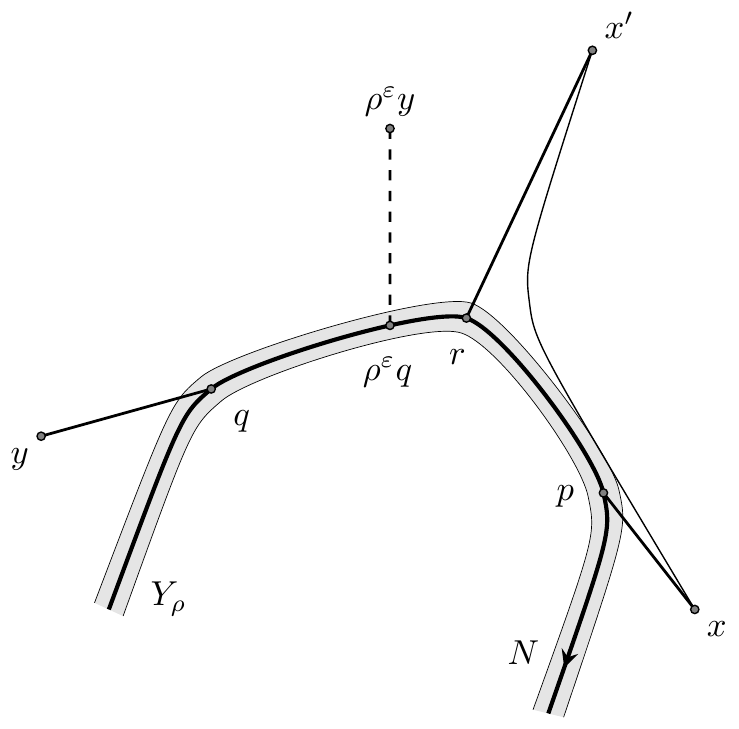}
	\caption{Case 1}
	\label{fig:satisfaire Greendlinger geodesique Case1}
	\end{figure}
	
	\paragraph{Case 1.}
	Assume that $\rho^\epsilon q$ belongs to $\nerf qrN$ (see Fig.~\ref{fig:satisfaire Greendlinger geodesique Case1}).
	Since $q$ is a projection of $y$ on $N$ we have $\dist yr \geq  \dist{\rho^\epsilon y}r + \len \rho  \sympy{-1*Faux4_SCsatisfaireGreendlinger()}$.
	Combined with the lower bound of $\gro x{x'}y$ given by~\ref{enu: greendlinger geodesic - gromov product},  we get
	\begin{displaymath}
		\gro x{x'}y \geq \gro x{x'}r + \dist {\rho^ \epsilon y}r + \len \rho  \sympy{-1*Faux5_SCsatisfaireGreendlinger()} \geq \gro x{x'}{\rho^\epsilon y} +\len \rho  \sympy{-1*Faux5_SCsatisfaireGreendlinger()}.
	\end{displaymath}
	
	\begin{figure}[ht]
		\subfigure[][{$\rho^\epsilon q$} between {$r$} and {$p$}]{
			\includegraphics[width=0.46\linewidth]{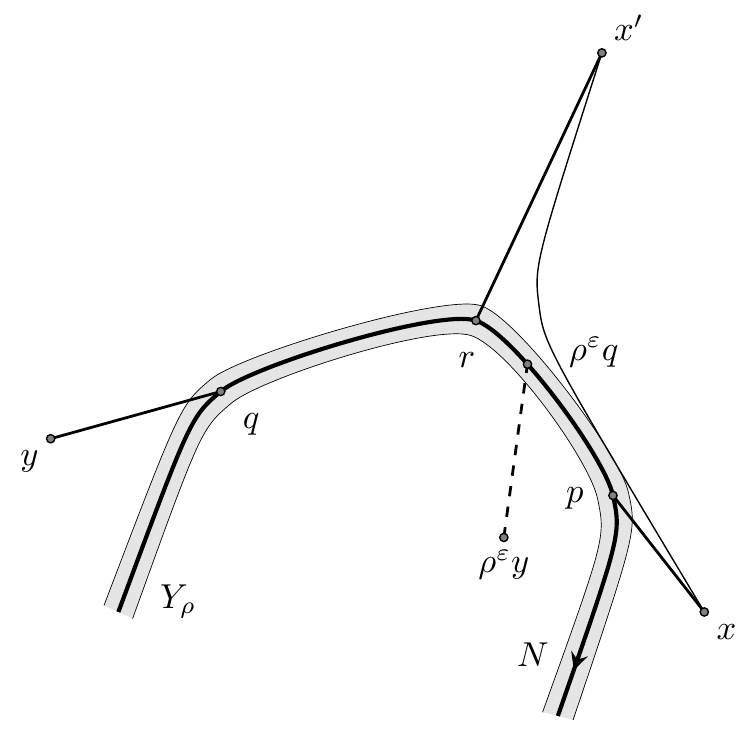}
			\label{fig:satisfaire Greendlinger geodesique Case2a}
		}
		\subfigure[][{$p$} between {$r$} and {$\rho^\epsilon q$}]{
			\includegraphics[width=0.50\linewidth]{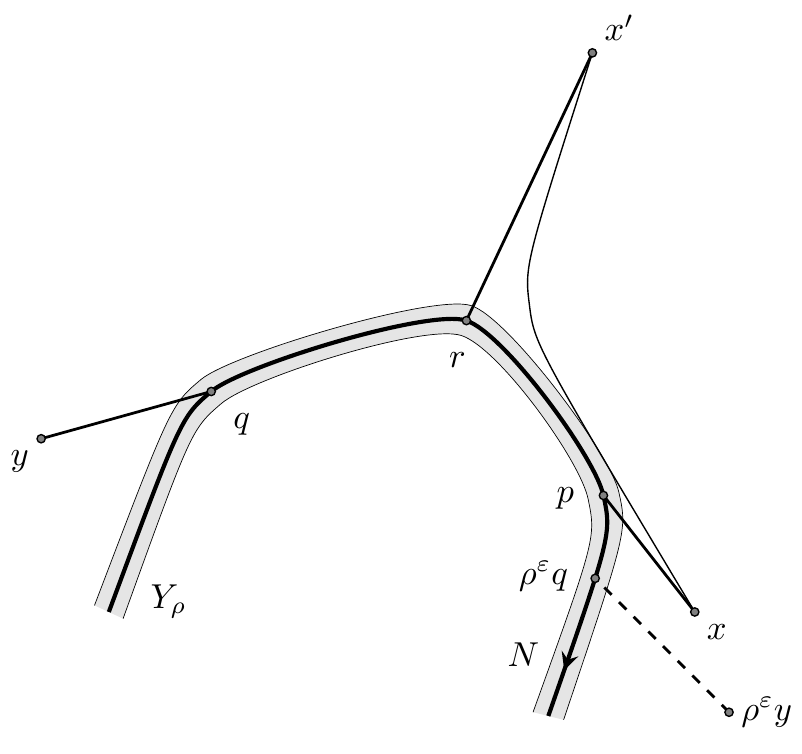}
			\label{fig:satisfaire Greendlinger geodesique Case2b}
		}
	\caption{Case 2}
	\label{fig:satisfaire Greendlinger geodesique Case2}
	\end{figure}
	
	\paragraph{Case 2.}
	Assume now that $\rho^\epsilon q$ does not belong to $\nerf qrN$.
	We claim that $\gro xr{\rho^\epsilon q} \leq  \sympy{Faux8_SCsatisfaireGreendlinger(c)}$.
	If $\rho^\epsilon q$ lies on $N$ between $r$ and $p$ (see Fig.~\ref{fig:satisfaire Greendlinger geodesique Case2a}) it follows from the definition of $N$.
	If not (see Fig.~\ref{fig:satisfaire Greendlinger geodesique Case2b}) $N$ being a $\len \rho$-local geodesic we have 
	\begin{displaymath}
		\gro xr{\rho^\epsilon q} \leq \dist p{\rho^\epsilon q} + \gro xrp = \len \rho - \dist pq +\gro xrp.
	\end{displaymath}
	The point $p$ is a projection of $x$ on $N$, thus $\gro xrp \leq \sympy{Faux6_SCsatisfaireGreendlinger()}$.
	Moreover by~\ref{enu: greendlinger geodesic - distances} $\len \rho - \dist pq \leq \sympy{Faux7_SCsatisfaireGreendlinger(c)}$, which completes the proof of our claim. 
	Applying the triangle inequality we get $\gro x{x'}{\rho^\epsilon q} \leq \gro x{x'}r + \gro xr{\rho^\epsilon q} \leq   \gro x{x'}r  + \sympy{Faux8_SCsatisfaireGreendlinger(c)}$.
	Combined with \ref{enu: greendlinger geodesic - distances} and \ref{enu: greendlinger geodesic - gromov product} it gives
	\begin{displaymath}
		\gro x{x'}y \geq \gro x{x'}{\rho^\epsilon q} + \dist {\rho^\epsilon q}{\rho^\epsilon y} +\frac 12\len \rho \sympy{-1*Faux9_SCsatisfaireGreendlinger(c)}.
	\end{displaymath}
	In both cases $\gro x{x'}{\rho^\epsilon y} \leq\gro x{x'}y  + \sympy{-1*Faux10_SCsatisfaireGreendlinger(c)} \leq \gro x{x'}y +2\alpha$, which contradicts our assumption on $y$.
\end{proof}

\subsection{$P$-reduced isometries}
\label{sec: reduced isometries}

\begin{defi}
	Let $g$ be an element of $G$. 
	The isometry $g$ is $P$-reduced if its image $\bar g $ in $\bar G$ is hyperbolic and for all $\rho \in P$, $\diaminter{Y_g}{Y_\rho} \leq \len \rho/2 + \sympy{V_Preduced}$.
\end{defi}

\rem Since $P$ is invariant under conjugation, all conjugates of a $P$-reduced isometry are also $P$-reduced.

\paragraph{} The next proposition explains how to construct $P$-reduced elements of $G$.
To that end we need to assume that the elements of $P$ are proper powers of small isometries.

\begin{prop}
\label{obtaining reduced isometry}
	Let $n \in \N^*$. 
	We assume that 
	\begin{enumerate}
		\item for all $\rho \in P$, there exists $r \in G$ such that $\len r \leq \sympy{F_SCobtainingRedIsomHyp1()}$ and $\rho = r^n$,
		\item $A(G,X) \leq \sympy{F_SCobtainingRedIsomHyp2()}$
	\end{enumerate}
	Let $g \in G$, such that its image $\bar g$ in $\bar G$ is hyperbolic.
	Then, there exists $u \in K$ such that $ug$ is $P$-reduced.
\end{prop}

\begin{proof}
	We choose $u \in K$ such that for all $v \in K$, $\len {ug} \leq \len {vg} + \sympy{Faux1_SCobtainingRedIsom()}$.
	Since $\bar g = \bar u \bar g$ is a hyperbolic element of $\bar G$, so is $ug$ in $G$.
	We suppose now that the isometry $ug$ is not $P$-reduced.
	There is $\rho \in P$, such that $\diaminter{Y_{ug}}{Y_\rho} > \len\rho/2 + \sympy{Faux2_SCobtainingRedIsom()}$.
	By assumption, there exists $r \in G$ such that $\len r \leq  \sympy{Faux3_SCobtainingRedIsom()}$ and $\rho = r^n$.
	From Proposition~\ref{embedded quasi-convex in an axis}, $Y_\rho = Y_r$ (\resp  $Y_{ug}$) lies in the $\sympy{Faux4_SCobtainingRedIsom()}$-neighbourhood of  $A_r$ and  (\resp $A_{ug}$).
	Hence $\diaminter{A_{ug}}{A_r} > \len \rho/2 + \sympy{Faux5_SCobtainingRedIsom()}$.
	Note that $ug$ and $r$ do not generate an elementary subgroup.
	The group $G$ satisfies indeed the small centralizers hypothesis.
	If it was the case, $\bar g$ would have finite order as $\bar r$, which contradicts the fact that $\bar g$ is hyperbolic.
	Thus Proposition~\ref{recouvrement axes} leads to $\len {ug} > \len \rho/2 -A(G,X) + \sympy{Faux6_SCobtainingRedIsom()}$.
	It follows from our assumptions and Lemma~\ref{shortening class length} that there exists $m \in \Z$ such that $\len{\rho^mug} < \len {ug} +  A(G,X)  \sympy{-1*Faux7_SCobtainingRedIsom()}$.
	However $\rho^mu$ belongs to $K$.
	This last inequality contradicts the definition of $u$.
	Consequently $ug$ is $P$-reduced.
\end{proof}

\begin{lemm}
\label{element reduit donne points proches}
	Let $g$ be a $P$-reduced element of $G$.
	Let $x$ and $x'$ be two points of $X$.
	For all $\rho \in P$ we have
	\begin{displaymath}
		\diaminter{\geo x{x'}}{Y_\rho} \leq \frac 12 \len \rho  + d\left(x, Y_g\right) + d\left(x',Y_g\right) + \sympy{F_SCreducedVsClose()}.
	\end{displaymath}
	In particular, if $d\left(x, Y_g\right) + d\left(x',Y_g\right) \leq \sympy{F_SCreducedVsCloseCoro()}$, then $x$ and $x$' are $P$-closed.
\end{lemm}

\begin{proof}
	Let $\rho$ be an element of $P$.
	Let $y$ and $y'$ be respective projections of $x$ and $x'$ on $Y_g$.
	One knows by (\ref{eqn: diameter with gromov products}) that 
	\begin{displaymath}
		\diaminter{\geo x{x'}}{Y_\rho} 
		\leq \diaminter{\geo y{y'}}{Y_\rho} + \gro y{y'}x + \gro y{y'}{x'}+\sympy{Faux1_SCreducedVsClose()}.
	\end{displaymath}
	However $g$ is $P$-reduced, therefore $\diaminter{\geo y{y'}}{Y_\rho} \leq \diaminter{Y_g}{Y_\rho} \leq \len g/2 +\sympy{Faux2_SCreducedVsClose()}$.
	On the other hand, $\gro y{y'}x \leq \dist xy = d\left(x, Y_g\right)$. 
	Similarly $\gro y{y'}{x'} \leq d\left(x', Y_g\right)$. 
\end{proof}

\begin{prop}
\label{satisfaire Greendlinger axe}
	Let $\alpha \geq 0$.
	Let $g$ be a $P$-reduced element of $G$.
	Let $x$ be a point of $X$ such that for all $u \in K$, $d\left(x,Y_g \right) \leq d\left( ux, Y_g\right) + 2\alpha$.
	Then, there exists $k_0$ such that for all $k \geq k_0$, for all $\rho \in P$, $\diaminter{\geo x{g^kx}}{Y_\rho} \leq \len \rho - c$ where $c =  \sympy{F_SCsatisfaireGreendlingerAxe(Qalpha)}$.
\end{prop}

\begin{proof}
	Let $y$ be a projection of $x$ on $Y_g$.
	The family $P$ only contains a finite number of conjugacy classes.
	Since $g$ is hyperbolic, there exists $k_0$ such that for all $k \geq k_0$, for all $\rho \in P$, $\dist y{g^ky} > \len \rho/2 +\sympy{Faux14_SCsatisfaireGreendlingerAxe()}$.
	Assume now that our proposition is false i.e., there exists $k \geq k_0$ and $\rho \in P$ such that $\diaminter{\geo x{g^kx}}{Y_\rho} > \len \rho - \sympy{Faux0_SCsatisfaireGreendlingerAxe(c)}$.
	The point $y$ is a projection of $x$ on $Y_g$, thus $\gro y{g^ky}x \leq d(x,Y_g)$.
	Moreover $Y_g$ is $\sympy{Faux1_SCsatisfaireGreendlingerAxe()}$-quasi-convex.
	It follows from our assumption on $x$ that that for all $u \in K$, $\gro y{g^ky}x  \leq \gro y{g^ky}{ux} + \sympy{Faux2_SCsatisfaireGreendlingerAxe(Qalpha)}$.
	On the other hand, $g$ is $P$-reduced.
	By Lemma~\ref{element reduit donne points proches}, $y$ and $g^ky$ are $P$-close.
	According to Proposition~\ref{satisfaire Greendlinger geodesique} $\diaminter{\geo x{g^ky}}{Y_\rho} \leq \len \rho - c'$, where $c' = \sympy{Faux3_SCsatisfaireGreendlingerAxe(Qalpha)}$.
	The same inequality holds if one replaces $\geo x{g^ky}$ by $\geo y{g^kx}$.
	We now denote by $p$ and $q$ respective projections of $x$ and $g^kx$ on $Y_\rho$.
	According to Lemma~\ref{res:diam quasi-convex and projections}
	\begin{equation}
	\label{eqn: satisfaire Greendlinger axe}
		\dist pq \geq \diaminter {\geo x{g^kx}}{Y_\rho} - \sympy{Faux4_SCsatisfaireGreendlingerAxe()} > \len \rho  \sympy{-1*Faux5_SCsatisfaireGreendlingerAxe(c)}.
	\end{equation}
	
	\paragraph{Claim.} $y$ is a $\sympy{Faux9_SCsatisfaireGreendlingerAxe()}$-projection of $p$ on $Y_g$.
	Thanks to Lemma~\ref{res: exending projections on quasi-convex} it is sufficient to show that $\gro xyp \leq \sympy{Faux8_SCsatisfaireGreendlingerAxe()}$.
	Assume that this statement is false.
	Let $z \in Y_\rho$. 
	By hyperbolicity we have
	\begin{displaymath}
		\min \left\{\gro xyp , \gro yzp \right\} \leq \gro xzp + \sympy{Faux6_SCsatisfaireGreendlingerAxe()} \leq \sympy{Faux7_SCsatisfaireGreendlingerAxe()}.
	\end{displaymath}
	Thus for every $z \in Y_\rho$, $\gro yzp \leq \sympy{Faux7_SCsatisfaireGreendlingerAxe()}$.
	In particular $p$ is a $\sympy{Faux10_SCsatisfaireGreendlingerAxe()}$-projection of $y$ on $Y_\rho$.
	Using Lemma~\ref{res:diam quasi-convex and projections} we obtain that 
	\begin{math}
		\dist pq
		\leq \diaminter{\geo y{g^kx}}{Y_\rho} + \sympy{Faux11_SCsatisfaireGreendlingerAxe()}.
		\leq \len \rho \sympy{-1*Faux12_SCsatisfaireGreendlingerAxe(cBis)},
	\end{math}
	which contradicts (\ref{eqn: satisfaire Greendlinger axe}).
	
	\paragraph{} In the the same way, we prove that $g^ky$ is a $\sympy{Faux9_SCsatisfaireGreendlingerAxe()}$-projection of $q$ on $Y_g$.
	It follows then from Lemma~\ref{res:diam quasi-convex and projections} that 
	\begin{displaymath}
		\dist y{g^ky}\leq \diaminter{\geo pq}{Y_g}+\sympy{Faux13_SCsatisfaireGreendlingerAxe()} \leq \diaminter {Y_\rho}{Y_g}+\sympy{Faux13_SCsatisfaireGreendlingerAxe()}. 
	\end{displaymath}
	By assumption $g$ is $P$-reduced.
	Consequently, $\dist y{g^ky} \leq \len \rho /2 + \sympy{Faux14_SCsatisfaireGreendlingerAxe()}$, which contradicts our assumption on $k$.
	Thus the proposition is true.
\end{proof}

\subsection{Foldable configurations}
\label{sec:foldable configurations}

\paragraph{} 
In this section, we are interested in the following situation.
Let $x$, $p$ and $q$ be three points of $X$ such that $x$ and $p$ (\resp $x$ and $q$) are $P$-close.
We assume that $p$ and $q$ have the same image $\bar p = \bar q$ in $\bar X$, but are distinct as points of $X$.
We would like to understand the reason why $p \neq q$ in $X$ and which transformation could move $p$ closer to $q$.

\paragraph{}
The idea is roughly the following.
Since $\bar p = \bar q$, there exists $g \in K\setminus\{1\}$ such that $q = gp$.
By the Greendlinger Lemma (Proposition~\ref{Greendlinger Lemma}), there exists $\rho \in P$, such that
\begin{displaymath}
	\diaminter{\geo pq}{Y_\rho} > \len \rho \sympy{-1*F_SCfoldableBla1()}.
\end{displaymath}
However $x$ and $p$ (\resp $x$ and $q$) are $P$-closed. 
Hence, half of the overlap between $Y_\rho$ and $\geo pq$ is covered by $\geo xp$ and the other half by $\geo xq$ (see Fig.~\ref{fig:folding roughly speaking}).
Using $\rho$ we translate the point $p$.
In particular there exists $\epsilon \in \{\pm 1\}$ such that
\begin{displaymath}
	\gro  {\rho^\epsilon p}q x \geq \gro pqx + \frac 12 \len \rho \sympy{-1*F_SCfoldableBla2()}.
\end{displaymath}
\begin{figure}[ht]
	\centering
	\includegraphics{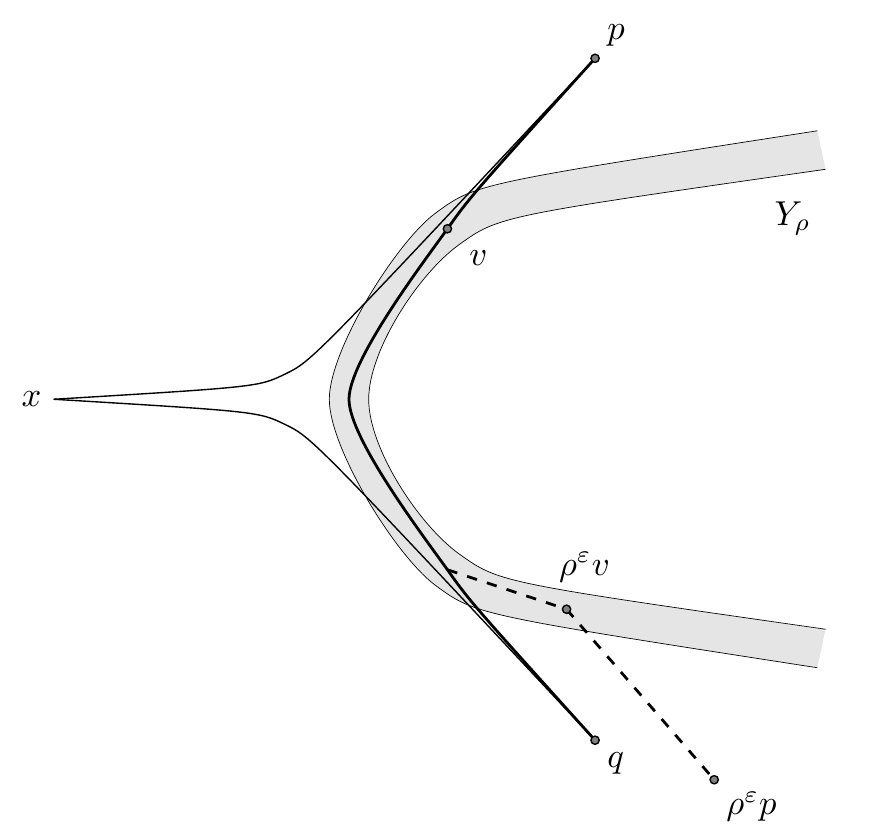}
	\caption{Folding a geodesic.}
	\label{fig:folding roughly speaking}
\end{figure}
By iterating the process, we increase at each step $\gro pqx$ (which is bounded above by $\dist xq$) until $p=q$.
To that end we need the points $x$ and $\rho^\epsilon p$ to be $P$-close, which is unfortunately not exactly the case: 
we might approximatively have 
\begin{displaymath}
	\diaminter{\geo x{\rho^\epsilon p}}{Y_\rho} \simeq \frac 12 \len \rho + \sympy{F_SCfoldableBla3()}
\end{displaymath}

The definition of \emph{foldable configuration} gives a set of conditions on $x$, $p$ and $q$ which are sufficient to detail the previous discussion and which will be still satisfied by $x$, $\rho^\epsilon p$ and $q$.

\begin{defi}[Foldable configuration]
\label{defi:foldable configuration}
	Let $x$, $p$, $q$ and $y$ be four points of $X$.
	We say that the configuration $(x,p,q,y)$ is \emph{foldable} if there exist $s, t \in X$ satisfying the following conditions (see Fig.~\ref{fig:definition foldable configuration}).
	\begin{foldable}
		\item \label{enu: def foldable - point s}
		$s$ and $p$ are $P$-close and $\dist xs \leq \gro pqx + \sympy{F_SCdefFoldable()}$,
		\item \label{enu: def foldable - point t}
		$t$ and $q$ are $P$-close and $\dist xt \leq \gro pqx + \sympy{F_SCdefFoldable()}$. 
		\item \label{enu: def foldable - point y}
		$s$ and $y$ are $P$-close and $\gro syp =0$.
	\end{foldable}
\end{defi}

\begin{figure}[ht]
	\centering
	\includegraphics{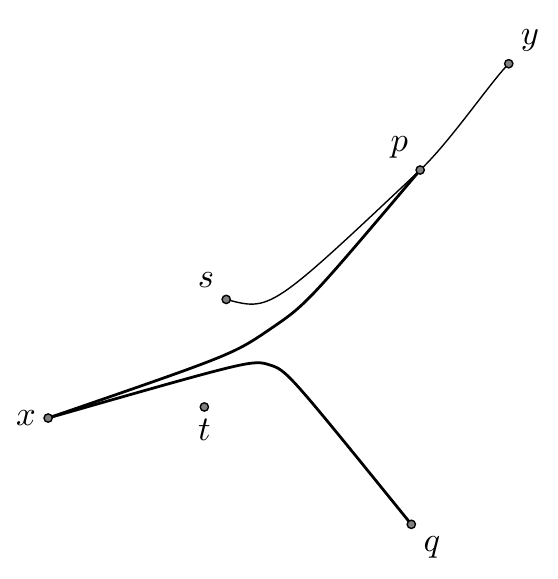}
	\caption{Definition of a foldable configuration.}
	\label{fig:definition foldable configuration}
\end{figure}

\rem 
This framework is a little more general than the one presented above.
It naturally arises when folding a geodesic $\geo xy$ on the axe of a relation (see Proposition~\ref{lift overlap})
The reason to keep track of the point $y$ will appear in Proposition~\ref{BI: part C}.

\begin{prop}
\label{folding proposition}
	Let $(x,p,q,y)$ be a foldable configuration such that $\bar p = \bar q$ but $p \neq q$.
	There exist $\rho \in P$ and $\epsilon \in \{ \pm 1 \}$ satisfying the followings.
	\begin{enumerate}
		\item \label{enu: folding proposition - diaminter xp Y}
		$\diaminter{\geo xy}{Y_\rho} \geq \len \rho/2 \sympy{-1*F_foldingPropOverlap()}$,
		\item \label{enu: folding proposition - gro x rho p q}
		$\gro {\rho^\epsilon p}q x \geq \gro pqx +\len \rho/2 \sympy{-1*F_foldingPropGromov()}$,
		\item \label{enu: folding proposition - foldable configuration}
		the configuration $(x,\rho^\epsilon p,q,\rho^\epsilon y)$ is foldable.
		\item \label{enu: folding proposition - gro with y}
		$\gro xyp \leq \sympy{F_foldingPropSmallPG1()}$ and $\gro x{\rho^\epsilon y}{\rho^\epsilon p} \leq \sympy{F_foldingPropSmallPG2()}$
	\end{enumerate}
\end{prop}

\begin{proof}
The points $s$ and $t$ are the one given by the definition of a foldable configuration.
We assumed that $\bar p = \bar q$ but $p \neq q$.
By Greendlinger's Lemma there exists $ \rho \in P$ such that $\diaminter{\geo pq}{Y_\rho} \geq \len \rho \sympy{-1*Faux1_foldingProp()}$.
We denote by $N$ a nerve of $\rho$ and by $u$, $v$, $w$ and $z$ respective projections of $x$, $p$, $q$ and $y$ on $N$.
According to Proposition~\ref{projection conditionnelle sur une axe version double}, $u$ lies on $N$ between $v$ and $w$ (see Fig.~\ref{fig:folding proposition uvwr}). 
Moreover we have
\begin{enumeratealph}
	\item \label{enu: folding proposition - dist vw}
	$\dist vw \geq \len \rho \sympy{-1*Faux2_foldingProp()}$,
	\item \label{enu: folding proposition - dist uv}
	$\len \rho /2 \sympy{-1*Faux3_foldingProp()} \leq \dist uv \leq \len \rho /2 + \sympy{Faux4_foldingProp()}$,
	\item \label{enu: folding proposition - dist uw}
	$\len \rho /2 \sympy{-1*Faux3_foldingProp()} \leq \dist uw \leq \len \rho /2 + \sympy{Faux4_foldingProp()}$,
	\item \label{enu: folding proposition - dist xu}
	$\dist{\dist xu}{\gro pqx} \leq \sympy{Faux5_foldingProp()}$,
	\item \label{enu: folding proposition - gro spv}
	$\gro spv \leq \sympy{Faux6_foldingProp()}$,
\end{enumeratealph}

\begin{figure}[ht]
\centering
	\includegraphics{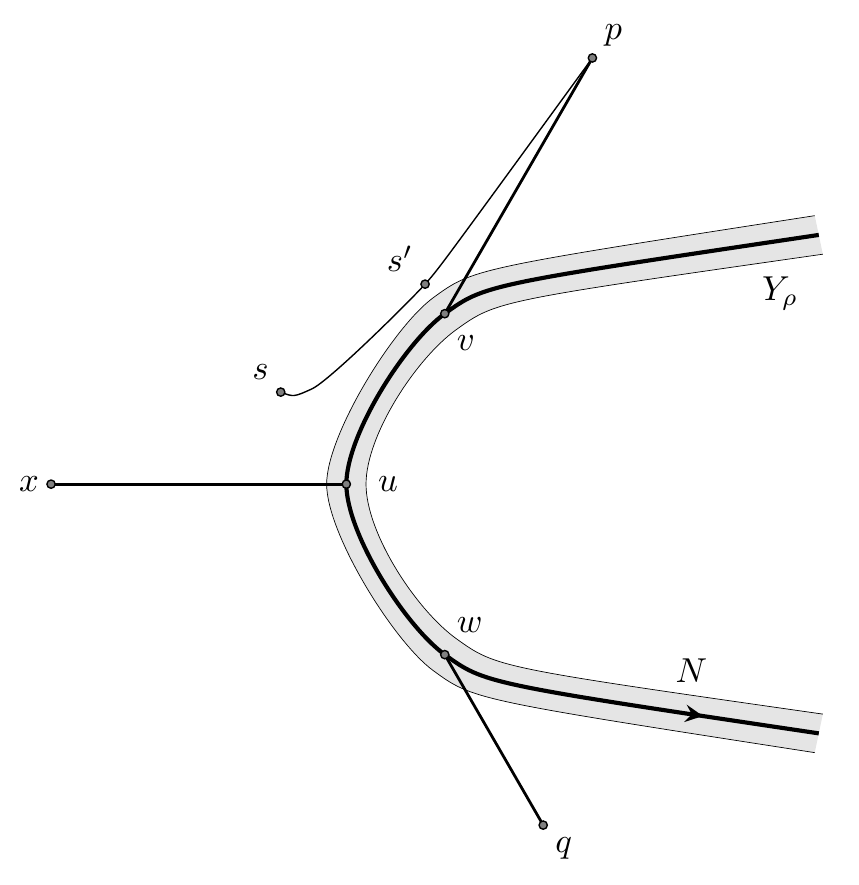}
	\caption{Positions of the points  $u$, $v$, $w$ and $s'$}
	\label{fig:folding proposition uvwr}
\end{figure}

\paragraph{On the configuration $(x,p,q,y)$.}
The points $u$ and $v$ are respective projections of $x$ and $p$ on $N$, thus $\dist xp \geq \dist xu + \dist uv \sympy{-1*Faux7_foldingProp()}$.
Combined with Points~\ref{enu: folding proposition - dist uv} and \ref{enu: folding proposition - dist xu}, we get
\begin{equation}
\label{eqn: satellite point - claim 1}
	\dist xp > \gro pqx + \sympy{Faux10_foldingProp()} \geq \dist xs + \sympy{Faux9_foldingProp()}.
\end{equation}
By hyperbolicity, $\min\left\{\gro xyp, \dist xp - \dist xs  \right\} \leq \gro syp +\sympy{Faux8_foldingProp()} \leq \sympy{Faux8bis_foldingProp()}$.
According to~(\ref{eqn: satellite point - claim 1}) we necessary have $\gro xyp \leq \sympy{Faux11_foldingProp()}$, which proves the first part of Point~\ref{enu: folding proposition - gro with y}.
The nerve $N$ is contained in the $\sympy{Faux12_foldingProp()}$-neighbourhood of $Y_\rho$.
Applying Proposition~\ref{res:diam quasi-convex and projections} with \ref{enu: folding proposition - dist uv} we get 
\begin{displaymath}
	\diaminter{\geo xy}{Y_\rho} \geq \diaminter{\geo xp}{Y_\rho} - \gro xyp \geq \len \rho /2 \sympy{-1*Faux13_foldingProp()},
\end{displaymath}
which corresponds to Point~\ref{enu: folding proposition - gro with y}.

\paragraph{Claim 1.} $\dist uz \leq \len \rho /2 + \sympy{Faux20_foldingProp()}$.
By hyperbolicity, we have 
\begin{displaymath}
	\gro syu \leq \max \left\{ \dist xs - \dist xu + 2 \gro xyu, \gro xyu \right\} + \sympy{Faux14_foldingProp()}.
\end{displaymath}
By \ref{enu: folding proposition - dist xu} we know that $\dist xs \leq \gro pqx + \sympy{Faux15_foldingProp()} \leq \dist xu + \sympy{Faux16_foldingProp()}$.
On the other hand the triangle inequality leads to $\gro xyu \leq \gro xyp + \gro xpu \leq \sympy{Faux17_foldingProp()}$.
It follows that $\gro syu \leq \sympy{Faux18_foldingProp()}$.
However $z$ is a projection of $y$ on $N$.
The points $s$ and $y$ being $P$-close Proposition~\ref{res:diam quasi-convex and projections} yields 
\begin{displaymath}
	\dist uz \leq \diaminter{\geo uy}{Y_\rho} +  \sympy{Faux19_foldingProp()} \leq \diaminter{\geo sy}{Y_\rho} +\gro syu + \sympy{Faux19_foldingProp()} \leq \frac 12 \len \rho + \sympy{Faux20_foldingProp()}.
\end{displaymath}

\paragraph{Claim 2.} $\gro zyp \leq \sympy{Faux24_foldingProp()}$.
By triangle inequality, $\gro zyp \leq \gro xyp + \gro xpv + \dist vz$.
The Gromov products on the left hand side of the inequality are small ($\gro xyp \leq \sympy{Faux11_foldingProp()}$ and $\gro xpv \leq \sympy{Faux20bis_foldingProp()}$) therefore it is sufficient to find an upper bound for $\dist vz$.
In particular we can assume that $\dist vz > \sympy{Faux22_foldingProp()}$.
Note that, since $\gro xyp \leq \sympy{Faux21_foldingProp()}$ the points $z$ and $u$ cannot belong to the same component of $N\setminus\{v\}$.
In other words $v$ lies between $u$ and $z$.
It follows from Claim 1 and Point~\ref{enu: folding proposition - dist uv} that $\dist vz = \dist uz - \dist uv \leq \sympy{Faux23_foldingProp()}$.

\begin{figure}[ht]
\centering
	\includegraphics{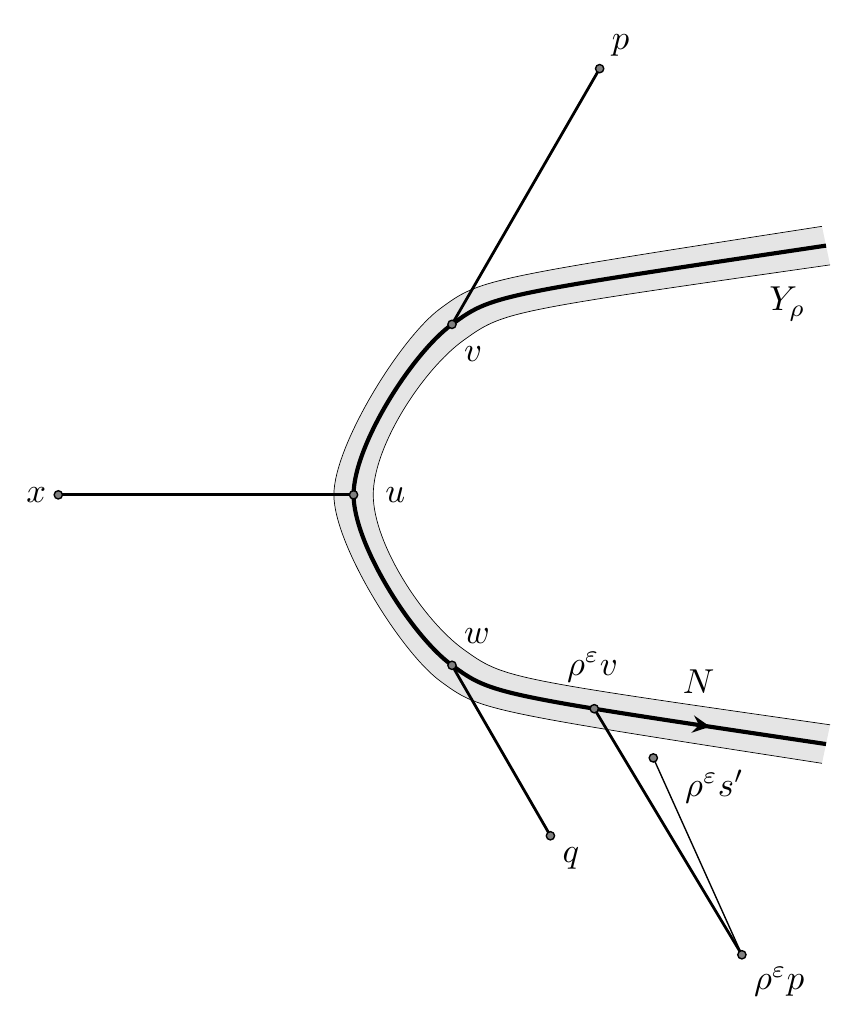}
	\caption{Positions of the point $\rho^\epsilon v$, $\rho^\epsilon s'$ and $\rho^\epsilon p$}
	\label{fig:folding proposition rho vrpy}
\end{figure}

\paragraph{Translation by $\rho$.}
The isometry $\rho$ acts by translation on $N$.
Therefore there exists $\epsilon \in \{\pm 1\}$ such that $\rho^\epsilon v$ and $w$ belong to the same component of $N \setminus\{v\}$ (see Fig.~\ref{fig:folding proposition rho vrpy}).

\paragraph{Claim 3.} $\dist x{\rho^\epsilon v} \leq \gro {\rho^\epsilon p}qx + \sympy{Faux33_foldingProp()}$.
Note that $\dist uv \leq \len \rho \leq \dist {\rho^\epsilon v}v$.
Thus $u$ lies on $N$ between $v$ and $\rho^\epsilon v$.
Since $N$ is a $\len \rho$-local geodesic,  $\dist{\rho^\epsilon v}u = \len \rho - \dist uv \geq \len \rho /2 \sympy{-1*Faux25_foldingProp()}$.
We now distinguish two cases.
If $\rho^\epsilon v$ lies on $N$ between $u$ and $w$.
Then $\gro xq{\rho^\epsilon v} \leq \sympy{Faux26_foldingProp()}$ and $\gro x{\rho^\epsilon p}{\rho^\epsilon v} \leq \sympy{Faux27_foldingProp()}$.
By hyperbolicity we obtain
\begin{displaymath}
	\dist x{\rho^\epsilon v} \leq \gro {\rho^\epsilon p}qx + \max\left\{ \gro x{\rho^\epsilon p}{\rho^\epsilon v}, \gro xq{\rho^\epsilon v}\right\} + \sympy{Faux28_foldingProp()} \leq \gro{\rho^\epsilon p}qx + \sympy{Faux29_foldingProp()}.
\end{displaymath}
Assume now that $w$ lies on $N$ between $u$ and $\rho^\epsilon v$.
As previously we show that $\dist xw \leq \gro {\rho^\epsilon p}qx + \sympy{Faux30_foldingProp()}$.
On the other hand $N$ is a $\len\rho$-local geodesic, thus using Point~\ref{enu: folding proposition - dist vw}, $\dist w{\rho^\epsilon v} = \len\rho - \dist vw \leq \sympy{Faux31_foldingProp()}$.
It follows from the triangle inequality that $\dist x{\rho^\epsilon v} \leq \dist xw + \dist w{\rho^\epsilon v} \leq \gro {\rho^\epsilon p}qx + \sympy{Faux32_foldingProp()}$, which completes the proof of our claim.

\paragraph{}
Combined with Point~\ref{enu: folding proposition - dist xu}, we get in particular
\begin{displaymath}
	\gro{\rho^\epsilon p}qx \geq \dist xu +\dist u{\rho^\epsilon v} \sympy{-1*Faux34_foldingProp()} \geq \gro pqx + \frac 12 \len \rho \sympy{-1*Faux35_foldingProp()}, 
\end{displaymath}
which is exactly Point~\ref{enu: folding proposition - gro x rho p q}.
We now prove that $(x,\rho^\epsilon p,q,\rho^\epsilon y)$ is foldable.
Note that the point $t$ already satisfies the condition~\ref{enu: def foldable - point t}.
Let us denote by $s'$ a projection of $v$ on $\geo sp$.
Since $s$ and $p$ are $P$-close, so are $s'$ and $p$ and thus $\rho^\epsilon s'$ and $\rho^\epsilon p$.
On the other hand, by Point~\ref{enu: folding proposition - gro spv}, $\dist v{s'} \leq \sympy{Faux36_foldingProp()}$.
Using Claim 3 we obtain $\dist x{\rho^\epsilon s'} \leq \dist x{\rho^\epsilon v} + \dist v{s'} \leq \gro {\rho^\epsilon p}qx + \sympy{Faux37_foldingProp()}$.
Consequently $\rho^\epsilon s'$ satisfies the condition \ref{enu: def foldable - point s}.
Since $\gro syp = 0$ there exists a geodesic joining $s$ to $y$ which extends the geodesic between $s$ and $p$ containing $s'$.
In particular $\gro{\rho^\epsilon s'}{\rho^\epsilon y}{\rho^\epsilon p} = \gro {s'}yp=0$.
The points $s$ and $y$ being $P$-close, so are $s'$ and $y$ and thus $\rho^\epsilon s'$ and $\rho^\epsilon y$.
Thus \ref{enu: def foldable - point y} is also fulfilled and $(x,\rho^\epsilon p,q,\rho^\epsilon y)$ is foldable.

\paragraph{}
In only remains to prove that $\gro x{\rho^\epsilon y}{\rho^\epsilon p} \leq \sympy{Faux41_foldingProp()}$.
The isometry $\rho$ acts on $N$ by translation of length $\len \rho$.
Moreover by Claim 1, $\dist uz \leq \len \rho /2 + \sympy{Faux38_foldingProp()}$.
Thus $\dist u{\rho^\epsilon z} \geq \len \rho /2 \sympy{-1*Faux39_foldingProp()}$.
In particular $\gro x{\rho^\epsilon y}{\rho^\epsilon z} \leq \sympy{Faux40_foldingProp()}$.
The triangle inequality and Claim 2 lead to $\gro x {\rho^\epsilon y}{\rho^\epsilon p} \leq \gro x{\rho^\epsilon y}{\rho^\epsilon z} + \gro zyp \leq \sympy{Faux41_foldingProp()}$, which completes the proof of Point~\ref{enu: folding proposition - gro with y} and of the proposition.
\end{proof}

\subsection{Lifting figures of $\bar X$ in $X$}

\paragraph{} In this section we try to find the best way to lift in $X$ a figure of $\bar X$.
Lemma~\ref{releve distance a une geodesique} (\resp Lemma~\ref{releve distance a un axe}) explains how to lift a point of $\bar X$ which is close to a geodesic (\resp the cylinder of an isometry) with a point of $X$ having a similar property.
In Proposition~\ref{lift overlap} we are interested in the following situation.
Let $x$ and $y$ be two $P$-close points of $X$ and $g$ a $P$-reduced isometry of $G$.
We assume that $\geo {\bar x}{\bar y}$ and $Y_{\bar g}$ have a large overlap in $\bar X$ (for instance larger than $\len{\bar g ^k}$ with $k \gg 1$) and would like to ``lift'' this overlap.
By replacing if necessary $g$ by a conjugate of $g$ we may translate $Y_g$ such that $\geo xy$ and $Y_g$ have more or less a non-empty intersection. 
However there is no reason that this overlap should be as large in $X$ as in $\bar X$.
We face the same kind of problem exposed at the beginning of Section~\ref{sec:foldable configurations}.
Nevertheless, lifting the endpoints of $\geo{\bar x}{\bar y}\cap Y_{\bar g}$, one can build a foldable configuration.
In the same way as explained in Section~\ref{sec:foldable configurations}, we will use this configuration in Section~\ref{sec:Burnside groups} in order to translate $y$ by elements of $P$ and fold the geodesic $\geo xy$ onto $Y_g$.

\begin{lemm}
\label{releve distance a une geodesique}
	Let $x$ and $x'$ be two $P$-close points of $X$.
	Let $y \in X$ such that for all $u \in K$, $\gro x{x'}y \leq \gro x{x'}{uy} + 2\delta$. 
	Moreover we assume that $\gro {\bar x}{\bar x'}{\bar y} \leq \sympy{F_SCliftDistGeoAssumption()}$,
	Then $\gro x{x'}y  \leq \sympy{F_SCliftDistGeoConclusion()}$.
\end{lemm}

\begin{proof}
	The points $x$ and $x'$ are $P$-close. 
	Hence by Proposition~\ref{satisfaire Greendlinger geodesique}, for all $\rho \in P$, $\diaminter{\geo xy}{Y_\rho}$ and $\diaminter{\geo {x'}y}{Y_\rho}$ are smaller than $\len \rho \sympy{-1*Faux1_SCliftDistGeo()}$.
	The result follows then from Proposition~\ref{Preserving shape Lemma}.
\end{proof}

\begin{lemm}
\label{releve distance a un axe}
	Let $g$ be a $P$-reduced element of $G$.
	Let $x \in X$ such that for all $u \in K$, $d\left(x,Y_g\right) \leq d\left(ux, Y_g\right) + 2\delta$.
	We assume also that $d\left( \bar x, Y_{\bar g} \right) \leq \sympy{F_SCliftDistAxeAssumptionQuo()}$.
	Then $d\left(x, Y_g \right) \leq  \sympy{F_SCliftDistAxeConclusion()}$.
\end{lemm}

\begin{proof}
	By Proposition~\ref{satisfaire Greendlinger axe}, there exists $k_0 \in \N$ such that for all $k \geq k_0$, for all $\rho \in P$, $\diaminter{\geo x{g^kx}}{Y_\rho} \leq \len \rho \sympy{-1*Faux1_SCliftDistAxe()}$.
	However $\bar g$ is a hyperbolic isometry.
	Therefore, there exists $k \geq k_0$ such that $\len {g^k} > \sympy{Faux2_SCliftDistAxe()}$ and $\len {\bar g^k} > \sympy{Faux3_SCliftDistAxe()}$.
	It follows from Lemma~\ref{distance axe gromov product} that the distance from $x$ to $Y_g$ is approximatively given by $\gro{g^{-k}x}{g^kx}x$.
	The same works for $\bar x$ and $Y_{\bar g}$.
	More precisely,
	\begin{displaymath}
		\gro {\bar g^{-k} \bar x}{\bar g^k\bar x}{\bar x} 
		\leq d\left(\bar x, Y_{\bar g}\right) + \sympy{Faux4_SCliftDistAxe()}
		\leq \sympy{Faux5_SCliftDistAxe()}.
	\end{displaymath}
	Applying Proposition~\ref{Preserving shape Lemma} we get 
	\begin{displaymath}
		d\left(x,Y_g\right)
		\leq \gro{g^{-k}x}{g^kx}x + \sympy{Faux6_SCliftDistAxe()}
		\leq  \sympy{Faux7_SCliftDistAxe()},
	\end{displaymath}
	which completes the proof of the lemma.
\end{proof}

\begin{prop}
\label{lifting distances}
	Let $k \in \N$.
	Let $L \geq \sympy{F_SCliftingDistHypL()}$.
	Let $g$ be a $P$-reduced element of $G$ such that $\len {\bar g^k} > \sympy{F_SCliftingDistHypTL()}$.
	Let $p$ and $q$ be two points of $X$ satisfying the followings
	\begin{enumerate}
		\item \label{enu: lifting distances - hyp bqr X}
		$d\left(\bar p, Y_{\bar g} \right), d\left(\bar q, Y_{\bar g} \right) \leq\sympy{F_SCliftingDistHypAxeQuo()}$,
		\item \label{enu: lifting distances - hyp X}
		for all $u \in K$, $d\left(p,Y_g \right) \leq d\left(up, Y_g\right) + \sympy{F_SCliftingDistHypAxeBase()}$ and $d\left(q,Y_g \right) \leq d\left(uq, Y_g\right) + \sympy{F_SCliftingDistHypAxeBase()}$,
		\item \label{enu: lifting distances - hyp dist pq}
		$\dist{\bar p}{\bar q} \geq \len{\bar g^k} + L$.
	\end{enumerate}
	Then $\dist pq \geq \len{g^k} + L \sympy{-1*F_SCliftingDistCl()}$.
\end{prop}

\begin{proof}
	Let $\bar N$ be a nerve of $\bar g^k$ (in $\bar X$).
	We denote by $\bar r$ and $\bar s$ respective projections of $\bar p$ and $\bar q$ on $\bar N$.
	The isometry $\bar g^k$ acts on $\bar N$ by translation of length $\len{\bar g^k}$.
	By replacing if necessary $g$ by $g^{-1}$, we can assume that $\bar s$ and $\bar g^k \bar r$ belong to the same component of $\bar N\setminus\{ \bar r\}$.
	Since $\len{\bar g^k} > \sympy{Faux1_SCliftingDist()}$, $Y_{\bar g}$ is contained in the $\sympy{Faux2_SCliftingDist()}$-neighbourhood of $\bar N$.
	In particular $\dist {\bar p}{\bar r} \leq \sympy{Faux3_SCliftingDist()}$ and $\dist {\bar q}{\bar s} \leq \sympy{Faux3_SCliftingDist()}$.
	It follows from the triangle inequality that $\dist{\bar r}{\bar s} \geq \dist{\bar p}{\bar q} \sympy{-1*Faux4_SCliftingDist()} \geq \len{\bar g^k}$. 
	However $\bar N$ is a $\len{\bar g^k}$-local geodesic, thus $\bar g^k \bar r$ necessary belongs to $\nerf{\bar r}{\bar s}{\bar N}$ and $\gro{\bar p}{\bar q}{\bar g^k\bar r} \leq \sympy{Faux5_SCliftingDist()}$.
	Hence $\gro{\bar p}{\bar q}{\bar g^k\bar p}  \leq \gro{\bar p}{\bar q}{\bar g^k\bar r} + \dist{\bar r}{\bar p} \leq \sympy{Faux6_SCliftingDist()}$.
	According to Point~\ref{enu: lifting distances - hyp X} $p$ and $q$  are the respective lifts of $\bar p$ and $\bar q$ which are the ``closest'' to $Y_g$.
	Hence by Lemma~\ref{releve distance a un axe}, $p$ and $q$ belongs to the $(\sympy{Faux7_SCliftingDist()})$-neighbourhood of $Y_g$.
	It follows from Lemma~\ref{element reduit donne points proches} that for all $\rho \in P$, $\diaminter{\geo p{g^kp}}{Y_\rho}$ and $\diaminter{\geo q{g^kp}}{Y_\rho}$ are bounded above by $\len \rho \sympy{-1*Faux8_SCliftingDist()}$.
	Consequently by Proposition~\ref{Preserving shape Lemma} $\gro pq{g^kp} \leq \sympy{Faux9_SCliftingDist()}$.
	In particular
	\begin{equation}
	\label{eqn: lifting distances}
		\dist pq \geq \dist p{g^kp} + \dist {g^kp}q \sympy{-1*Faux10_SCliftingDist()}  \geq \len{g^k} + \dist {g^kp}q \sympy{-1*Faux10_SCliftingDist()}.
	\end{equation}
	However the map $X \rightarrow \bar X$ shorten the distances, thus
	\begin{displaymath}
		\dist{g^kp}q \geq \dist {\bar g^k \bar p}{\bar q} 
		\geq \dist{\bar p}{\bar q} - \dist {\bar g^k \bar p}{\bar p}
		\geq \dist{\bar p}{\bar q} -  \len{\bar g^k} - 2\dist{\bar p}{\bar r}
	\end{displaymath}
	Using Point~\ref{enu: lifting distances - hyp dist pq} we deduced that $\dist{g^kp}q \geq L \sympy{-1*Faux11_SCliftingDist()}$, which together with (\ref{eqn: lifting distances}) leads to the result.
\end{proof}

\begin{prop}
\label{lift overlap}
	Let $x$ and $y$ be two $P$-close points of $X$.
	Let $g$ be a $P$-reduced element of $G$.
	Let $k \in \N$ such that $\len{\bar g^k} >  \sympy{F_SCliftOverlapHypTL()}$.
	Let $L \geq \sympy{F_SCliftOverlapHypL()}$ such that $\diaminter{\geo {\bar x}{\bar y}}{Y_{\bar g}} \geq \len {\bar g^k} +  L$.
	There exists three points $r,p,q \in X$ and $v \in K$ satisfying the following properties
	\begin{enumerate}
		\item \label{enu: lift overlap - p=q in bar X}
		$\bar p = \bar q$.
		\item \label{enu: lift overlap - estimation distances aux objets}
		$d\left(r, vY_g\right) \leq \sympy{F_SCliftOverlapClDistAxeR()}$, $d\left(q,vY_g\right) \leq \sympy{F_SCliftOverlapClDistAxeQ()}$, $\gro xqr \leq \sympy{F_SCliftOverlapClGroR()}$ and $\gro xyp \leq \sympy{F_SCliftOverlapClGroP()}$, ,
		\item \label{enu: lift overlap - estimate distance releves}
		$\dist rq \geq \len {g^k} + L \sympy{-1*F_SCliftOverlapClOverlap()}$.
		\item \label{enu: lift overlap - configuration pliable} 
		 The configuration $(x,p,q,y)$ is foldable.
	\end{enumerate}
\end{prop}

\begin{proof}
	Let us denote by $\bar a$ and $\bar b$ respective projections of $\bar x$ and $\bar y$ on $Y_{\bar g}\subset \bar X$.
	By Proposition~\ref{res:diam quasi-convex and projections}, $\dist{\bar a}{\bar b} \geq \len{\bar g^k}  + L \sympy{-1*Faux1_SCliftOverlap()}$.
	Recall that $\bar X$ is obtained by attaching cones on $X / K$.
	Hence $\bar a$ and $\bar b$ may not belong to $\nu \left(X\right)$, the image of $X$ in $\bar X$. 
	However these cones have diameter $\sympy{Faux2_SCliftOverlap()}$.
	Thus there exists two points $\bar r$ and $\bar z$ in $\geo{\bar a}{\bar b}\cap \nu \left(X\right)$, such that $\dist{\bar a}{\bar r}, \dist{\bar b}{\bar z} \leq \sympy{Faux3_SCliftOverlap()}$.
	In particular $\dist{\bar r}{\bar z} \geq \len{\bar g^k} + L \sympy{-1*Faux4_SCliftOverlap()}$.
	Since $Y_{\bar g}$ is $\sympy{Faux5_SCliftOverlap()}$-quasi-convex, $\bar r$ and  $\bar z$ are in the $\sympy{Faux6_SCliftOverlap()}$-neighbourhood of  $Y_{\bar g}$.
	Moreover, $\gro {\bar x}{\bar y}{\bar r}, \gro {\bar x}{\bar y}{\bar z} \leq \sympy{Faux7_SCliftOverlap()}$ and $\gro {\bar x}{\bar z}{\bar r} \leq \sympy{Faux8_SCliftOverlap()}$.
	The next step of the proof consists in lifting this figure in $X$.
	First we define lifts of $\bar r$ and $\bar z$ which are as close as possible from $\geo xy$.
	Let $r, z \in X$ be respective pre-images of $\bar r$ and $\bar z$ such that for all $u \in K$, we have in $X$ $\gro xyr \leq \gro xy{ur} + \sympy{Faux9_SCliftOverlap()}$ and  $\gro xyz \leq \gro xy{uz}+ \sympy{Faux9_SCliftOverlap()}$.
	Since $x$ and $y$ are $P$-close, Lemma~\ref{releve distance a une geodesique} leads to $\gro xyr , \gro xyz \leq \sympy{Faux10_SCliftOverlap()}$.
	In particular there is a point $p$ on $\geo xy$ such that $\dist pz \leq \sympy{Faux11_SCliftOverlap()}$ and $\gro xyp \leq \gro xyz + \dist pz \leq \sympy{Faux12_SCliftOverlap()}$.

	\paragraph{} We now chose a conjugate of $g$ whose axes in $X$  approximatively passes through $r$.
	To that end, we fix $v\in K$ such that for all $u \in K$, we have $d\left(r,vY_g\right) \leq d\left(ur,vY_g\right) + \sympy{Faux13_SCliftOverlap()}$.
	By assumption $g$ is $P$-reduced.
	Hence $vY_g$ is the cylinder of $vgv^{-1}$ which is $P$-reduced as well.
	By Lemma~\ref{releve distance a un axe}, $d\left(r,vY_g\right) \leq \sympy{Faux14_SCliftOverlap()}$.
	We chose for $z$ a lift of $\bar z$ which was close to $\geo xy$.
	Unfortunately $z$ is not necessarily in the neighbourhood of $vY_g$.
	That is why we have to introduce a second pre-image of $\bar z$. 
	Let $w \in K$ such that for all $u \in K$, $d\left(wz,vY_g\right) \leq d\left(uwz,vY_g\right) + \sympy{Faux15_SCliftOverlap()}$.
	By Lemma~\ref{releve distance a un axe}, $d\left(wz,vY_g\right) \leq \sympy{Faux16_SCliftOverlap()}$.
	We finally put $q = wp$.
	In particular $d(q, vY_g)\leq \sympy{Faux17_SCliftOverlap()}$.
	Moreover $\bar p = \bar q$, which proves Point~\ref{enu: lift overlap - p=q in bar X}.
	
	\paragraph{} 
	By construction $\gro xyr \leq \sympy{Faux18_SCliftOverlap()}$.
	However $x$ and $y$ are $P$-close.
	Hence for all $\rho \in P$, $\diaminter{\geo xr}{Y_\rho} \leq \diaminter{\geo xy}{Y_\rho} +\gro xyr \leq\len \rho \sympy{-1*Faux19_SCliftOverlap()}$.
	On the other hand, $d(r,vY_g)$ and $d(wz,vY_g)$ are bounded above by $\sympy{Faux20_SCliftOverlap()}$.
	The isometry $vgv^{-1}$ being $P$-reduced, Lemma~\ref{element reduit donne points proches} implies that for all $\rho \in P$, $\diaminter{\geo r{wz}}{Y_\rho} \leq \len \rho \sympy{-1*Faux21_SCliftOverlap()}$.
	Since $\gro {\bar x}{\bar z}{\bar r} \leq \sympy{Faux8_SCliftOverlap()}$, applying Lemma~\ref{Preserving shape Lemma} we get 
	\begin{math}
		\gro xqr 
		\leq \gro x{wz}r +\dist pz 
		\leq \sympy{Faux23_SCliftOverlap()},
	\end{math}
	which completes the proof of Point~\ref{enu: lift overlap - estimation distances aux objets}.
	(In the same way, we can prove that $\gro xpr \leq \sympy{Faux29_SCliftOverlap()}$.)
	Note that $vgv^{-1}$, $r$ and $wz$ satisfy the assumptions of Proposition~\ref{lifting distances}.
	Therefore $\dist r{wz} \geq \len{g^k} + L \sympy{-1*Faux24_SCliftOverlap()}$.
	Thus $\dist rq \geq \len{g^k} + L \sympy{-1*Faux25_SCliftOverlap()}$, which gives Point~\ref{enu: lift overlap - estimate distance releves}.
	
	\paragraph{}It only remains to prove that $(x,p,q,y)$ is foldable.
	In the definition of foldable configuration we choose $s=x$. 
	Since $x$ and $y$ are $P$-close and $p$ lies on a geodesic between them, Assumption \ref{enu: def foldable - point s} is fulfilled. 
	So is the condition \ref{enu: def foldable - point y}.
	We choose for $t$ the point $r$.
	We proved that $d(r,vY_g) \leq \sympy{Faux26_SCliftOverlap()}$ and $d(q,vY_g) \leq \sympy{Faux27_SCliftOverlap()}$.
	Moreover $vgv^{-1}$ is $P$-reduced.
	By Lemma~\ref{element reduit donne points proches}, $r$ and $q$ are $P$-close.
	On the other hand $\gro xqr \leq \sympy{Faux28_SCliftOverlap()}$ and $\gro xpr \leq \sympy{Faux29_SCliftOverlap()}$.
	Therefore by hyperbolicity $\dist xr \leq \gro pqx +\sympy{Faux30_SCliftOverlap()}$.
	Thus Condition~\ref{enu: def foldable - point t} holds.
\end{proof}

%% file: 4_burnside_groups.tex

\section{Burnside groups}
\label{sec:Burnside groups}

\subsection{General framework}
\label{framework sequence for Burnside}
\paragraph{}
This section is dedicated to the proof of our main theorem.
Let $(X,x_0)$ be a geodesic, proper, hyperbolic  pointed space.
Let $G$ be a non-elementary, torsion-free group acting freely, properly, co-compactly, by isometries on $X$.

\paragraph{}
In order to study the quotient $G/G^n$, T.~Delzant and M.~Gromov provides in \cite{DelGro08} a sequence of appropriate hyperbolic groups $(G_k)$ whose direct limit is $G/G^n$.
We recall here the main steps of this construction as it is exposed in \cite{Coulon:2013tx}.

\paragraph{}
The constants $\delta_1$, $r_0$, $\delta_0$ and $\Delta_0$ are the one given at the end of Section~\ref{sec:small cancellation constants}.
The rescaling parameter $\lambda_n$ is defined by
\begin{equation*}
	\lambda_n = \frac {\pi \sinh r_0}{5\sqrt{nr_0 \delta_1}}.
\end{equation*}
The integer $n_0$ is chosen large enough in such a way that $\lambda_{n_0}$ satisfy a set of inequalities%
\footnote{
	In this article, the exact statement of the inequalities it should satisfy is not important.
	There are chosen in such a way that one can iterate the small cancellation process explained below.
	The conditions to fulfill coarsely say that $\lambda_n \delta_1\ll \min\left\{\delta_0, \Delta_0\right\}$.
	For more details see \cite{Coulon:2013tx}.
}.
For our purpose, we also require that $\lambda^{-1}_{n_0}\geq 500$.
We build by induction two sequences $(X_k)$ and $(G_k)$ as follows.

\paragraph{Initialization.} 
Among other things, we can assume, by rescaling $X$ if necessary, that $X$ is $\delta$-hyperbolic, with $\delta \leq \delta_0$ and $A(G,X) \leq \Delta_0/2$.
Up to increase $n_0$, we may also require that $\rinj GX \geq 20 \sqrt{r_0\delta_1/n_0}$.
We fix now $\xi$ such that
\begin{equation*}
	40(\xi -1) \sqrt{r_0\delta_1/n_0} \geq \sympy{F_BIdefXi()}
\end{equation*}
and an odd integer $n \geq \max\{100, n_0, 2 \epsilon +1\}$ satisfying
\begin{equation*}
	\frac {500 \pi \sinh r_0}n \leq 20 \sqrt{r_0\delta_1/n_0}
\end{equation*}
We put $X_0 = X$ and $G_0=G$.
For simplicity of notation we write $\lambda$ instead of $\lambda_{n_0}$.
\paragraph{Heredity.}
We assume that $X_k$ and $G_k$ are built and satisfy (among others) the following assumptions.
\begin{enumerate}
	\item The metric space $X_k$ is geodesic, proper and $\delta$-hyper\--bolic, with $\delta \leq \delta_0$.
	\item The group $G_k$ acts properly, co-compactly by isometries on $X_k$ and satisfies the small centralizers hypothesis (i.e. it is non-elementary and all its elementary subgroups are cyclic).
	\item $A\left(G_k, X_k\right) \leq \Delta_0/2$.
	\item $\rinj{G_k}{X_k}  \geq 20 \sqrt{r_0\delta_1/n_0} \geq \sympy{F_BIrinjGk()}$. In particular, the injectivity radius of $G_k$ satisfies $2(\xi-1)\rinj{G_k}{X_k} \geq \sympy{F_BIdefXi()}$.
\end{enumerate}
We denote by $R_k$ the set of elements $g \in G_k$ such that $g$ is hyperbolic, not a proper power and $\len[espace=X_k]g \leq \sympy{F_BIdefIsomRk()}$.
There exists a subset $R_k^0$ of $R_k$ stable under conjugation such that $R_k$ is the disjoint union of $R_k^0$ and the set of all inverses of $R_k^0$.
We define $P_k$ by $P_k = \left\{ g^n, g \in R_k^0\right\}$.
This set satisfies the hypothesis of the small cancellation theorem (Theorem \ref{theo:small cancellation theorem}), i.e. $\Delta\left(P_k, X_k \right) \leq \Delta_0$ and $\rinj {P_k}{X_k} \geq \sympy{F_BIrinjPk()}$.
Let $G_{k+1}$ be the quotient $G_k/ \ll P_k\gg$.
The space $\bar X_k$ is the one constructed from $X_k$ by small cancellation (see Section~\ref{sec:small cancellation}).
It is $\bar \delta$-hyperbolic, with $\bar \delta \leq \delta_1$.
We define $X_{k+1}$ as the rescaled space $\lambda \bar X_k$.
Using the conditions satisfied by $\lambda$, one can prove that $X_{k+1}$ and $G_{k+1}$ satisfy also the assumptions (i)--(iv).
Moreover the canonical map $\nu_k : X_k \rightarrow X_{k+1}$ has the following property: for all $x,x' \in X_k$, $\dist[X_{k+1}] {\nu_k (x)}{\nu_k(x')} \leq \lambda \dist[X_k]x{x'}$.

\paragraph{}The sequence $(G_k)$ constructed in this way approximates the Burnside group $G/G^n$ in the sense that $\dlim G_k = G/G^n$.

\notas
\begin{enumerate}
	\item For all $k \in \N$ the kernel of the projection $G \twoheadrightarrow G_k$ is denoted by $K_k$.
	In particular, for all $k \in \N$, $K_k \lhd K_{k+1}$.
	\item Let $x$ be a point of $X$ (\resp $g$ be an element of $G$).
	For simplicity of notation, we still write $x$ (\resp $g$) for its image by the natural map $X \rightarrow X_k$ (\resp $G \twoheadrightarrow G_k$).
\end{enumerate} 

\subsection{Close points, reduced elements of rank $k$}

\rem From now on, unless otherwise stated, all the metric objects (distances, diameters, Gromov's products) are measured with the distance of $X_k$ (and never with the one of $\bar X_k)$.

\begin{defi}
	Let $k \in \N$.
	Two points $x$ and $x'$ of $X$ are close of rank $k$ if for all $j < k$, for all $\rho \in P_j$, $\diaminter{\geo x{x'}}{Y_\rho}\leq \len \rho/2 + \sympy{F_BIcloseRankK()}$ in the space $X_j$.
\end{defi}

\begin{defi}
	Let $k \in \N$.
	An element $g$ of $G$ is reduced of rank $k$ if $g$ is hyperbolic as element of $G_k$ and for all $j < k$, $\diaminter{Y_g}{Y_\rho}\leq \len \rho/2 + \sympy{F_BIreducedRankK()}$ in the space $X_j$.
\end{defi}

\rem Note that \emph{being close} (\resp \emph{reduced}) \emph{of rank 0} is an empty condition. 
Any two points of $X$ are close of rank 0.
Any hyperbolic element of $G$ is reduced of rank 0. 

\begin{prop}
\label{lifting reduced elements}
	Let $k \in \N$.
	Let $g \in G$.
	 If $g$ is hyperbolic in $G_k$ then there exists $u \in K_k$ such that $ug$ is reduced of rank $k$.
\end{prop}

\begin{proof}
	The proof is by induction on $k$.
	Since every hyperbolic element of $G$ is reduced of rank 0, the proposition is true for $k=0$.
	Assume now that the proposition holds for $k \in \N$.
	Let $g \in G$ such that $g$ is hyperbolic in $G_{k+1}$.
	By Proposition~\ref{obtaining reduced isometry} there exists $u \in K_{k+1}$ such that $ug$ is $P_k$-reduced, i.e. for all $\rho \in P_k$, $\diaminter{Y_{ug}}{Y_\rho}\leq \len \rho/2 + \sympy{Faux1_BIliftingReducedElts()}$ in the space $X_k$.
	Note that $g = ug$ in $G_{k+1}$. 
	Thus $ug$ is hyperbolic in $G_{k+1}$ and therefore in $G_k$.
	We apply the induction hypothesis on $ug$: there exists $v \in K_k$ such that $vug$ is reduced of rank $k$.
	However $vug = ug$ in $G_k$. 
	Hence for all $j \leq k$, for all $\rho \in P_j$, $\diaminter{Y_{vug}}{Y_\rho}\leq \len \rho/2 +\sympy{Faux2_BIliftingReducedElts()}$ in the space $X_j$, which means that $vug$ is reduced of rank $k+1$.
	Moreover, since $K_k \lhd K_{k+1}$, $vu \in K_{k+1}$.
	Consequently the proposition holds for $k+1$.
\end{proof}

\subsection{Elementary moves in $X$}

	Recall that $x_0$ is a base point of $X$.
	
	\begin{defi}
		Let $y$ and $z$ be two points of $X$. 
		\begin{itemize}
			\item We say that $z$ is the image of $y$ by a \emph{$(\xi,n)$-elementary move} (or simply an \emph{elementary move}), if there exist $g \in G$ such that
			\begin{enumerate}
				\item $\diaminter{\geo {x_0}y}{Y_g} \geq \len{g^m}$ in the space $X$, with $m \geq n/2-\xi$.
				\item $z = g^{-n}y$ in $X$.
			\end{enumerate}
			\item We say that $z$ is the image of $y$ by a \emph{sequence of elementary moves}, and we write $y \rightarrow z$, if there exists a finite sequence of points of $X$, $y=y_0, y_1, \dots, y_l=z$ such that for all $j \in \intvald 0{l-1}$, $y_{j+1}$ is the image of $y_j$ by an elementary move.
		\end{itemize}
	\end{defi}
	
	Our theorems are consequences of the following one
	\begin{theo}
	\label{the:identifying two points}
		Let $y$ be a point of $X$.
		An element $g \in G$ belongs to $G^n$ if and only if there exist two sequences of elementary moves which respectively send $y$ and $gy$ to the same point.
	\end{theo}
	
	\rem Assume that there are two sequences of elementary moves which respectively send $y$ and $gy$ to the same point.
	By definition this common point can be written $uy = vgy$ where $u$ and $v$ belong to $G^n$. 
	Since $G$ acts freely on $X$ it directly follows that $g$ belongs to $G^n$.
	What we need to prove is the other direction.
	To that end we first show the following induction proposition.

	\begin{prop}
	\label{big induction proposition}
		Let $k \in \N$.
		\begin{enumerate}
		\renewcommand{\labelenumi}{(\Alph{enumi})}
			\item 
			Let $y \in X$. 
			There exists $u \in K_k$ such that $x_0$ and $uy$ are close of rank $k$ and $uy$ is the image of $y$ by a sequence of elementary moves.
			\item 
			Let $y,z \in X$ such that $x_0$ and $y$ (\resp $x_0$ and $z$) are close of rank $k$.
			If $y=z$ in $X_k$, then $z$ is the image of $y$ by a sequence of elementary moves.
			\item 
			Let $y \in X$ such that $x_0$ and $y$ are close of rank $k$.
			Let $g$ be an element of $G$ which is reduced of rank $k$.
			We assume that there exists an integer $m \geq n/2-\xi$ such that 
			\begin{displaymath}
				\diaminter{\geo {x_0}y}{Y_g} \geq \len {g^m} + \sympy{F_BIinductionPropOverlap()} \text{ in }X_k.
			\end{displaymath}
			Then there exist $u,v \in K_k$ such that $uy$ is the image of $y$ by a sequence of elementary moves and 
			\begin{displaymath}
				\diaminter{\geo {x_0}{uy}}{vY_g}\geq \len {g^m} + \sympy{F_BIinductionPropOverlap()} \text{ in } X.	
			\end{displaymath}
		\end{enumerate}
	\end{prop}
	
	\begin{proof}
		The rest of this section is devoted to the proof of this proposition.
		The proof is by induction of $k$.
		If $k =0$, all the conclusions are already contained in the hypothesis (take $u=v=1$). 
		Hence the proposition is true for $k=0$.
		Assume now that the proposition holds for $k \in \N$.
		
		\begin{lemm}
		\label{BI: reduire distance}
			Let $y \in X$ such that $x_0$ and $y$ are close of rank $k$ but not close of rank $k+1$.
			There exists $u \in K_{k+1}$ such that
			\begin{enumerate}
				\item $x_0$ and $uy$ are close of rank $k$,
				\item $uy$ is the image of $y$ by a sequence of elementary moves,
				\item $\dist[X_k]{x_0}{uy} < \dist[X_k]{x_0}y  \sympy{-1*F_BIreduireDist()}$.
			\end{enumerate}
		\end{lemm}
		
		\begin{proof}
			By assumption, there exists $r \in R_k^0$ such that
			\begin{displaymath}
				\diaminter{\geo {x_0}y}{Y_r} > \frac 12\len{r^n} + \sympy{Faux1_BIreduireDist()} \text{ in } X_k.
			\end{displaymath}
			Applying Lemma~\ref{shortening geodesic}, there exists $\kappa \in \Z$ such that $\dist[X_k]{x_0}{r^{\kappa n}y} < \dist[X_k]{x_0}y \sympy{-1* Faux2_BIreduireDist()}$.
			However $r$ is hyperbolic in $G_k$.
			By Proposition~\ref{lifting reduced elements}, there exists $s \in G$ which is reduced of rank $k$ such that $s=r$ in $G_k$.
			In particular $s^n$ belongs to $K_{k+1}$ and $\diaminter{\geo {x_0}y}{Y_s} > \len{s^n}/2 + \sympy{Faux3_BIreduireDist()}$ in $X_k$.
			We put $m = \left\lfloor n/2-\xi \right\rfloor +1$.
			Recall that $(\xi -1)\rinj {G_k}{X_k} \geq  \sympy{Faux4_BIreduireDist()}$.
			It follows that 
			\begin{displaymath}
				\len[espace=X_k] {s^n} 
				\geq 2\len[stable, espace={X_k}]{s^m} + 2(\xi-1)\len[stable, espace={X_k}] s 
				\geq 2\len[espace=X_k]{s^m} + \sympy{Faux5_BIreduireDist()}.
			\end{displaymath}
			Consequently we have in $X_k$, $\diaminter{\geo {x_0}y}{Y_s} > \len{s^m} +  \sympy{Faux6_BIreduireDist()}$, with $m \geq n/2-\xi$.
			By construction $x_0$ and $y$ are close of rank $k$ and $s$ is reduced of rank $k$.
			Applying the induction hypothesis (Prop.~\ref{big induction proposition}(C)), there exist $u,v \in K_k$ such that $uy$ is the image of $y$ by a sequence of elementary moves and 
			\begin{displaymath}
				\diaminter{\geo {x_0}{uy}}{vY_s} \geq \len {s^m} +  \sympy{Faux7_BIreduireDist()} \geq \len{vs^mv^{-1}} \text{ in } X.	
			\end{displaymath}
			Therefore $(vs^{\kappa n}v^{-1})uy$ is the image of $uy$ by an elementary move.
			However, by induction hypothesis (Prop.~\ref{big induction proposition}(A)), there exists $w \in K_k$ such that $x_0$ and $w(vs^{\kappa n}v^{-1})uy$  are close of rank $k$ and $w(vs^{\kappa n}v^{-1})uy$ is the image of $(vs^{\kappa n}v^{-1})uy$ by a sequence of elementary moves.
			
			\paragraph{} Let us now summarize. 
			Using a finite number of elementary moves, we have done the following transformations: 
			\begin{displaymath}
				y \rightarrow uy \rightarrow (vs^{\kappa n}v^{-1})uy \rightarrow w(vs^{\kappa n}v^{-1})uy.
			\end{displaymath}
			On the other hand $u,v,w \in K_k$ and $s^n \in K_{k+1}$. 
			Thus $w(vs^{\kappa n}v^{-1})u$ belongs to $K_{k+1}$ and $w(vs^{\kappa n}v^{-1})u = s^{\kappa n}=r^{\kappa n}$ in $G_k$.
			Hence
			\begin{displaymath}
				\dist[X_k]{x_0}{w(vs^{\kappa n}v^{-1})uy} = \dist[X_k]{x_0}{r^{\kappa n}y} < \dist[X_k]{x_0}y  \sympy{-1*Faux8_BIreduireDist()}.
			\end{displaymath}
			which concludes the proof of the lemma.
		\end{proof}
		
		\begin{lemm}
		\label{BI: distance minimale implique proche}
			Let $y \in X$. 
			There exists $u \in K_{k+1}$ such that $x_0$ and $uy$ are close of rank $k+1$ and $uy$ is the image of $y$ by a sequence of elementary moves.
		\end{lemm}
		\rem This lemma proves Prop.~\ref{big induction proposition}(A) for $k+1$.
		
		\begin{proof}
			Let $\mathcal U$ be the set of elements of $u \in K_{k+1}$ such that $x_0$ and $uy$ are close of rank $k$ and $uy$ is the image of $y$ by a sequence of elementary moves.
			According to the induction hypothesis (Prop.~\ref{big induction proposition}(A)), $\mathcal U$ is non-empty (more precisely $\mathcal U\cap K_k \neq \emptyset$).
			Hence we can choose $u \in \mathcal U$ such that for all $u' \in \mathcal U$, $\dist[X_k]{x_0}{uy} \leq \dist[X_k]{x_0}{u'y} + \delta$.
			We claim that $x_0$ and $uy$ are close of rank $k+1$.
			On the contrary, suppose that this assertion is false.
			By construction of $\mathcal U$, $x_0$ and $uy$ are close of rank $k$.
			By Lemma~\ref{BI: reduire distance}, there exists $v$ in $K_{k+1}$ such that $vu$ belongs to $\mathcal U$ and $\dist[X_k]{x_0}{vuy} < \dist[X_k]{x_0}{uy} \sympy{-1*Faux1_BIminDistGivesClose()}$, which contradicts the definition of $u$.
		\end{proof}
		
		\begin{lemm}
		\label{BI: folding geodesic}
			Let $y \in X$ such that $x_0$ and $y$ are close of rank $k$.
			Let $p,q \in X_k$ such that the configuration $\left(x_0,p,q,y\right)$ is foldable in $X_k$.
			We assume that $p$ and $q$ are equal in $X_{k+1}$ but not in $X_k$.
			There exists $u \in K_{k+1}$ such that
			\begin{enumerate}
				\item $x_0$ and $uy$ are close of rank $k$,
				\item $uy$ is the image of $y$ by a sequence of elementary moves,
				\item $\gro {up}q{x_0} \geq \gro pq{x_0} + \sympy{F_BIfoldingGeoIncGro()}$ in $X_k$,
				\item the configuration $(x_0,up,q,uy)$ is foldable and 
					\begin{displaymath}
						\gro {x_0}{uy}{up} \leq \sympy{F_BIfoldingGeoBoundedGro()}.
					\end{displaymath}
			\end{enumerate}
		\end{lemm}
		
		\begin{proof}
			Let us apply Proposition~\ref{folding proposition} in $X_k$ with $(x_0,p,q,y)$.
			There exist $r \in R_k^0$ and $\epsilon \in \{\pm 1\}$ satisfying the followings.
			\begin{itemize}
				\item $\diaminter{\geo{x_0}y}{Y_r} \geq \len{r^n}/2 \sympy{-1*Faux1_BIfoldingGeo()}$.
				\item $\gro{r^{\epsilon n}p}q{x_0} \geq \gro pq{x_0} + \len{r^n}/2 \sympy{-1*Faux2_BIfoldingGeo()}$.
				\item The configuration $\left(x_0, r^{\epsilon n}p,q,r^{\epsilon n}y\right)$ is foldable.
				Furthermore 
				\begin{displaymath}
					\gro{x_0}{r^{\epsilon n}y}{r^{\epsilon n}p} \leq \sympy{Faux3_BIfoldingGeo()}.
				\end{displaymath}
			\end{itemize}
			However $r$ is hyperbolic in $G_k$.
			By Proposition~\ref{lifting reduced elements}, there exists $s \in G$ which is reduced of rank $k$ such that $s=r$ in $G_k$.
			In particular $s^n$ belongs to $K_{k+1}$.
			Moreover, we have $\diaminter{\geo {x_0}y}{Y_s} \geq \len{s^n}/2 \sympy{-1*Faux4_BIfoldingGeo()}$ in $X_k$.
			We put $m = \left\lfloor n/2-\xi \right\rfloor +1$.
			Just as in Lemma~\ref{BI: reduire distance}, we have $\len[espace=X_k] {s^n} \geq 2\len[espace=X_k]{s^m} + \sympy{Faux5_BIfoldingGeo()}$.
			Consequently we get in $X_k$, $\diaminter{\geo {x_0}y}{Y_s} > \len{s^m}  + \sympy{Faux6_BIfoldingGeo()}$, with $m \geq n/2-\xi$.
			By construction $x_0$ and $y$ are close of rank $k$ and $s$ is reduced of rank $k$.
			Applying the induction hypothesis (Prop.~\ref{big induction proposition}(C)), there exist $u,v \in K_k$ such that $uy$ is the image of $y$ by a sequence of elementary moves and 
			\begin{displaymath}
				\diaminter{\geo {x_0}{uy}}{vY_s} \geq \len {s^m} + \sympy{Faux7_BIfoldingGeo()}  \geq  \len{vs^mv^{-1}} \text{ in } X.	
			\end{displaymath}
			Therefore $(vs^{\epsilon n}v^{-1})uy$ is the image of $uy$ by an elementary move.
			By induction hypothesis (Prop.~\ref{big induction proposition}(A)), there exists $w \in K_k$ such that $x_0$ and $w(vs^{\epsilon n}v^{-1})uy$  are close of rank $k$ and $w(vs^{\epsilon n}v^{-1})uy$ is the image of $(vs^{\epsilon n}v^{-1})uy$ by a sequence of elementary moves.
						
			\paragraph{} Let us now summarize. 
			Using a finite number of elementary moves, we have done the following transformations: 
			\begin{displaymath}
				y \rightarrow uy \rightarrow (vs^{\epsilon n}v^{-1})uy \rightarrow w(vs^{\epsilon n}v^{-1})uy.
			\end{displaymath}
			On the other hand $u,v,w \in K_k$ and $s^n \in K_{k+1}$. 
			Thus $w(vs^{\epsilon n}v^{-1})u$ belongs to $K_{k+1}$ and $w(vs^{\epsilon n}v^{-1})u = s^{\epsilon n}=r^{\epsilon n}$ in $G_k$.
			Consequently the configuration $(x_0, w(vs^{\epsilon n}v^{-1})up,q,w(vs^{\epsilon n}v^{-1})uy)$ is foldable (in $X_k$) and 
			\begin{displaymath}
				\gro {x_0}{w(vs^{\epsilon n}v^{-1})uy}{w(vs^{\epsilon n}v^{-1})up}\leq \sympy{Faux8_BIfoldingGeo()}.
			\end{displaymath}
		\end{proof}
		
		\begin{lemm}
		\label{BI: produit Gromov maximal implique egalite des points}
			Let $y \in X$ such that $x_0$ and $y$ are close of rank $k$.
			Let $p,q \in X_k$ such that the configuration $\left(x_0,p,q,y\right)$ is foldable in $X_k$ and $\gro {x_0}yp \leq \sympy{F_BIgroovMaxGivesSamePointHyp()}$.
			We assume that $p$ and $q$ are equal in $X_{k+1}$.
			There exists $u \in K_{k+1}$ such that
			\begin{enumerate}
				\item $x_0$ and $uy$ are close of rank $k$,
				\item $uy$ is the image of $y$ by a sequence of elementary moves,
				\item in $X_k$, $up=q$ and $\gro{x_0}{uy}q \leq \sympy{F_BIgroovMaxGivesSamePointCl()}$.
			\end{enumerate}
		\end{lemm}
		
		\begin{proof}
			Let us denote by $\mathcal U$ the set of elements $u \in K_{k+1}$ such that,
			\begin{itemize}
				\item $x_0$ and $uy$ are close of rank $k$,
				\item $uy$ is the image of $y$ by a sequence of elementary moves,
				\item in $X_k$, the configuration $(x_0,up,q,uy)$ is foldable, furthermore 
				\begin{displaymath}
					\gro{x_0}{uy}{up} \leq \sympy{Faux1_BIgroovMaxGivesSamePoint()}.	
				\end{displaymath}
			\end{itemize}
			The set $\mathcal U$ is non empty ($1 \in \mathcal U$).
			On the other hand, for all $u \in \mathcal U$, $\gro {up}q{x_0}$ is bounded above by $\dist[X_k] q{x_0}$ in $X_k$.
			Hence we can choose $u \in \mathcal U$ such that for all $u' \in \mathcal U$, $\gro {up}q{x_0} \geq \gro{u'p}q{x_0} \sympy{-1*Faux2_BIgroovMaxGivesSamePoint()}$ in $X_k$.
			We claim that $up=q$.
			On the contrary, suppose that this assertion is false.
			By definition of $\mathcal U$, the configuration $\left(x_0,up,q,uy\right)$ is foldable in $X_k$.
			Therefore applying Lemma~\ref{BI: folding geodesic}, there exists $v \in K_{k+1}$ such that $vu$ belongs to $\mathcal U$ and $\gro{vup}q{x_0} \geq \gro{up}q{x_0} + \sympy{Faux3_BIgroovMaxGivesSamePoint()}$ in $X_k$, which contradicts the definition of $u$.
			Consequently $up=q$ in $X_k$.
			It follows from the definition of $\mathcal U$ that $\gro{x_0}{uy}q \leq \sympy{Faux4_BIgroovMaxGivesSamePoint()}$ in $X_k$.
		\end{proof}
		
		\begin{lemm}
		\label{BI: part B}
			Let $y, z \in X$ such that $x_0$ and $y$ (\resp $x_0$ and $z$) are close of rank $k+1$.
			If $y=z$ in $X_{k+1}$ then $z$ is the image of $y$ by a sequence of elementary moves.
		\end{lemm}
		\rem This lemma proves Prop.~\ref{big induction proposition}(B) for $k+1$.
		
		\begin{proof}
			By assumption $x_0$ and $y$ are close of rank $k$. 
			Moreover $x_0$ and $y$ (\resp $x_0$ and $z$) are $P_k$-close in $X_k$.
			Thus the configuration $\left(x_0,y,z,y\right)$ is foldable in $X_k$ (take $s=t=x_0$ in Definition~\ref{defi:foldable configuration}) and $\gro {x_0}yy = 0$.
			Applying Lemma~\ref{BI: produit Gromov maximal implique egalite des points}, there exists $u \in K_{k+1}$ such that $uy$ is the image of $y$ by a sequence of elementary moves, $uy=z$ in $X_k$ and $x_0$ and $uy$ are close of rank $k$.
			By assumption, $x_0$ and $z$ are also close of rank $k$.
			According to the induction hypothesis (Prop.~\ref{big induction proposition}(B)), $z$ is the image of $uy$ by a sequence of elementary moves.
			Hence $z$ is the image of $y$ by a sequence of elementary moves.
		\end{proof}
		
		\begin{lemm}
		\label{BI: part C}
			Let $y \in X$ such that $x_0$ and $y$ are close of rank $k+1$.
			Let $g \in G$ which is reduced of rank $k+1$.
			We assume that there exists an integer $m \geq n/2-\xi$ such that
			\begin{displaymath}
				\diaminter{\geo {x_0}y}{Y_g} \geq \len {g^m} +  \sympy{F_BIpartCHyp()}, \text{ in } X_{k+1}.	
			\end{displaymath}
			Then there exist $u,v \in K_{k+1}$ such that $uy$ is the image of $y$ by a sequence of elementary moves and 
			\begin{displaymath}
				\diaminter{\geo {x_0}{uy}}{vY_g} \geq \len {g^m} + \sympy{F_BIpartCCl()} \text{ in } X.	
			\end{displaymath}
		\end{lemm}
		\rem This lemma proves Prop.~\ref{big induction proposition}(C) for $k+1$.
		
		\begin{proof}
			Exceptionally we begin the proof by working in $\bar X_k = \lambda^{-1} X_{k+1}$ (instead of $X_{k+1}$).
			Written in $\bar X_k$, our assumption says that $\diaminter{\geo {x_0}y}{Y_g} \geq \len {g^m} +  \sympy{Faux1_BIpartC()}$.
			According to Proposition~\ref{lift overlap}, there exist $r,p,q \in X_k$ and $v \in K_{k+1}$ satisfying the following
			\begin{enumerate}
				\item $d\left(r, vY_g\right) \leq \sympy{Faux2_BIpartC()}$, $d\left(q,vY_g\right) \leq \sympy{Faux3_BIpartC()}$, $\gro {x_0}yp \leq \sympy{Faux4_BIpartC()}$ and $\gro {x_0}qr \leq \sympy{Faux5_BIpartC()}$ in $X_k$,
				\item $\dist[X_k] rq \geq \len[espace = X_k] {g^m} + \sympy{Faux6_BIpartC()}$.
				\item $\bar p = \bar q$ in $\bar X_k$ and thus in $X_{k+1}$. Moreover the configuration $(x_0,p,q,y)$ is foldable in $X_k$.
			\end{enumerate}
			Applying Lemma~\ref{BI: produit Gromov maximal implique egalite des points}, there exists $u \in K_{k+1}$ such that 
			\begin{itemize}
				\item $x_0$ and $uy$ are close of rank $k$,
				\item $uy$ is the image of $y$ by a sequence of elementary moves,
				\item in $X_k$, $up=q$ and $\gro{x_0}{uy}q \leq \sympy{Faux7_BIpartC()}$. 
			\end{itemize}
			In $X_k$ we have
			\begin{displaymath}
				\diaminter{\geo {x_0}{uy}}{vY_g}
				\geq \diaminter{\geo rq}{vY_g} - \gro {x_0}{uy}q - \gro {x_0}qr.
			\end{displaymath}
			On the other hand $d\left(r, vY_g\right)\leq \sympy{Faux8_BIpartC()}$ and $d\left(q,vY_g\right) \leq \sympy{Faux9_BIpartC()}$, thus
			\begin{displaymath}
				 \diaminter{\geo rq}{vY_g} 
				 \geq \dist rq \sympy{-1*Faux10_BIpartC()}
				 \geq \len[espace = X_k] {g^m} + \sympy{Faux11_BIpartC()}.
			\end{displaymath}
			It follows that in $X_k$, $\diaminter{\geo {x_0}{uy}}{vY_g} \geq \len{g^m} +\sympy{Faux12_BIpartC()}$.	
			(Recall that $\lambda^{-1} \geq 500$.)	
			According to the induction hypothesis (Prop.~\ref{big induction proposition}(C)) there exist $u',v' \in K_k$ such that $u'uy$ is the image of $uy$ by a sequence of elementary moves and $\diaminter{\geo {x_0}{u'uy}}{v'vY_g} \geq \len {g^m} + \sympy{Faux12_BIpartC()}$ in $X$.
			In particular $u'u, v'v \in K_{k+1}$ and $u'uy$ is the image of $y$ by a sequence of elementary moves, which ends the proof of the lemma.
		\end{proof}

		Lemmas~\ref{BI: distance minimale implique proche}, \ref{BI: part B} and \ref{BI: part C} proves that Proposition~\ref{big induction proposition} holds for $k+1$.
	\end{proof}

\begin{proof}[Proof of Theorem~\ref{the:identifying two points}]
	Let $g \in G$ such that its image in $G/G^n$ is trivial.
	By construction the direct limit of the sequence $(G_k)$ is $G/G^n$.
	There exists $k \in \N$ such that $g$ is trivial in $G_k$.
	In particular, $y = gy$ in $X_k$.
	By Proposition~\ref{big induction proposition}(A), there exist $u,v \in K_k$ such that $x_0$ and $uy$ (\resp $x_0$ and $vgy$) are close of rank $k$.
	Moreover $uy$ (\resp $vgy$) is the image of $y$ (\resp $gy$) be a sequence of elementary moves.
	However $u$ and  $v$ belong to $K_k$, thus $uy = vgy$ in $X_k$.
	Applying Proposition~\ref{big induction proposition}(B), $vgy$ is the image of $uy$ by a sequence of elementary moves.
\end{proof}